\documentclass[a4paper,11pt,reqno]{amsart}
\usepackage[left=3cm,right=3cm,top=3cm,bottom=3cm]{geometry}
\usepackage[utf8]{inputenc}
\usepackage[english]{babel}
\usepackage{amsmath}
\usepackage{float}
\usepackage{amssymb}
\usepackage{amsfonts}
\usepackage{mathrsfs}
\usepackage{mathtools}
\usepackage{latexsym}
\usepackage{url}
\usepackage{amsthm}
\usepackage{hyperref}
\usepackage{enumitem}
\usepackage[most]{tcolorbox}
\usepackage{cleveref}

\usepackage{svg}
\usepackage{tikz-cd} 
\newcommand{\bq}{\begin{equation}}
\newcommand{\eq}{\end{equation}}
\newcommand{\bqs}{\begin{equation}\begin{split}}
\newcommand{\eqs}{\end{split}\end{equation}}
\newcommand{\bqn}{\begin{equation*}}
\newcommand{\eqn}{\end{equation*}}
\newcommand{\eps}{\varepsilon}
\newcommand{\M}{\mathcal{M}}
\newcommand{\G}{\mathcal{G}}

\theoremstyle{plain}
\newtheorem{theorem}{Theorem}[section]

\newtheorem{defn}[theorem]{Definition}
\newtheorem{cor}[theorem]{Corollary}
\newtheorem{prop}[theorem]{Proposition}
\newtheorem{lem}[theorem]{Lemma}

\theoremstyle{remark}
\newtheorem{remark}[theorem]{Remark}

\NewEnviron{solidbox}{
    \begin{center}
    \par
    \begin{tikzpicture}
    \node[rectangle,minimum width=0.85\textwidth] (m) {
        \begin{minipage}{0.8\textwidth}\BODY\end{minipage}
    };
    \draw (m.south west) rectangle (m.north east);
    \end{tikzpicture}
    \end{center}
}

\NewEnviron{dashedbox}{
    \begin{center}
    \par
    \begin{tikzpicture}
    \node[rectangle,minimum width=0.85\textwidth] (m) {
        \begin{minipage}{0.8\textwidth}\BODY\end{minipage}
    };
    \draw[dashed] (m.south west) rectangle (m.north east);
    \end{tikzpicture}
    \end{center}
}

\DeclarePairedDelimiterX\braket[2]{\langle}{\rangle}{#1 , #2}

\newcommand{\nnorm}[1]{
    {\left\vert\kern-0.25ex\left\vert\kern-0.25ex\left\vert #1
    \right\vert\kern-0.25ex\right\vert\kern-0.25ex\right\vert}
}

\newcommand{\defgl}{\mathrel{=\!\!\mathop:}}
\newcommand{\defgr}{\mathrel{\mathop:\!\!=}}
\newcommand{\C}{\mathbb{C}}
\newcommand{\R}{\mathbb{R}}

\newcommand{\N}{\mathbb{N}}

\newcommand{\Ob}[1]{\mathrm{Ob}(\mathcal{C})}

\begin{document}
\title[Obstacle scattering]{Resonances and weighted zeta functions\\ for obstacle scattering via smooth models} 
\author{Benjamin Delarue}
\email{bdelarue@math.uni-paderborn.de}
\author{Philipp Schütte}
\email{pschuet2@mail.uni-paderborn.de}
\author{Tobias Weich}
\email{weich@math.uni-paderborn.de}
\address{Institut f\"ur Mathematik, Universit\"at Paderborn, Paderborn, Germany}

\begin{abstract}
We consider a geodesic billiard system consisting of a complete Riemannian manifold and an obstacle submanifold with boundary at which the trajectories of the geodesic flow experience specular reflections. We show that if the geodesic billiard system is hyperbolic on its trapped set and the latter is compact and non-grazing the techniques for open hyperbolic systems developed by Dyatlov and Guillarmou  \cite{Dyatlov.2016a} can be applied to a smooth model for the discontinuous flow defined by the non-grazing billiard trajectories. This allows us to obtain a meromorphic resolvent for the generator of the billiard flow. As an application we prove a meromorphic continuation of weighted zeta functions together with explicit residue formulae. In particular, our results apply to scattering by convex obstacles in the Euclidean plane. 
\end{abstract}

\maketitle



\section{Introduction} \label{intro}

Open hyperbolic flows combine two interesting dynamical phenomena: chaotic behavior on the one hand and escape towards infinity on the other hand. In the mathematical physics literature there are two paradigmatic example classes of such flows: geodesic flows on Schottky surfaces and convex obstacle scattering. 

The geodesic flows on (convex co-compact) Schottky surfaces are mathematically much easier to handle (see e.g.\ \cite{Dal10} for an introduction). They are complete smooth flows on Riemannian locally symmetric spaces whose algebraic structure allows for an application of powerful techniques from harmonic analysis and structure theory. One has e.g.\ meromorphic continuations of zeta functions \cite{Fri86, Gui92}, precise estimates on the counting of periodic trajectories \cite{Gui86,Lal89}, or exact correspondences with  the quantum counterpart of the geodesic flow \cite{GHW18} given by the Laplace-Beltrami operator. 

The example of obstacle scattering, in contrast, has the advantage that it is much less abstract and can be seen as a concrete model of a physical particle in a two-dimensional plane performing specular reflections at a finite number of hard obstacles. It has therefore been intensively studied in the physics literature in the context of classical \cite{GR89cl}, semiclassical \cite{GR89sc}, or quantum-mechanical \cite{GR89qm} dynamical systems and allows for numerical \cite{Cvi89, wirzba1999quantum, lu2003, barkhofen2014resonance, weich2014formation} as well as  physical experiments \cite{pance2000quantum,barkhofen2013experimental,potzuweit2012weyl}. Obstacle scattering has however been also in the focus of mathematical literature, see e.g. \cite{Ikawa.1988,NSZ09}.

In this article we focus on the theory of Ruelle-Pollicott resonances. These resonances were introduced by Ruelle \cite{Rue76} and Pollicott \cite{Pol81} to describe the convergence to equilibrium of hyperbolic flows. In most modern formulations the resonances occur as a discrete spectrum in an anisotropic function space and as poles of a meromorphic resolvent. The existence of such a discrete resonance spectrum has been established in many different settings such as Anosov flows on compact manifolds \cite{BL07, FS11,Dyatlov.2016}, Morse-Smale flows \cite{zbMATH07201731, zbMATH07208782},  geodesic flows on manifolds with cusps \cite{bonthonneau2017ruelle}, basic sets of Axiom-A flows \cite{Dyatlov.2016a}, general Axiom-A flows \cite{meddane2021morse}, higher rank Anosov actions \cite{bonthonneau2020ruelle}, or finite horizon Sinai billiards \cite{zbMATH06830045}. The respective meromorphic resolvents are not only useful to define the Ruelle resonances but also have many additional applications such as meromorphic continuation of zeta functions \cite{zbMATH06203676, Dyatlov.2016} or Poincar\'e series \cite{DR20} and to geometric inverse problems \cite{zbMATH07097501, bonthonneau2019local}. The geodesic flow on Schottky surfaces provides concrete examples of basic sets of an Axiom-A flow and it neatly fits into the setting of open hyperbolic systems treated in \cite{Dyatlov.2016a}. In contrast, for obstacle scattering this is not directly the case: As is typical for billiard flows, there are technical difficulties such as the only piecewise smoothness of the flow and singularities caused by grazing trajectories. The aim of this paper is to establish a meromorphic resolvent for the dynamics of obstacle scattering and provide a rigorous framework to define the Ruelle-Pollicott resonances.

Let us first state a simplified version of our main result: Let $\Omega_i\subset \R^n$ , $i=1,\ldots, N$ and $n\geq 2$, be a finite number of disjoint, compact, connected, strictly convex obstacle sets with smooth boundaries and non-empty interiors. Then the \emph{phase space} of the billiard dynamics is given by the unit tangent bundle $M = S (\R^n\setminus \mathring{\Omega})$, $\Omega\defgr \bigcup_{i = 1}^N \Omega_i$, and the billiard dynamics on the  manifold with boundary $M$ is described by a  flow $\varphi^b: \mathbb{R}\times M\rightarrow M$ whose trajectories coincide with straight motion at unit speed unless they intersect the boundary $\partial M=(S\R^n)|_{\partial \Omega}$, where they undergo an \emph{instantaneous fiber-wise reflection} at the tangent lines $T_x\partial \Omega$, $x\in \partial \Omega$. If a trajectory intersects the unit tangent bundle $S\partial \Omega\subset \partial M$ of the obstacle boundary $\partial \Omega$,  we call it a \emph{grazing} trajectory. 
We demand that the billiard flow $\varphi^b$ satisfy the \\

\begin{minipage}{0.9\textwidth}\textbf{\emph{no-grazing condition}}: Every grazing trajectory undergoes only finitely many boundary reflections either in forward or backward time.\\
\end{minipage} 

A classical geometric condition which is sufficient but in general not necessary for the no-grazing condition to hold is the so-called \emph{no-eclipse condition} \cite{Morita1991}: For any two obstacles $\Omega_i, \Omega_j$ the convex hull of their union does not intersect any distinct third obstacle $\Omega_k$. 
\begin{figure}[h]
\centering
\includegraphics[width=0.49\textwidth, trim={0cm 0cm 0cm 0cm}, clip]{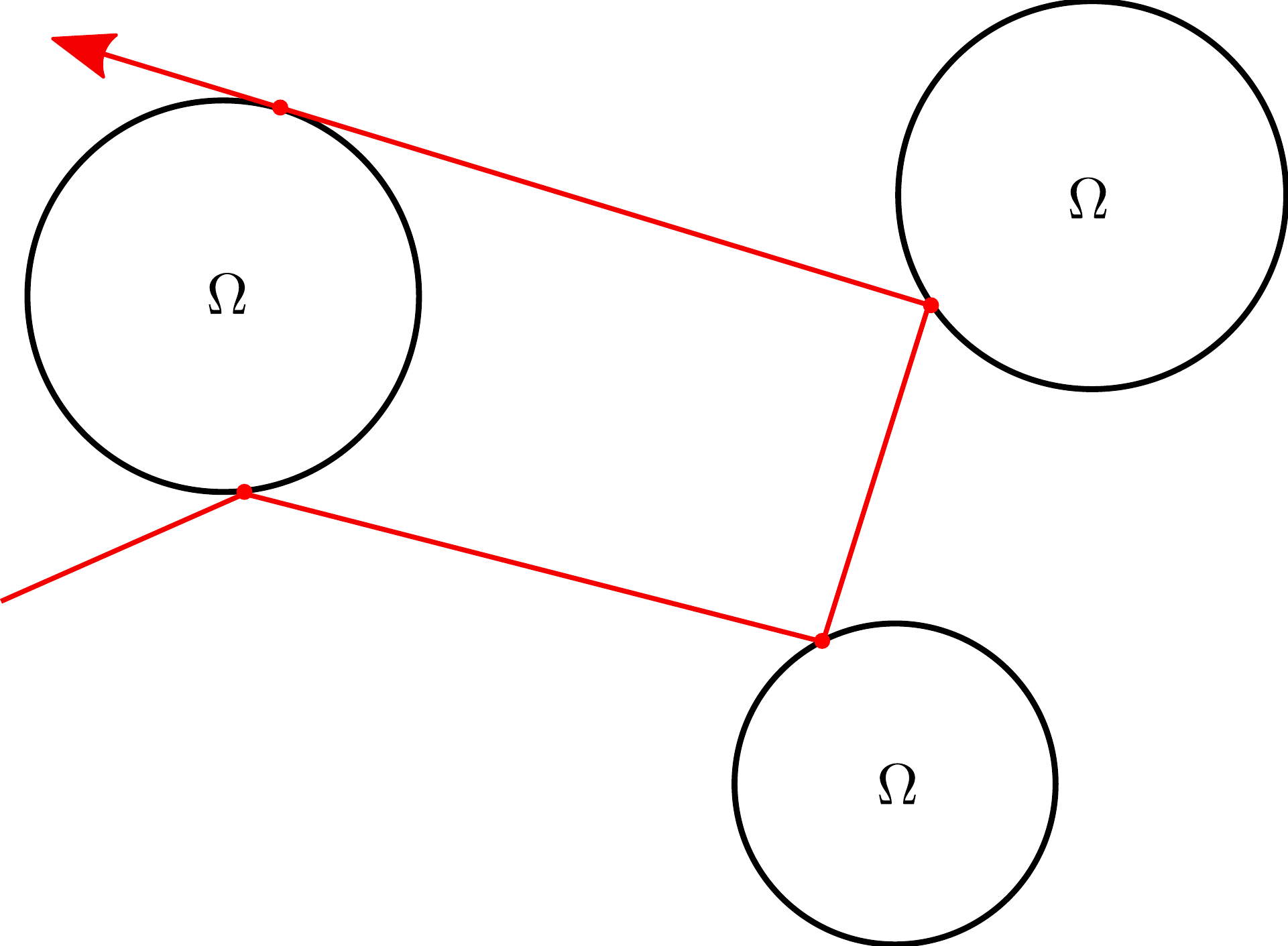}
\caption{Spatial projection of a grazing trajectory in a Euclidean billiard in $\R^2$ with obstacles given by three discs. Note that the above configuration of obstacles satisfies the no-eclipse condition -- in particular, it satisfies the less restrictive no-grazing condition.}
\label{fig05}
\end{figure}

Under the assumption of the no-grazing condition one deduces that the trapped set 
\begin{equation}\label{eq:Kb}
K^b \defgr \{(x, v)\in M \,|\, \varphi^b_t(x, v) \text{ stays in a bounded set for all }t\in \R\}
\end{equation}
contains no grazing trajectories.

In Section~\ref{sec:CBillP} we  introduce for a class of open sets $U\subset M$ spaces of \emph{smooth billiard functions} $\mathrm{C}^\infty_\mathrm{Bill}(U)$ and \emph{billiard test functions} $\mathrm{C}^\infty_\mathrm{Bill, c}(U) = \mathrm{C}^\infty_\mathrm{Bill}(U)\cap \mathrm{C}_\mathrm{c}^\infty(U)$ whose main property to be noted here is that $\mathrm{C}_\mathrm{c}^\infty(U\setminus \partial M) \subset \mathrm{C}^\infty_{\mathrm{Bill, c}}(U)\subset \mathrm{C}_\mathrm{c}^\infty(U)$. We then prove the following theorem:
\begin{theorem}\label{thm:intro}
There is an open neighborhood  $U$ of $K^b$ in $M$ such that with the backward escape time  
$\tau^-_U(x, v) \defgr \sup\{t < 0 \,|\, \varphi^b_t(x, v)\not\in U\}$, $(x, v)\in U$, and for $\lambda\in \C, \mathrm{Re}(\lambda)\gg 0$, 
\[
\mathbf{R}_U(\lambda)f(x, v) \defgr \int_0^{-\tau^-_U(x, v)} \mathrm{e}^{-\lambda t} f\big( \varphi^b_{-t}(x, v) \big) \mathrm{d}t, \qquad f\in\mathrm{C}^\infty_\mathrm{Bill, c}({U}),\, (x, v)\in U,
\]
yields a well-defined family of linear maps $\mathbf{R}_U(\lambda): \mathrm{C}^\infty_\mathrm{Bill, c}(U) \rightarrow \mathrm{C}(U)$ and its matrix coefficients
\[
\langle \mathbf{R}_U(\lambda) f, g\rangle_{\mathrm{L}^2(M)},\qquad f, g\in \mathrm{C}^\infty_\mathrm{Bill, c}({U})
\]
extend from holomorphic functions on $\mathrm{Re}(\lambda) \gg 0$ to meromorphic functions on $\mathbb{C}$ with poles contained in a discrete set of complex numbers that is independent of $f$ and $g$ and the choice of $U$. 
\end{theorem}
We will also see that on $\mathrm{C}^\infty_{\mathrm{Bill}}(U)$ we can define an infinitesimal generator $\mathbf{P}$ of the billiard flow and $\mathbf{R}_U(\lambda)$ can be seen as its resolvent in the sense that $(\mathbf{P} + \lambda) \mathbf{R}_U(\lambda) = \mathrm{id}$ on $\mathrm{Re}(\lambda) \gg 0$. The locations of the possible poles for the matrix coefficients are then defined to be the \emph{Ruelle-Pollicott resonances of the billiard flow}. 

We will actually work without additional further effort in the more general setting of billiards on Riemannian manifolds of arbitrary dimension under the assumption of a compact trapped set for the non-grazing dynamics and uniform hyperbolicity on the trapped set. \Cref{thm:intro} is then deduced as a special case of the more general \Cref{cor:Pmeromorphic}, see \Cref{rem:Thm1cor}. Furthermore, in Appendix~\ref{app:bundles} we explain how the results are transferred to vector bundles.

We obtain the meromorphic resolvent by constructing a smooth model of the billiard flow (Sections~\ref{sec:model} and \ref{construction}) and showing that this flow fits into the framework of \cite{Dyatlov.2016a} (Section~\ref{zeta}). For the latter step we  crucially use results of Conley-Easton \cite{Conley.1971} and Robinson \cite{Robinson.1980} similarly as it was done in \cite{Guillarmou.2017, Dyatlov.2018}. The construction of the smooth model allows to work with the well-known notion of resolvent of a smooth vector field (Theorem \ref{thm:resolvent}) and by \cite{Dyatlov.2016a} we directly get the meromorphic continuation and precise wavefront estimates for this resolvent. Now this allows for a large variety of applications: A first such application is established in this article by the \emph{meromorphic continuation to $\mathbb{C}$} of weighted zeta functions, which in the most simple case of constant weight take the form
\begin{equation*}
\sum_\gamma \frac{T_\gamma^\# \exp\left( -\lambda T_\gamma \right)}{\vert \mathrm{det}\left( \mathrm{id} - \mathcal{P}_\gamma \right)\vert}, \qquad \mathrm{Re}(\lambda) \gg 0,
\end{equation*}
where the sum ranges over all closed trajectories $\gamma$ of the billiard dynamics, $T_\gamma^\#$ denotes the primitive period of $\gamma$ and $\mathcal{P}_\gamma$ its Poincar\'{e} map (for precise definitions and the general theorem we refer to Section~\ref{zeta_continuation}). In the case of 3-disk systems these weighted zeta functions allow the efficient numerical calculation of invariant Ruelle densities (see \cite[Appendix]{Schuette.2021}). In addition, the meromorphic resolvent allows the rigorous study of semiclasscial zeta functions that are of great importance for understanding quantum resonances. Based on the present framework of smooth models Chaubet and Petkov meromorphically continued semiclassical zeta functions and solved the modified Lax-Philipps conjecture. Furthermore in \cite{BSW22} the results regarding meromorphically continued zeta functions were used to provide a rigorous interpretation of semiclassical zeta functions for Wigner distributions as introduced by Ekchardt~et.~al.~\cite{EFMW92}. 
Another application is given by Yann Chaubet's work \cite{Chaubet.2021} on counting asymptotics of periodic trajectories with a fixed reflection number at some fixed obstacle based on previous results in a non-billiard setting \cite{chaubet2021closed}. 

Let us finally mention a recent related result on the definition of Ruelle-Pollicott resonances for billiard flows. In \cite{zbMATH06830045} Baladi, Demers and Liverani perform a meromorphic continuation of the billiard resolvent for the Sinai billiard with finite horizon together with a spectral gap in order to establish exponential mixing. This setting is technically much more challenging than ours because the grazing trajectories cannot be separated from the dynamically relevant region, which is possible in our case by the no-grazing condition.

\subsection*{Acknowledgments}
This project was initiated during P.S.'s stay at MIT in spring 2020 and profited from numerous fruitful discussions with Semyon Dyatlov. We greatly acknowledge his input. We furthermore thank Colin Guillarmou for very helpful discussions concerning the ideas drawn from his work \cite{Guillarmou.2017} and  Yann Chaubet for stimulating discussions and helpful suggestions on an earlier version of the manuscript. This work has received funding from the Deutsche Forschungsgemeinschaft (DFG) (Grant No. WE 6173/1-1 Emmy Noether group ``Microlocal Methods for Hyperbolic Dynamics'' as well as through the Priority Programme (SPP) 2026 ``Geometry at Infinity''). P.S. gratefully acknowledges support from the Studienstiftung des deutschen Volkes.

\section{Geometric setup} \label{geo_setup}

As mentioned before, the main purpose of this work is to handle the case of obstacle scattering in the Euclidean plane. However, we can perform all constructions without considerable further effort in a more general setting by replacing the Euclidean plane by a complete connected Riemannian manifold of arbitrary dimension, see e.g. \cite{burago2002}. In this section we will introduce the notation and define the billiard flow. 

Let $(\Sigma,g)$ be a complete connected Riemannian manifold of dimension $n\geq 1$ without boundary, $S\Sigma \defgr \{(x,v)\in T\Sigma\,|\, g_x(v,v)=1\}\subset T\Sigma$ its unit tangent bundle, and denote by $\mathrm{pr}: S\Sigma\rightarrow \Sigma$, $(x,v)\mapsto x$, the bundle projection as well as by $\varphi^g:\R\times S\Sigma\to S\Sigma$, $(t, x, v)\mapsto \varphi^g_t(x,v)$, the geodesic flow. Let $\Omega\subset \Sigma$ be an $n$-dimensional submanifold with boundary of $\Sigma$. We do not assume $\Omega$ to be connected -- in fact, we regard the connected components of $\Omega$ as \emph{obstacles}. Denoting the manifold boundary and the manifold interior of $\Omega$ by $\partial\Omega$ and $\mathring{\Omega}\defgr \Omega\setminus \partial\Omega$, respectively, we note that $\mathring{\Omega}\neq \emptyset$ and $\partial\Omega\neq \emptyset$ if $\Omega\neq \emptyset$ and define the phase space for our \emph{geodesic billiard} to be
\begin{equation*}
M\defgr S(\Sigma\setminus \mathring{\Omega}) = \mathrm{pr}^{-1}(\Sigma\setminus\mathring{\Omega}) ,
\end{equation*}
the unit tangent bundle over $\Sigma\setminus\mathring{\Omega}$. The space $M$ is a $(2n - 1)$-dimensional submanifold with boundary of $S\Sigma$. Its boundary and manifold interior are given by
\begin{equation}
\partial M = (S\Sigma)|_{\partial\Omega},\qquad \mathring{M} = S(\Sigma\setminus\Omega) .
\end{equation}
Each fiber $S_x \Sigma$, $x\in \partial\Omega$, of the $(n-1)$-sphere bundle $\partial M$  intersects the tangent space $T_x (\partial\Omega)$ in the $(n-2)$-sphere $S_x (\partial\Omega)$. The latter divides $S_x \Sigma$ into two disjoint open hemispheres\footnote{If $n=1$, then $S_x (\partial\Omega)$ is empty; it remains true that the $0$-sphere $S_x\Sigma$ is divided into two disjoint open ``hemispheres'' in this case, namely the two points of which it consists. Eq.\ \eqref{eq:singinout} then holds with $\partial_\mathrm{g} M = \emptyset$.}: One of them is the \emph{inward} hemisphere $(S_x\Sigma)_\mathrm{in}$ containing all vectors ``pointing towards $\Omega$'', i.e., all $v\in S_x\Sigma$ satisfying $g_x(v, n_x) > 0$ with $n_x$ the inward unit normal vector at $x\in\partial\Omega$. The other one is the \emph{outward} hemisphere $(S_x\Sigma)_\mathrm{out}$ defined by $g_x(v, n_x) < 0$. This fiber-wise decomposition effects a decomposition of the $(2n-2)$-dimensional manifold $\partial M$ into a disjoint union
\bq
\partial M  = \partial_\mathrm{g} M \sqcup \partial_\mathrm{in} M \sqcup \partial_\mathrm{out} M\label{eq:singinout}
\eq
of the \emph{grazing boundary} 
\bqn
\partial_\mathrm{g} M \defgr \partial M\cap T (\partial\Omega),
\eqn
a $(2n-3)$-dimensional submanifold of $\partial M$, as well as the \emph{inward} and \emph{outward boundaries} 
\bqn
\partial_\mathrm{in/out} M:= \bigcup_{x\in \partial\Omega}(S_x\Sigma)_\mathrm{in/out},
\eqn 
which are open in $\partial M$. See Figure \ref{fig01} for an illustration of the three boundary components.
\begin{figure}[h]
\centering
\includegraphics[width=0.8\textwidth, trim={0cm 0cm 0cm 0cm}]{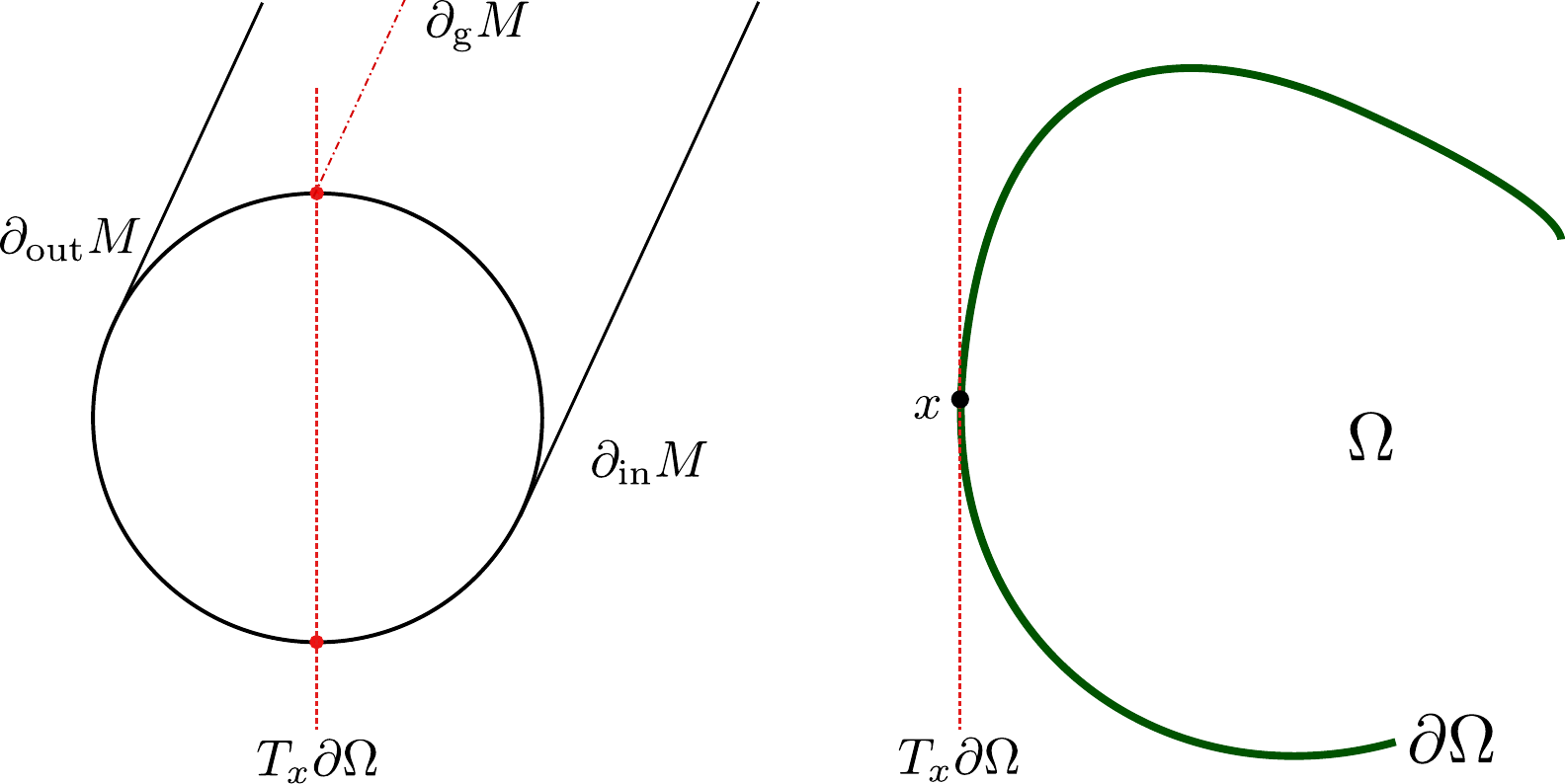}
\caption{The boundary components $\partial_\mathrm{g} M$, $\partial_\mathrm{in} M$, and $\partial_\mathrm{out} M$ in an example where $\Sigma=\R^2$. Notice how $\partial_\mathrm{g} M$ divides each connected component of $\partial M$ into two disjoint connected open subsets. Here the sphere bundle $\partial M=(S\Sigma)|_{\partial \Omega}$, which locally looks like a cylinder, has been cut at the fiber over a point $x\in \partial \Omega$.}
\label{fig01}
\end{figure}

\begin{remark}
The constructions carried out and the theorems proved in the subsequent sections could be easily generalized to a setting where we replace the obstacle boundary $\partial\Omega$ by an arbitrary codimension-1 submanifold $B\subset \Sigma$ at which the specular reflections of $\varphi^g$ occur and which is not necessarily the boundary of an $n$-dimensional submanifold $\Omega$ with boundary. This setting would be more general because not every codimension-1 submanifold occurs as the boundary of a codimension-0 submanifold. However, the example of the circle $\Sigma=S^1$ with $B=\{\mathrm{pt}\}$ a single point, for which our constructions do not work due to the lack of a distinction between ``interior'' and ``exterior'', shows that we would then have to impose some additional assumptions on $B$ or to generalize our methods in a somewhat tedious way, so we stick to the case $B = \partial\Omega$ to simplify the presentation.
\end{remark}

\subsection{Non-grazing billiard dynamics}\label{sec:billiard}

We  focus on a \emph{non-grazing} geodesic billiard dynamics. That is, we define trajectories on $M\setminus\partial_\mathrm{g} M$ by the rule that they are given by those of the geodesic flow $\varphi^g$ until they hit $\partial M$. Then we distinguish two cases: 
\begin{enumerate}
 \item If a trajectory hits $\partial M$ in a non-grazing way, then the velocity vector is reflected at the tangent hyperplane of the obstacle.
 \item If a trajectory hits $\partial M$ in a grazing way, then it ceases to exist. 
\end{enumerate}
\begin{remark}\label{rem:notflow}Before we make the above rules precise, we emphasize that the obtained trajectories will not actually combine to a flow $\varphi$ on $M\setminus\partial_\mathrm{g} M$ in the technical sense because the flow property $\varphi_t(\varphi_{t'}(p))=\varphi_{t+t'}(p)$ is violated for some points $p$ when $t+t'=0$, as explained around \eqref{eq:flowprop}.  For simplicity of the terminology, we shall adopt the convention that we call the map $\varphi$ formed by the billiard trajectories a \emph{flow} regardless of the partial violation of the flow property. We will see that this does not cause any problems. This generalization of terminology will be applied exclusively to $\varphi$; all other maps called \emph{flow} are honest flows in the usual sense.
\end{remark}
For convex obstacles one could of course continue the flow at grazing reflections as in the introduction. However, in the following we would  anyway be forced to remove such grazing trajectories because they lead to a non-smooth dynamics. It is therefore more convenient to start right away with a non-complete flow that stops at grazing reflections. As a side effect we do not need to make any further a priori assumptions on the nature of the boundary such as absence of inflection points (cf. \cite[Assumption A3]{Chernov.2006}). Instead, the necessary assumptions can be efficiently formulated as compactness of the trapped set of the non-grazing billiard flow and hyperbolicity of the flow on that set, see Section~\ref{section_trapped}.

\subsubsection{Definition of the non-grazing billiard flow}

To translate the above rule into a formal definition, we first define a tangential reflection $\partial M\to \partial M$, $(x,v)\mapsto (x,v')$, by letting $v'\in S_x\Sigma$ be the reflection of $v\in S_x\Sigma$ at the hyperplane $T_x (\partial\Omega)\subset T_x\Sigma$. In terms of the inward unit normal $n_x$ this reflection is given by $v' = v - 2g_x(v, n_x) n_x$. The tangential reflection fixes $\partial_\mathrm{g} M$ and interchanges the two boundary components $\partial_\mathrm{in} M$ and $\partial_\mathrm{out} M$. 

For each point $(x,v)\in M\setminus \partial_\mathrm{g} M$ we now define the following extended real numbers which correspond to the first boundary intersection in forward and backward time:
\begin{align}\begin{split}
t_-(x,v) &\defgr \begin{cases}\sup\{t < 0\,|\, \varphi^g_{t}(x,v)\in S\Omega\} , & (x,v)\in \mathring{M}\cup \partial_\mathrm{in} M,\\
\sup\{t < 0\,|\, \varphi^g_{t}(x,v')\in S\Omega\} , & (x,v)\in \partial_\mathrm{out} M
\end{cases}\qquad \in [-\infty,0),\\
t_+(x,v) &\defgr \begin{cases}\inf\{t > 0\,|\, \varphi^g_{t}(x,v)\in S\Omega \}, &(x,v)\in \mathring{M}\cup \partial_\mathrm{out} M,\\
\inf\{t > 0\,|\, \varphi^g_{t}(x,v')\in S\Omega \}, &(x,v)\in \partial_\mathrm{in} M
\end{cases}\qquad \in (0,\infty].\label{eq:tpm}\end{split}
\end{align}
Here the inequalities $t_+(x, v) > 0$ and $t_-(x, v) < 0$ follow in the case $(x,v)\in \mathring{M}$ from the continuity of  $\varphi^g$ and the fact that $\mathring{M}$ is open in $S\Sigma$, and in the case $(x,v)\in \partial M\setminus \partial_\mathrm{g} M$ from the fact that the flow $\varphi^g$ is transversal to $\partial M\setminus \partial_\mathrm{g} M$. 

We begin with the local definition of the billiard trajectories near the boundary: First, let $(x, v) \in \mathring{M}$. Then the billiard flow simply equals the geodesic flow as long as it does not meet an obstacle boundary: 
\begin{equation}
\varphi_t(x, v) \defgr \varphi_t^g(x, v), \qquad  t_-(x, v) < t < t_+(x, v).
\label{eq:localflow1}
\end{equation}
If the geodesic flow meets an obstacle boundary in a non-grazing way, i.e., if $t_\pm(x, v)\in \R$ and $\varphi_{t_\pm(x, v)}^g(x, v)\not\in \partial_\mathrm{g} M$, then we extend the trajectory via the following obvious definition:
\begin{equation} \label{eq:bdrydefmin}
\varphi_{t_\pm(x, v)}(x, v) \defgr\varphi_{t_\pm(x, v)}^g(x, v) .
\end{equation}
Now let $(x, v)\in\partial M\setminus \partial_\mathrm{g} M$. Then we have to distinguish between inward and outward components:
\begin{equation}
\begin{split}
\varphi_t(x, v) \defgr 
\begin{cases}
\varphi_t^g(x, v), &~~(x, v) \in\partial_\mathrm{out} M,~ 0\leq t< t_+(x, v)\\
&\text{or}~ (x, v) \in\partial_\mathrm{in} M,~ t_-(x, v)< t\leq 0,\medskip\\
\varphi_t^g(x, v'), &~~(x, v) \in\partial_\mathrm{in} M,~ 0 < t< t_+(x, v')\\
&\text{or}~ (x, v) \in\partial_\mathrm{out} M,~ t_-(x, v')< t < 0 .
\end{cases}
\end{split}
\label{eq:localflow2}
\end{equation}
The fact that $\varphi^g$ is a flow implies that the trajectories of $\varphi$ also have the flow property except for boundary points where the flow property only holds modulo reflection. More precisely, whenever $\varphi_{t'}(\varphi_{t}(x, v))$ and $\varphi_{t+t'}(x, v)$ are defined by \eqref{eq:localflow1}, \eqref{eq:bdrydefmin}, or \eqref{eq:localflow2}, one has
\begin{equation} \label{eq:flowprop}
\varphi_{t'}(\varphi_{t}(x, v))=\begin{dcases}\varphi_{t+t'}(x, v),& t + t' \neq 0\text { or }\varphi_{t+t'}(x,v)\in \mathring M\text{ or } \\
&t<0,(x, v) \in\partial_\mathrm{in} M\text { or } t>0,(x, v) \in\partial_\mathrm{out} M,\\
(x,v'), & t+t'=0\text { and }t>0,(x, v) \in\partial_\mathrm{in} M \text { or }\\
& t+t'=0\text { and } t<0,(x, v) \in\partial_\mathrm{out} M,
\end{dcases}
\end{equation}
while $\varphi_0(x,v)=(x,v)$ for all $(x,v)\in M\setminus \partial_\mathrm{g} M$. 
The additional boundary reflection does not pose any problems for the remaining paper, though, as points on the boundary related by reflection will be indistinguishable in our smooth models anyway.

Finally we extend the trajectories to their maximal lengths, which is formally somewhat cumbersome: We define for $(x, v)\in M\setminus \partial_\mathrm{g} M$  recursively 
\begin{align*}
t^0_\pm(x,v) &\defgr t_\pm(x,v),\\
t^n_\pm(x,v) &\defgr \begin{cases} t^{n-1}_\pm(x,v) + t_\pm\big(\varphi^g_{t^{n-1}_\pm}(x,v)\big), & t^{n-1}_\pm(x,v)\in \R ~\text{and}~ \varphi^g_{t^{n-1}_\pm}(x,v) \not\in \partial_\mathrm{g} M,\\
t^{n-1}_\pm(x,v), &\text{else},
\end{cases}
\end{align*}
where $n\in \mathbb{N}$. Then the sequences $\{t^n_+(x,v)\}_{n\in \N_0}\subset (0,\infty]$ and $\{t^n_-(x,v)\}_{n\in \N_0}\subset[-\infty,0)$ 
are non-decreasing and non-increasing, respectively, and we put
\begin{equation*}
T_\mathrm{max}(x,v) \defgr \limsup_{n\to \infty} t^n_+(x,v)\in (0, \infty], \qquad T_\mathrm{min}(x,v) \defgr \liminf_{n\to \infty} t^n_-(x,v)\in [-\infty, 0) .
\end{equation*}
These are the maximal times for which the non-grazing billiard flow can be defined: Given $t\in (T_\mathrm{min}(x,v), T_\mathrm{max}(x,v))$ we can find some $N\in \N_0$ and real numbers $t_0, \ldots, t_N\in (T_\mathrm{min}(x,v), T_\mathrm{max}(x,v))$ with $\sum_{j=0}^N t_j=t$ and such that every term in the composition
\begin{equation} \label{eq:defflow}
\varphi_{t}(x, v) \defgr \varphi_{t_0} ( \varphi_{t_1}(\cdots(\varphi_{t_N}(x, v))\cdots) )
\end{equation}
is well-defined by either \eqref{eq:localflow1}, \eqref{eq:bdrydefmin}, or \eqref{eq:localflow2}. This definition of $\varphi_t(x, v)$ is then independent of the choice of the numbers $t_0, \ldots, t_N$ by virtue of the flow property \eqref{eq:flowprop}. The definition \eqref{eq:defflow} extends the trajectory through $(x, v)$ such that in summary, using the extended terminology from Remark \ref{rem:notflow}, we obtain a flow 
$$
\varphi:D\to M\setminus \partial_\mathrm{g} M,\qquad (t,x,v)\mapsto \varphi_t(x,v),
$$
on the domain given by
\bq
D \defgr \{(t, x, v)\in \mathbb{R}\times (M\setminus \partial_\mathrm{g} M) \,|\, t\in (T_\mathrm{min}(x,v),T_\mathrm{max}(x,v))\}.\label{eq:D}
\eq
Some basic properties of this domain will be proved below in \Cref{lem:D}. In particular, we shall see that $D$ is open in $\mathbb{R}\times (M\setminus \partial_\mathrm{g} M)$. In the following we will call the triple $(\Sigma, g, M)$ a \emph{geodesic billiard system} and $\varphi$ its associated \emph{non-grazing billiard flow}, keeping Remark \ref{rem:notflow} in mind.

\subsection{Properties of the non-grazing billiard flow}\label{sec:techn} 

By definition, the flow $\varphi$ has discontinuities at $\varphi^{-1}(\partial M\setminus \partial_\mathrm{g} M )$, which makes the following constructions necessary in the first place. A further property of the flow $\varphi$ that can be read off directly from its definition and which will become important below is its invariance under tangential reflections at the boundary:
\bq
\forall\ (t,x,v)\in D\cap((\R\setminus\{0\}) \times \partial M): \quad (t, x, v')\in D,~ \varphi(t, x, v) = \varphi(t, x, v').\label{eq:invariance}
\eq
While $\varphi$ itself is not continuous except in the trivial case $\partial M = \emptyset$, its composition with the projection $\mathrm{pr}: M\rightarrow \Sigma\setminus \mathring{\Omega}$ is continuous and describes the \emph{spatial billiard dynamics}. For graphical illustrations of $\mathrm{pr}\circ\varphi$  and $\varphi$ see Figures \ref{fig02} and \ref{fig03}.

\begin{figure}[H]
\centering
\includegraphics[width=0.7\textwidth, trim={0cm 0cm 0cm 0cm}]{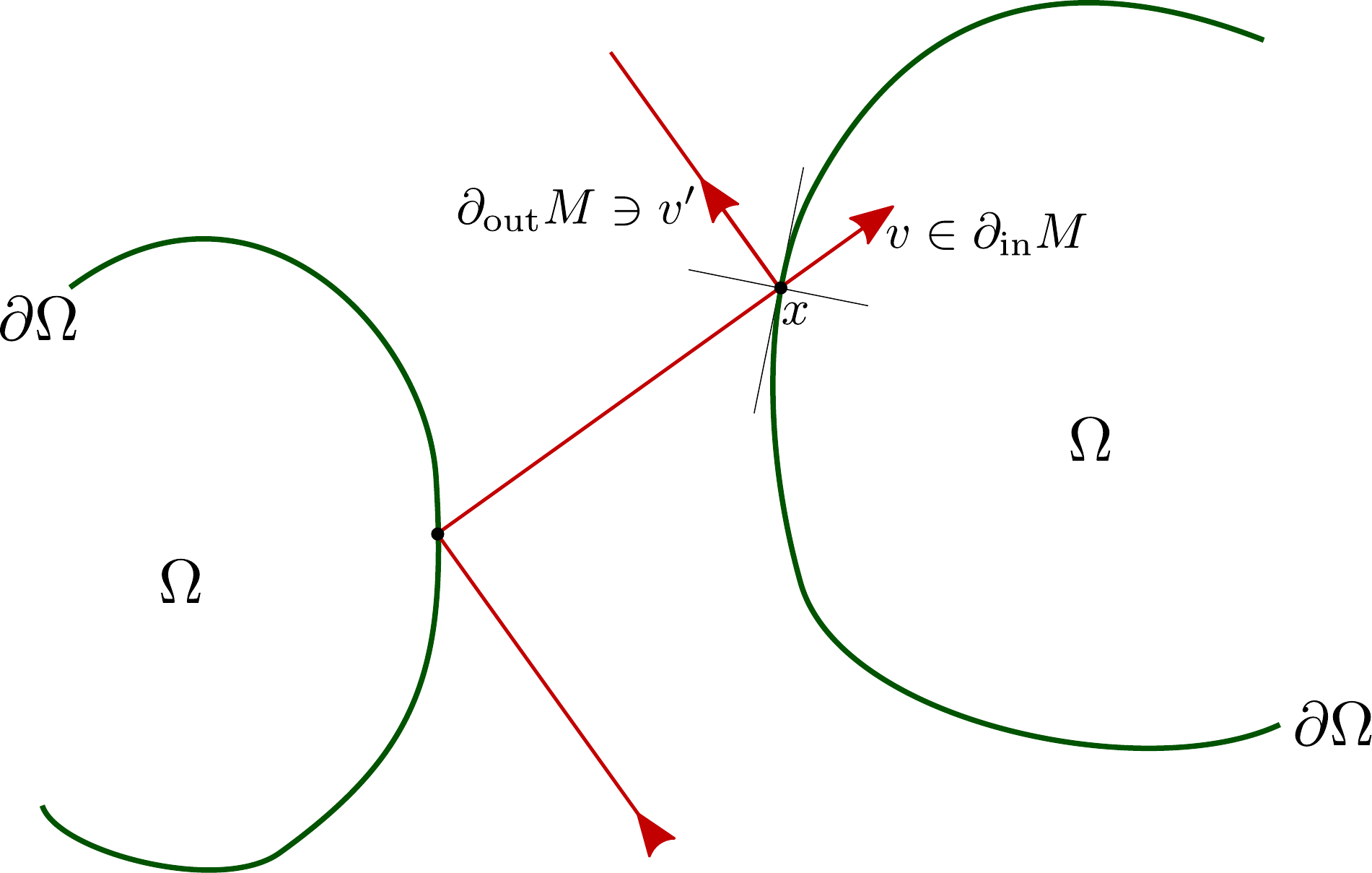}
\caption{A trajectory of the spatial billiard dynamics described by the composition $\mathrm{pr}\circ\varphi:\R\times M\to \Sigma\setminus\mathring \Omega$ in an example where $\Sigma=\R^2$ is the Euclidean plane. The arrows indicate a point $(x,v)\in \partial_\mathrm{in} M$ as well as its tangential reflection $(x,v')\in \partial_\mathrm{out} M$.}
\label{fig02}
\end{figure}

\begin{figure}[h]
\centering
\includegraphics[width=0.8\textwidth, trim={0cm 0cm 0cm 0cm}]{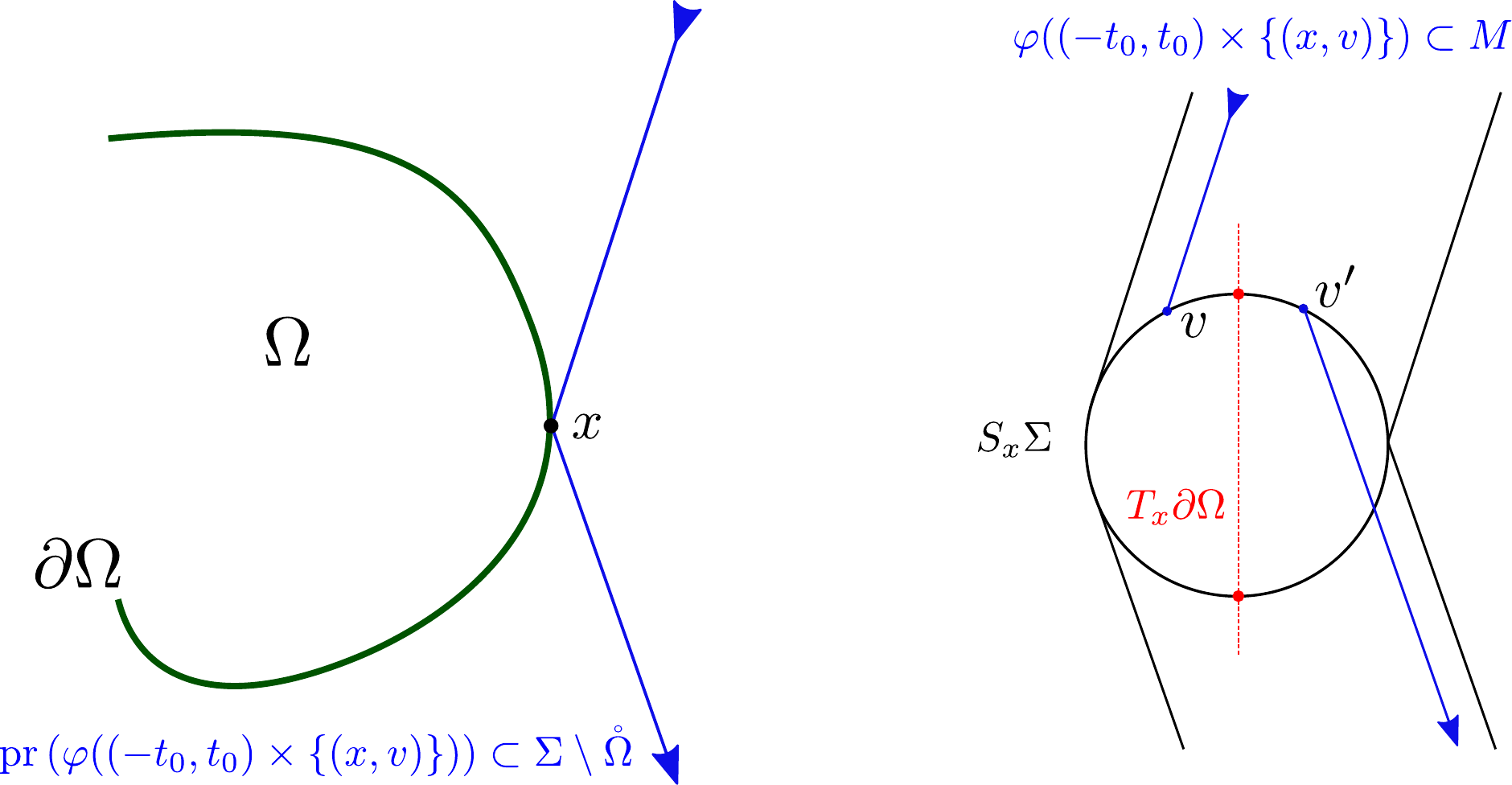}
\caption{The partial flow trajectory $\varphi((-t_0,t_0)\times \{(x,v)\})\subset M$ of a point $(x,v)\in \partial_\mathrm{in} M$ and its spatial projection in an example where $\Sigma=\R^2$ is the Euclidean plane. The right-hand side depicts the sphere bundle over the spatial trajectory, which locally looks like a cylinder.  Notice the discontinuity in the flow trajectory due to the tangential reflection $(x,v)\mapsto (x,v')$. Projecting the trajectory onto $\Sigma$ removes this discontinuity.}
\label{fig03}
\end{figure}

In spite of its discontinuous nature, the non-grazing billiard flow $\varphi$ possesses useful transversality properties at the non-grazing boundary $\partial M\setminus \partial_\mathrm{g} M$ which we collect in the technical Lemma \ref{lem:convexity} below. This lemma will allow us to prove, among other statements, that the domain $D$ is open in $\mathbb{R}\times (M\setminus \partial_\mathrm{g} M)$ and we will later use Lemma \ref{lem:convexity} to consider the flow-time as a ``coordinate'' transverse to $\partial M\setminus\partial_\mathrm{g} M$.  
\begin{lem}\label{lem:convexity}
There exists an open subset $N\subset\R\times (\partial M\setminus \partial_\mathrm{g}M)$ such that
\begin{enumerate}
\item[i)] One has the inclusions $\{0\}\times \left(\partial M\setminus \partial_\mathrm{g}M \right)\subset N \subset D$. 
\item[ii)] The set $N$ is invariant under tangential reflection in the sense that for all $(t,x,v)\in\R\times (\partial M\setminus \partial_\mathrm{g}M)$ one has $(t,x,v)\in N$ iff $(t,x,v')\in N$.
\item[iii)] The restricted map $\varphi|_N: N \to M$ is open and its image is contained in $M\setminus \partial_\mathrm{g}M$.
\item[iv)] Decomposing $N$ into the two disjoint open subsets 
$$
N_\mathrm{in}:=N\cap (\R\times \partial_\mathrm{in} M),\qquad N_\mathrm{out}:=N\cap (\R\times \partial_\mathrm{out} M),
$$
the two  maps $\varphi|_{N_\mathrm{in/out}}: N_\mathrm{in/out} \to M\setminus \partial_\mathrm{g}M$ are injective.
\item[v)] The two inverse maps $\varphi|_{N_\mathrm{in/out}}^{-1}:\varphi(N_\mathrm{in/out})\to N_\mathrm{in/out}$ are smooth.
\item[vi)] Decomposing $N_\mathrm{in}$ and $N_\mathrm{out}$ further into the subsets
  \bq
N^\pm_\mathrm{in/out}:=N_\mathrm{in/out}\cap (\R_\pm\times \partial_\mathrm{in/out} M), \qquad  \R_\pm:=\{t\in \R\,|\,\pm t \geq 0\},\label{eq:Ndecomp}
\eq
one has 
\bq
\varphi_t(x,v)=\begin{cases}\varphi_t^g(x,v), & (t,x,v)\in N^-_\mathrm{in}\cup N^+_\mathrm{out},\\
\varphi_t^g(x,v'), &(t,x,v)\in (N_\mathrm{in}\setminus N^-_\mathrm{in})\cup (N_\mathrm{out}\setminus N^+_\mathrm{out}).
\end{cases}\label{eq:varphivarphig}
\eq
\end{enumerate}
\end{lem}The proof of Lemma \ref{lem:convexity} is given in Appendix \ref{sec:proof1}. See Figure \ref{fig31} for an illustration of $N_\mathrm{in/out}$ and $\varphi(N_\mathrm{in/out})$ in a $2$-dimensional example.
\begin{remark}
Note that by \emph{i)} and \emph{iii)} the images $\varphi(N_\mathrm{in/out})$ are open subsets of $M\setminus \partial_\mathrm{g}M$ which intersect the boundary $\partial M\setminus \partial_\mathrm{g}M$ non-trivially. In particular, the sets $\varphi(N_\mathrm{in/out})$ are themselves manifolds with non-empty boundaries (except in the trivial case $\partial M=\emptyset$). In contrast, the open sets $N_\mathrm{in/out} \subset \R\times \partial_\mathrm{in/out}M$ are manifolds without boundary. 
\end{remark}

\begin{figure}[h]
\centering
\includegraphics[width=0.8\textwidth]{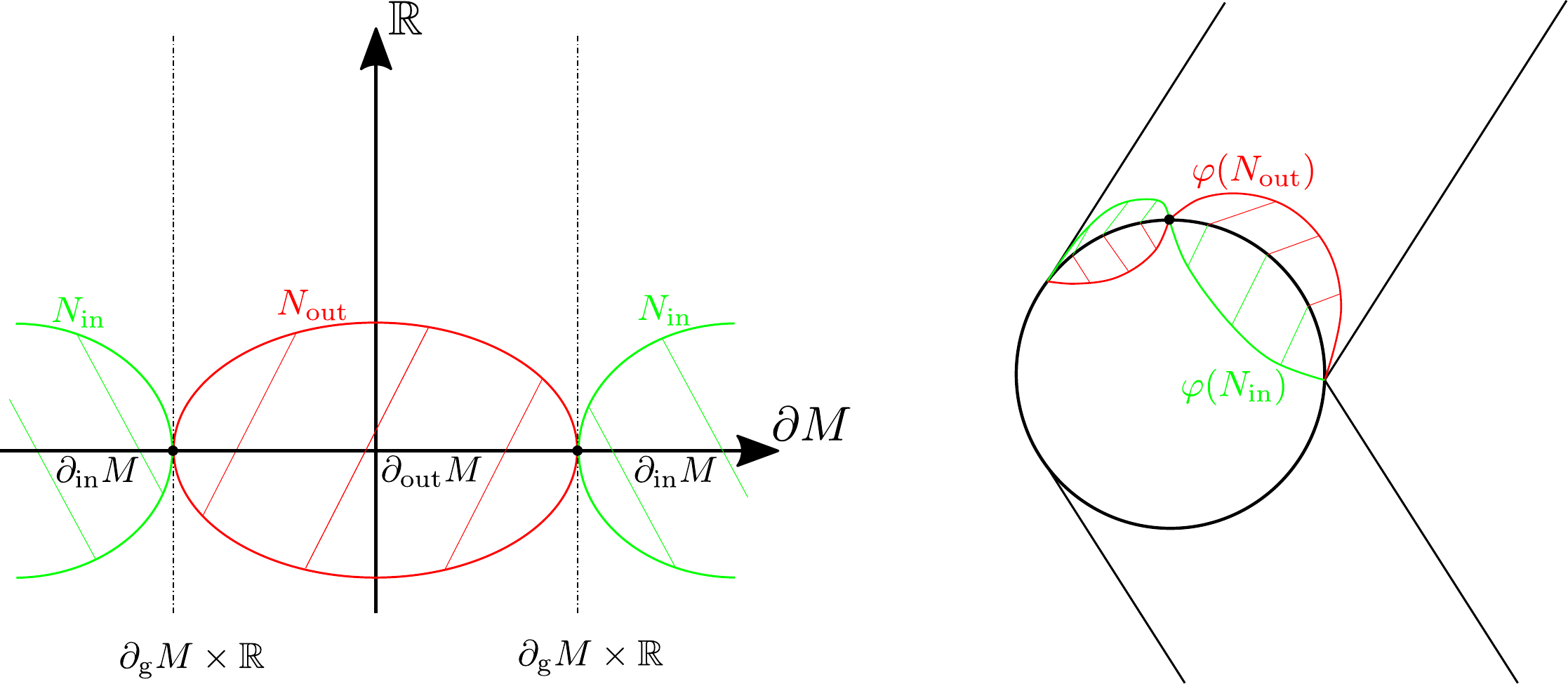}
\caption{Schematic illustration of the sets $N_\mathrm{in/out}$ from Lemma \ref{lem:convexity} and their images under $\varphi$ in a Euclidean setting as in Figure \ref{fig03}. Note that $N_\mathrm{in/out}$ and $\varphi(N_\mathrm{in/out})$ are actually $3$-dimensional. In the image on the left-hand side the dimension has been reduced by drawing the $2$-dimensional manifold $\partial M$ simply as a coordinate axis, whereas on the right-hand side the dimension has been reduced by focusing on the circle $S_x\Sigma\subset \partial M$ over some chosen point $x\in \partial \Omega$. }
\label{fig31}
\end{figure}

The following definition is motivated by the subsequent important lemma.
\begin{defn}\label{def:symm}We call a set $A\subset \partial M\setminus \partial_\mathrm{g} M$  \emph{reflection-symmetric} if for every point $(x,v)\in A$ its tangential reflection $(x,v')$ also belongs to $A$.
\end{defn}
Using this terminology we can conveniently describe a crucial continuity property of the non-grazing billiard flow $\varphi$ which puts the discontinuity of the latter into perspective:
\begin{lem}\label{lem:D}Let $O\subset M\setminus \partial_\mathrm{g} M$ be an open set such that $O\cap (\partial M\setminus \partial_\mathrm{g} M)$ is reflection-symmetric. Then $\varphi^{-1}(O)\subset D$ is open in $\R\times (M\setminus \partial_\mathrm{g} M)$ and 
invariant under tangential reflection in the sense that 
\bq
\forall\;(t,x,v)\in \R\times (\partial M\setminus \partial_\mathrm{g} M): (t,x,v)\in \varphi^{-1}(O) \iff (t,x,v')\in \varphi^{-1}(O).\label{eq:Dinv}
\eq
In particular, the domain $D=\varphi^{-1}(M\setminus \partial_\mathrm{g} M)$ is open in $\R\times (M\setminus \partial_\mathrm{g} M)$ and satisfies \eqref{eq:Dinv}. 
\end{lem}
The proof of Lemma \ref{lem:D} is given in Appendix \ref{sec:proof2}. 

Having shown in \Cref{lem:D} that $D$ is open in $\R\times (\partial M\setminus \partial_\mathrm{g} M)$ and recalling from \eqref{eq:D} that for any $(x,v)\in M\setminus \partial_\mathrm{g} M$ the set $\{t\in\R \,|\, (t, x,v)\in D\}$ is an open interval containing $0$, we see that $D$ is an honest \emph{flow domain}.

Finally, let us mention without detailing the proof  that using similar arguments as in the proof of \Cref{lem:D} one can show that the flow $\varphi$ is smooth on the set $\varphi^{-1}(\mathring M)$. The latter is open in $\R\times (M\setminus \partial_\mathrm{g} M)$ by \Cref{lem:D}.

\subsubsection{Trapped set and hyperbolicity} \label{section_trapped}

We define the \emph{trapped set} of $\varphi$ as those points for which the flow is globally defined and the trajectory remains within a compact region, i.e.:
\begin{equation} \label{eq:K}
\begin{split}
K \defgr \big\{(x, v)\in M\setminus \partial_\mathrm{g} M \,|\, &\R\times\{(x,v)\} \subset D,\\
&\exists ~\text{compact}~ W\subset M\setminus \partial_\mathrm{g} M ~\text{with}~ \varphi(\R\times\{(x, v)\})\subset W \big\}.
\end{split}
\end{equation}
An illustration of $K$ in a $2$-dimensional Euclidean setting can be found in Figure \ref{fig04}.
\begin{figure}[h]
\centering
\includegraphics[width=0.75\textwidth, trim={0cm 0cm 0cm 0cm}]{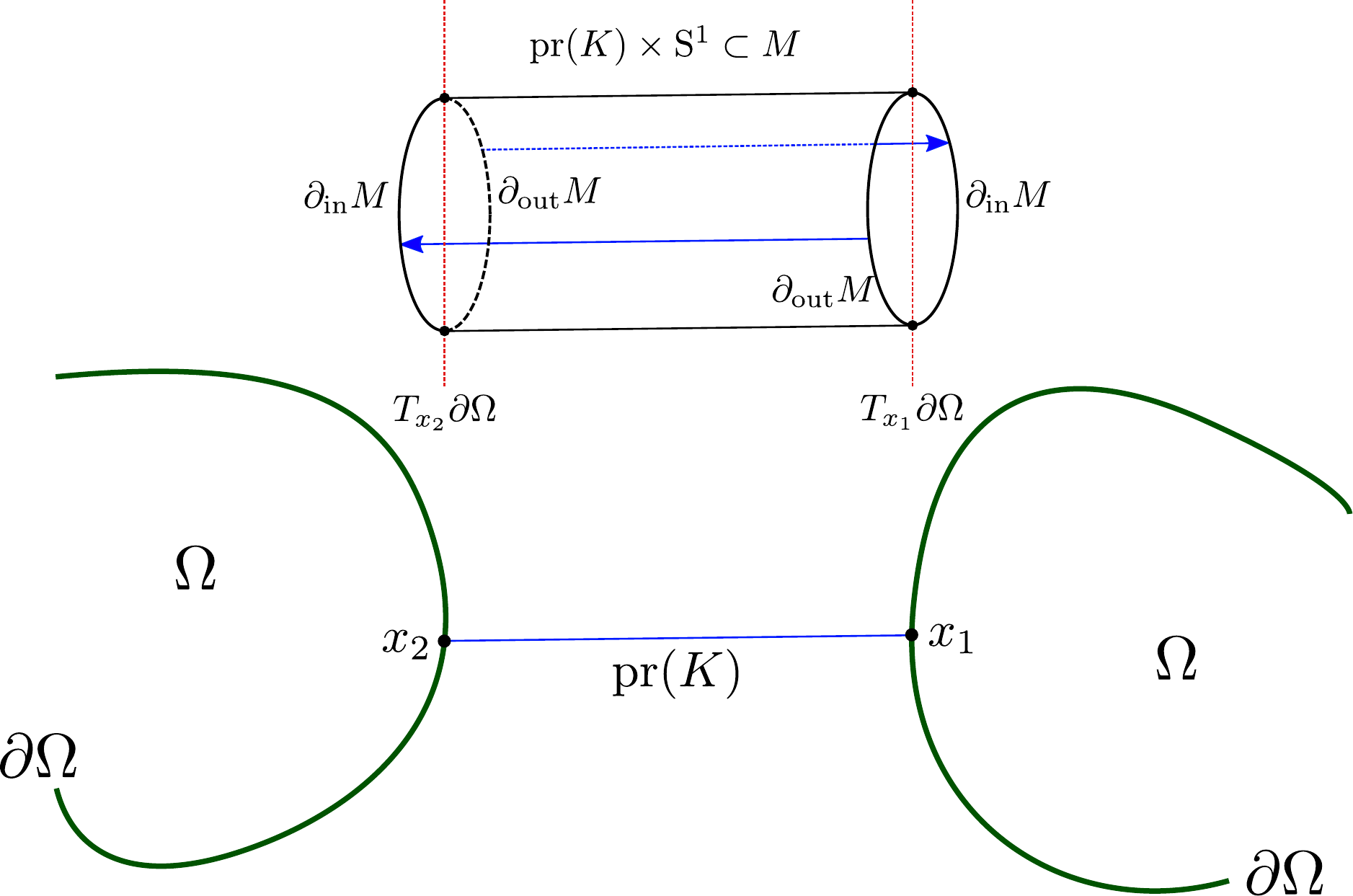}
\caption{Illustration of the trapped set $K$ in an example where $\Sigma=\R^2$ is the Euclidean plane and $\mathrm{pr}(K)$ consists of a single closed ``bouncing'' trajectory of the spatial billiard flow. }
\label{fig04}
\end{figure}
\begin{lem}
$K\cap (\partial M\setminus \partial_\mathrm{g} M)$ is  reflection-symmetric in the sense of \Cref{def:symm}.
\end{lem}
\begin{proof}
This follows immediately from the fact that for every $(x,v)\in K\cap (\partial M\setminus \partial_\mathrm{g} M)$ the trajectory $\varphi(\R\times \{(x,v')\})$ is well-defined and coincides  with $\varphi(\R\times \{(x,v)\})$ except at $t=0$, as follows from \eqref{eq:invariance}. 
\end{proof}
\begin{remark}
For the proof of our main results  (Theorem~\ref{thm:resolvent}, Corollary~\ref{cor:Pmeromorphic}) we will assume that  the trapped set $K$ is compact. Note that this assumption is a non-trivial condition on the global geometry of the obstacles because we work with the non-grazing dynamics. The compactness of the trapped set implies that all trapped trajectories are located at a strictly positive distance from the grazing trajectories. 
\end{remark}
\begin{lem}\label{lem:compact_trapped}
 If in the $n$-dimensional Euclidean convex obstacle scattering setup considered in the introduction the obstacles $\Omega=\cup_{i=1}^N\Omega_i\subset \R^n$ fulfill the no-grazing condition  (thus in particular if they fulfill the no-eclipse condition), then the trapped set $K$ of the non-grazing billiard flow, defined in \eqref{eq:K}, agrees with the trapped set $K^b$ of the complete billiard flow defined in \eqref{eq:Kb}. In particular, $K$ is compact.
\end{lem}
\begin{proof}
As our non-grazing billiard flow $\varphi$ is a restriction of the full billiard flow $\varphi^b$ up to the first grazing collision one clearly has $K\subset K^b$. Now the no-grazing condition implies that none of the trajectories in $K^b$ experiences a grazing collision which lets us infer the reverse inclusion $K^b\subset K$.  Finally, in view of the compactness of $\cup_{i=1}^N\Omega_i$, it is a well-known fact that $K^b$ is compact, see \cite[Section 1.3]{Florio} and the references given therein.
\end{proof}
\begin{remark}
 If the obstacles $\cup_{i=1}^N\Omega_i$ are not assumed to satisfy the no-grazing condition it is easy to construct examples that have a non-compact trapped set $K$ (but nevertheless a compact trapped set $K^b$ for the complete billiard flow $\varphi^b$), see Figure~\ref{fig:exmpl_noncpt}.
\end{remark}

\begin{figure}[h]
\centering
\includegraphics[width=0.5\textwidth, trim={0cm 0cm 0cm 0cm}, clip]{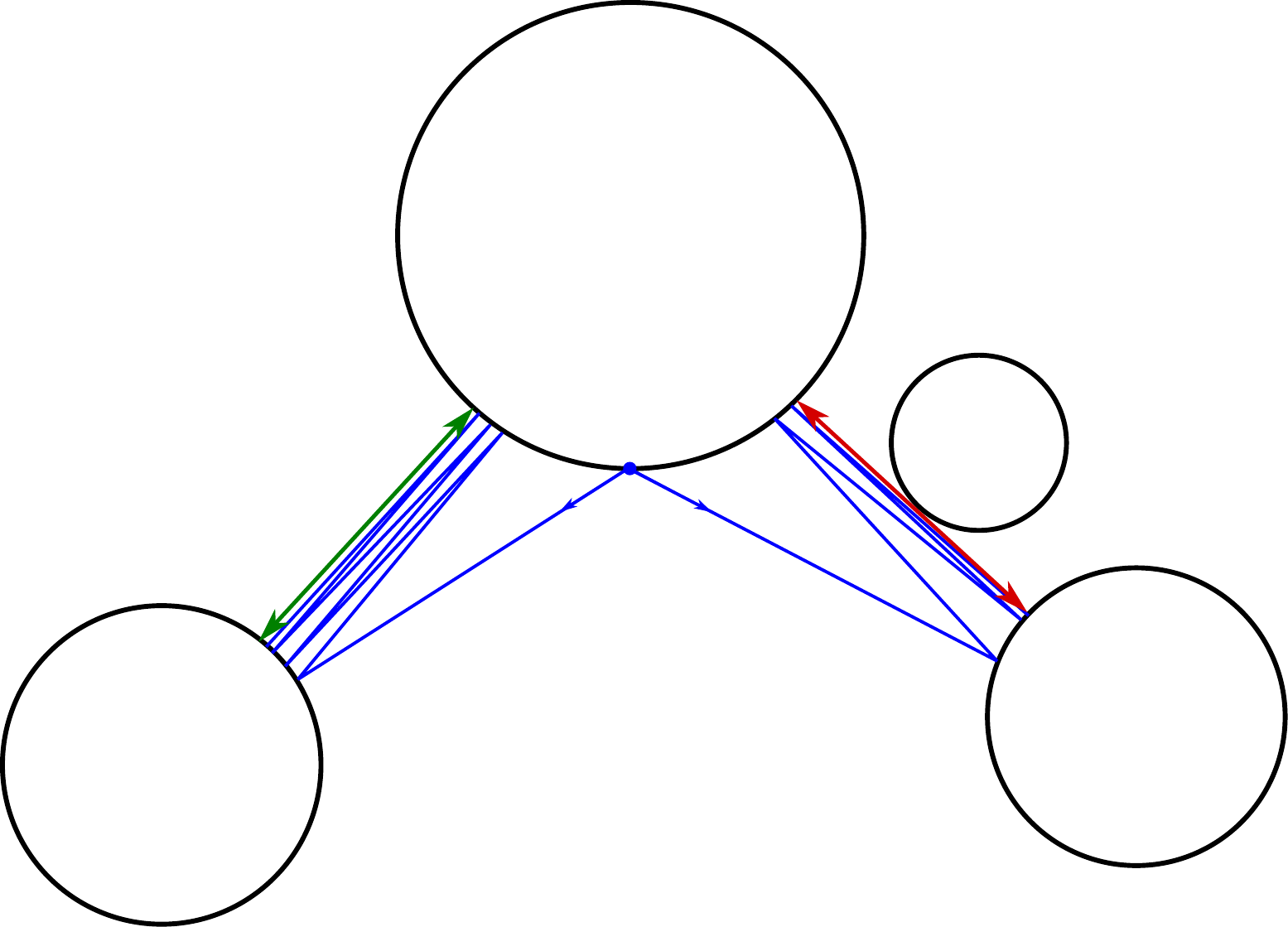}
\caption{The red trajectory belongs to the spatial projection of the trapped set $K^b$ of the complete billiard flow, but not to the spatial projection of $K$ as it contains a grazing collision. The green and blue trajectories, however, belong to the spatial projections of $K^b$ as well as  $K$. Note that the red trajectory clearly lies in the closure of the blue trajectory in $\Sigma\setminus \mathring \Omega$ and the latter touches the former tangentially, thus $K$ is not closed in $M\setminus \partial_\mathrm{g}M$ and hence not compact.}
\label{fig:exmpl_noncpt}
\end{figure}

Finally we introduce a notion of \emph{hyperbolicity} for the billiard dynamics.
\begin{defn} \label{def:hypbilliard}
The non-grazing billiard flow $\varphi$ is called \emph{hyperbolic on its trapped set $K$} if the following holds: For any $(x, v)\in K\cap \mathring{M}$ the tangent bundle exhibits a continuous splitting
\begin{equation} \label{eq:hyperbolic_splitting}
T_{(x, v)} M = \mathbb{R}\cdot X(x, v) \oplus E_s(x, v) \oplus E_u(x, v),
\end{equation}
where $X(x, v)$ denotes the flow direction at $(x, v)$, $E_{s/u}(x, v)$ is mapped onto $E_{s/u}(\varphi_t(x, v))$ under the differential of $\varphi_t$ whenever $\varphi_t(x, v)\in K\cap \mathring{M}$, and there exist constants $C_0, C_1 > 0$ such that
\begin{equation} \label{eq:def_hyperbolicity}
\begin{split}
\Arrowvert \mathrm{d}\varphi_t(x, v) W\Arrowvert_{\varphi_t(x, v)} &\leq C_0 \exp(-C_1 t) \Arrowvert W\Arrowvert_{(x, v)}, \quad t\geq 0,~ \varphi_t(x, v)\in K\cap \mathring{M},~ W\in E_s(x, v) \\
\Arrowvert \mathrm{d}\varphi_t(x, v) W\Arrowvert_{\varphi_t(x, v)} &\geq C_0^{-1} \exp(C_1 t) \Arrowvert W\Arrowvert_{(x, v)}, \quad t\geq 0,~ \varphi_t(x, v)\in K\cap \mathring{M},~ W\in E_u(x, v) ,
\end{split}
\end{equation}
where $\Arrowvert \cdot\Arrowvert$ denotes any continuous norm on the tangent bundle $TM$.
\end{defn}

\begin{remark} \label{remark:hyp1}
Let $\Omega\subset \R^n$ be the disjoint union of finitely many compact, connected, strictly convex sets with smooth boundaries and take on $\R^n$ the Euclidean metric. Then \cite[Chapter~4.4]{Chernov.2006} and \cite[Appendix]{ChaubetPetkovBilliards} show that the associated non-grazing billiard is hyperbolic on its trapped set. See also \cite[Section~5.2]{dyatlov2018notes} for an introductory exposition of hyperbolicity of dispersing billiards.  
\end{remark}

\subsection{Reflection-invariance of the canonical contact structure}\label{sec:contact1}

The sphere bundle $S\Sigma$ carries a canonical contact form $\alpha$ corresponding to the Liouville form (the tautological $1$-form) on the co-sphere bundle $S^\ast\Sigma$ under the diffeomorphism $S\Sigma\cong S^\ast\Sigma$ provided by the Riemannian metric $g$, i.e., $\alpha_{(x,v)}(w) \defgr g_x(v, \mathrm{d}\pi_{(x, v)} w)$ for the projection $\pi(x, v) = x$. The restriction of $\alpha$ to the submanifold $\partial M = (S\Sigma)|_{\partial\Omega} \subset S\Sigma$  can be pulled back along the tangential reflection map  $R:\partial M\to \partial M$, $(x,v)\mapsto (x,v')$. It turns out that $\alpha|_{\partial M}$ is invariant under this pullback, a property we will need later on:
\begin{lem}\label{lem:alphainv}
One has the equality of $1$-forms $R^\ast (\alpha|_{\partial M})=\alpha|_{\partial M}$.
\end{lem}
\begin{proof}
Let $\pi: \partial M = (S\Sigma)|_{\partial\Omega} \rightarrow \partial\Omega$ be the bundle projection $(x,v)\mapsto x$. Then we compute for $(x,v)\in \partial M$, $w\in T_{(x,v)}(S\Sigma|_{\partial\Omega})$ using the formula $v'=v - 2g_x(v, n_x) n_x$ featuring the inward normal vector $n_x\perp T_x (\partial\Omega)$:
\begin{align*}
\alpha_{(x,v)}(w)&=g_x(v, \mathrm{d}\pi_{(x, v)} (w)),\\
R^\ast (\alpha|_{\partial M})_{(x,v)}(w) &= g_x(v',\mathrm{d}(\pi\circ R)_{(x, v)}(w))\\
&= g_x(v - 2g_x(v, n_x) n_x, \mathrm{d}\pi_{(x, v)} (w))\\
&= \alpha_{(x,v)}(w) - 2 g_x(v, n_x) \underbrace{g_x(n_x, \mathrm{d}\pi_{(x, v)} (w))}_{=0}.
\end{align*}
Here we used the facts that $\pi\circ R=\pi$ and $\mathrm{d}\pi_{(x, v)}(w)\in T_x (\partial\Omega)$.
\end{proof}

\subsection{Billiard functions and the billiard generator}\label{sec:CBillP}

Although the non-grazing billiard flow $\varphi:D\to M\setminus \partial_\mathrm{g} M$ is not continuous and partially violates the flow property (recall Remark \ref{rem:notflow}), and therefore does not possess a generating vector field, we can associate with $\varphi$ natural function spaces as well as an operator providing a replacement for the generator of $\varphi$. 

Indeed, recall from \Cref{lem:D}  that $D$ is open in $\R\times (M\setminus \partial_\mathrm{g} M)$ and that more generally  for every open set $O\subset M\setminus \partial_\mathrm{g} M$ such that $O\cap (\partial M\setminus \partial_\mathrm{g} M)$ is reflection-symmetric  the inverse image $\varphi^{-1}(O)\subset D$ is open. Fix such a set $O$. Then we define the vector space of (compactly supported) \emph{smooth billiard functions} on $O$ by
\bq
\mathrm{C}^\infty_\mathrm{Bill}(O) \defgr \{f\in \mathrm{C}^\infty(O)\,|\, f\circ \varphi\in \mathrm{C}^\infty(\varphi^{-1}(O))\},\quad \mathrm{C}^\infty_{\mathrm{Bill},\mathrm{c}}(O) \defgr\mathrm{C}^\infty_{\mathrm{Bill}}(O)\cap  \mathrm{C}^\infty_\mathrm{c}(O). \label{eq:CBill}
\eq

These spaces are non-trivial: The smoothness of $\varphi$ on $\varphi^{-1}(\mathring M)$ implies that there is an injection $\mathrm{C}^\infty_\mathrm{c}(O\cap\mathring M)\hookrightarrow  \mathrm{C}^\infty_{\mathrm{Bill},\mathrm{c}}(O)$. We consider these spaces as interesting because they provide natural domains for the following differential operator
\begin{align}
\begin{split}
\mathbf{P}: \mathrm{C}^\infty_{\mathrm{Bill}}(O)&\longrightarrow \mathrm{C}^\infty_{\mathrm{Bill}}(O),\\
(\mathbf{P}f)(x,v) &\defgr \frac{\mathrm{d}}{\mathrm{d} t}\bigg|_{t = 0} f\circ \varphi_t(x,v),\qquad (x,v) \in O .
\label{eq:P}
\end{split}
\end{align}
The operator $\mathbf{P}$ is well-defined, preserves $\mathrm{C}^\infty_{\mathrm{Bill},\mathrm{c}}(O)$ and will henceforth be called the \emph{billiard generator}. This name is justified by \eqref{eq:P} which makes $\mathbf{P}$ a formal generator of the flow $\varphi$, where the discontinuity and the partial violation of the flow property of the latter are dealt with by passing to the function space $\mathrm{C}^\infty_{\mathrm{Bill}}(O)$. On the subspace $\mathrm{C}^\infty_\mathrm{c}(O\cap\mathring M)\subset \mathrm{C}^\infty_{\mathrm{Bill},\mathrm{c}}(O)$ the operator $\mathbf{P}$ simply acts as the geodesic vector field. In fact, if $X^g: \mathrm{C}^\infty(O)\to \mathrm{C}^\infty(O)$ denotes the generator of the geodesic flow $\varphi^g$ on the set $O$, then for every function $f\in \mathrm{C}^\infty_{\mathrm{Bill}}(O)$ the smooth function $\mathbf{P}f$ agrees with $X^g f$ on $O\cap\mathring M$ which is dense in $O$, so it follows that $\mathbf{P}f=X^gf$ on all of $O$. This shows that the billiard generator $\mathbf{P}$ is nothing but the restriction of $X^g$ to the domain $\mathrm{C}^\infty_{\mathrm{Bill}}(O)$:
\bq
\mathbf{P}=X^g|_{\mathrm{C}^\infty_{\mathrm{Bill}}(O)}.\label{eq:restrictionP}
\eq
In particular, we see that $X^g$ preserves $\mathrm{C}^\infty_{\mathrm{Bill}}(O)$ and $\mathrm{C}^\infty_{\mathrm{Bill}, \mathrm{c}}(O)$. 

Note that we can easily extend the construction \eqref{eq:CBill} to define analogous spaces of continuous billiard functions $\mathrm{C}_{\mathrm{Bill}}(O)$ and billiard functions of limited regularity $\mathrm{C}^N_{\mathrm{Bill}}(O)$.

We refrain from introducing topologies on the spaces $\mathrm{C}^\infty_{\mathrm{Bill}}(O)$, e.g. the subspace topologies induced by $\mathrm{C}^\infty(O)$, and $\mathrm{C}^\infty_{\mathrm{Bill},\mathrm{c}}(O)$ to avoid further technicalities. Instead we will introduce \emph{smooth models} for the billiard flow which will allow us to work with ordinary smooth functions and distributions on a smooth manifold as well as a smooth vector field $\mathbf{X}$ instead of the above defined operator $\mathbf{P}$.

\section{Smooth models for the non-grazing billiard flow}\label{sec:model} The fact that the non-grazing billiard flow is not continuous is highly inconvenient.  However, we shall see that a smooth model for $\varphi$ exists and is unique in a strong sense, so that we can consider it as intrinsic to the geodesic billiard system $(\Sigma, g, \Omega)$. The literature on billiards often presupposes a smooth model or works on $\varphi^{-1}(\mathring M)$ to begin with, see e.g.\ \cite{Chernov.2006}. Here we give a definition of smooth models in a slightly more general geometric setting and perform an explicit construction to show the existence of such a model. While smooth models no longer carry the bundle structure of $S\Sigma$, the smooth model flows remain \emph{contact flows}. This  constitutes a very convenient technical feature which is often implicit in concrete coordinate calculations in the billiard literature.


\begin{defn}\label{def:smoothmodel}A \emph{smooth model} for the non-grazing  billiard flow $\varphi:D\to M$ is a triple $(\M, \pi, \phi)$ consisting of a smooth manifold $\M$, a smooth surjection $\pi: M\setminus \partial_\mathrm{g}M\to \mathcal{M}$ such that $\mathcal D\defgr (\mathrm{id}_\R\times \pi)(D) \subset \R\times \mathcal M$ is open, and a smooth flow $\phi:\mathcal D\to \M$ such that
\begin{enumerate}
\item[i)] The restriction $\pi|_{\mathring M}$ is a diffeomorphism onto its image.
\item[ii)] The flows $\varphi$ and $\phi$ are intertwined by $\pi$:
		\bq
		\phi\circ (\mathrm{id}_\R\times \pi)|_D=\pi\circ\varphi.\label{eq:flowcompat}
		\eq
\end{enumerate}
\end{defn}
We emphasize that in the above definition $\phi$ must be a flow in the usual sense -- the exceptional generalized terminology introduced in Remark \ref{rem:notflow} only applies to $\varphi$. 

Definition \ref{def:smoothmodel} is motivated by the following existence and uniqueness results.
\begin{theorem}\label{thm:smooth_billiard}
There exists a smooth model $(\mathcal{M}, \pi, \phi)$ for $\varphi$ such that $\phi$ is a contact flow. 
\end{theorem}
\begin{proof}
In the subsequent Section \ref{construction} we give an explicit construction of a manifold $\M$, a map $\pi$, a flow $\phi$, and a contact form $\alpha_\mathcal{M}$ with the required properties, culminating in the final \Cref{cor:main}.
\end{proof}

\begin{prop}\label{prop:uniqueness}
Suppose that $(\M,\pi,\phi)$  and $(\M',\pi',\phi')$ are two smooth models for $\varphi$. 
Then $(\M,\phi)$ and $(\M',\phi')$ are uniquely smoothly conjugate. More precisely, there is a unique diffeomorphism $F:\M\to \M'$ such that $F\circ \pi=\pi'$, $(\mathrm{id}_\R\times F)(\mathcal D)=\mathcal D'$, and $F\circ \phi=\phi'\circ (\mathrm{id}_\R\times F)|_{\mathcal D}$.
\end{prop}

The proof of Proposition \ref{prop:uniqueness} is given in Appendix \ref{sec:proof3}.

\begin{cor}\label{cor:contact_prop}
Let $(\mathcal{M}, \pi, \phi)$ be a smooth model for $\varphi$. Then $\phi$ is a contact flow, i.e., there exists a contact form $\alpha_\mathcal{M}$ on $\mathcal{M}$ whose Reeb vector field is the generator $\mathbf{X}$ of $\phi$.
\end{cor}
\begin{proof}
A contact form with the desired property is provided by the pullback of the contact form whose existence is guaranteed by \Cref{thm:smooth_billiard} along the unique diffeomorphism of \Cref{prop:uniqueness}. 
\end{proof}

\subsection{Smooth models and the billiard generator}
Here we show that smooth models for the non-grazing billiard flow $\varphi$ are naturally related to the spaces of  billiard functions and the billiard generator $\mathbf{P}$ defined in Section \ref{sec:CBillP}. In the following, let $(\mathcal{M}, \pi, \phi)$ be a smooth model for $\varphi$ as in  \Cref{def:smoothmodel}, let $\mathcal O\subset \M$ be an open set, and write $O:=\pi^{-1}(\mathcal O)\subset M\setminus \partial_\mathrm{g} M$. Then $O\cap (\partial M\setminus \partial_\mathrm{g} M)$ is reflection-symmetric  by \Cref{lem:G}. Further, we denote by $\mathbf{X}: \mathrm{C}^\infty(\mathcal O) \rightarrow \mathrm{C}^\infty(\mathcal O)$ the generator of the smooth flow $\phi$ on $\mathcal O$.
\begin{prop} \label{prop:pullbacks}
The pullback $\pi^*: \mathrm{C}^\infty(\mathcal O) \to \mathrm{C}^\infty(O)$ is injective and  one has
\bq
\pi^*(\mathrm{C}^\infty(\mathcal O))=\mathrm{C}^\infty_{\mathrm{Bill}}(O),\qquad \pi^*(\mathrm{C}^\infty_\mathrm{c}(\mathcal O))=\mathrm{C}^\infty_{\mathrm{Bill},\mathrm{c}}(O).\label{eq:pullbacks}
\eq
Moreover, we have the equality 
	\begin{equation}
	\pi^* \circ \mathbf{X}\circ (\pi^*)^{-1} = \mathbf{P}\label{eq:XP}
	\end{equation}
	of linear operators $\mathrm{C}^\infty_{\mathrm{Bill}}(O)\to \mathrm{C}^\infty_{\mathrm{Bill}}(O)$ or $\mathrm{C}^\infty_{\mathrm{Bill},\mathrm{c}}(O)\to \mathrm{C}^\infty_{\mathrm{Bill},\mathrm{c}}(O)$.
\end{prop}
\begin{proof}The injectivity of $\pi^*$ is due to the surjectivity of $\pi$. The inclusions  $\pi^*(\mathrm{C}^\infty(\mathcal O))\subset \mathrm{C}^\infty_{\mathrm{Bill}}(O)$ and $\pi^*(\mathrm{C}^\infty_\mathrm{c}(\mathcal O))\subset \mathrm{C}^\infty_{\mathrm{Bill},\mathrm{c}}(O)$ follow from the fact that $\pi\circ\varphi:D\to \M$ is smooth by  \eqref{eq:flowcompat} and that $\pi$ is proper by \Cref{lem:G}. 

 To prove the reverse inclusions, let $f$ be a function that belongs to one of the billiard function spaces appearing on the right-hand side of \eqref{eq:pullbacks}. Using Lemma \eqref{lem:G}, define  $g:\mathcal O\to \C$ by
\[
g(p):=\begin{cases}
f(\pi|_{\mathring M}^{-1}(p)),& p\in \mathcal O\cap\pi(\mathring M),\\
f(\pi|_{\partial_\mathrm{in}M}^{-1}(p)), & p\in \mathcal O\cap\pi(\partial_\mathrm{in}M)=\mathcal O\cap \G.
\end{cases}
\]
Then we have $g\circ\pi=f$, in particular, $g$ has compact support if $f$ has compact support because $\pi$ is continuous. By the same argument as in the proof of \Cref{prop:uniqueness} proving that $g$ is smooth reduces to showing that for an open set $V\subset (\R\times \G)\cap \mathcal{D}$ containing $\{0\}\times \G$  the composition $g\circ \phi|_W: W\to \C$ is smooth, where $W:=\phi^{-1}(\mathcal O)\cap V$. To this end we use \eqref{eq:flowcompat}, by which  $g\circ \phi\circ (\mathrm{id}_\R\times \pi)|_D=g\circ \pi\circ \varphi=f\circ \varphi$. Since $\pi|_{\partial_\mathrm{in} M}:\partial_\mathrm{in} M\to  \G$ is a diffeomorphism by \Cref{lem:G}, we see that $g\circ \phi|_W$ is smooth iff the restriction of $f\circ \varphi$ to the set $(\mathrm{id}_\R\times \pi|^{-1}_{\partial_\mathrm{in} M})(W)\subset \R\times \partial_\mathrm{in} M$ is smooth. The latter holds true since $f$ is a smooth billiard function.  Indeed, $(\mathrm{id}_\R\times \pi|^{-1}_{\partial_\mathrm{in} M})(W)$ is open in  $\R\times \partial_\mathrm{in} M$ which is a  boundary submanifold of $\R\times (M\setminus \partial_\mathrm{g} M)$ and the restriction of the smooth function $f\circ \varphi$ to that boundary submanifold is again smooth. We conclude that $f=\pi^\ast g$.

To finally prove \eqref{eq:XP}, we first note that the generator $\mathbf{X}$ of $\phi$ acts on $f\in \mathrm{C}^\infty(\mathcal O)$ via
	\bqn
	\mathbf{X} f(p) = \frac{\mathrm{d}}{\mathrm{d} t}\bigg|_{t = 0} f\circ \phi_t(p), \quad p\in \mathcal O.
	\eqn
	Given $p\in O$ and $f\in \mathrm{C}^\infty(\mathcal O)$ we therefore calculate using \eqref{eq:flowcompat}
	\bqn
	\mathbf{P} (f\circ \pi)(p) = \frac{\mathrm{d}}{\mathrm{d} t}\bigg|_{t = 0} f\circ \pi\circ \varphi_t(p) = \frac{\mathrm{d}}{\mathrm{d} t}\bigg|_{t = 0} f\circ \phi_t(\pi(p)) = \mathbf{X} f(\pi(p)) ,
	\eqn
	finishing the proof.
\end{proof}

\subsection{Smooth trapped set, closed trajectories, and hyperbolicity}

We already defined the trapped set of the non-grazing billiard flow $\varphi$ in Section \ref{section_trapped}. Given a smooth model $(\M,\pi,\phi)$ of $\varphi$ as in Definition \ref{def:smoothmodel} we define the corresponding notion of trapped set for $\phi$ as
\begin{equation} \label{eq:Ksmooth}
\mathcal{K} \defgr \{p\in \M \,|\, \R\times\{p\} \subset\mathcal{D},\; \exists ~\text{compact}~\mathcal{W} \subset \M ~\text{with}~ \phi(\R\times \{p\})\subset \mathcal{W}\} .
\end{equation}
We then get the following dynamical correspondence between the non-grazing billiard flow and its smooth model flow:
\begin{prop} \label{prop:trajectory_bijection}
The equalities $\mathcal{K} = \pi(K)$, $\pi^{-1}(\mathcal{K}) = K$ hold, and there exists a natural period-preserving bijection between the closed trajectories of $\phi$ and the closed trajectories of $\varphi$. More precisely:
\begin{enumerate}
\item[i)] Given $p\in \M$ with $\phi_T(p) = p$ for some $T > 0$ and $\phi_t(p)\neq p$ for all $t\in (0,T)$, there exists $(x,v)\in M\setminus \partial_\mathrm{g} M$ such that $\varphi_T(x,v)=(x,v)$, $\varphi_t(x,v)\neq (x,v)$ for all $t\in (0,T)$, and $\pi(x,v) = p$.

\item[ii)] Conversely, given $(x,v)\in M\setminus \partial_\mathrm{g} M$ with $\varphi_T(x,v) = (x,v)$ for some $T > 0$ and $\varphi_t(x,v)\neq (x,v)$ for all $t\in (0,T)$, then $\pi(x,v)\in \mathcal{K}$, $\phi_T(\pi(x,v))=\pi(x,v)$, and $\phi_t(\pi(x,v))\neq\pi(x,v)$ for all $t\in (0,T)$.
\end{enumerate}
\end{prop}

\begin{proof}
The equalities $\mathcal{K} = \pi(K)$, $\pi^{-1}(\mathcal{K}) = K$ follow from \eqref{eq:flowcompat} and the facts that $\pi$ is continuous and also proper by Lemma \ref{lem:G}.

Let $p = \pi(x, v)\in \mathcal{M}$ be as in \emph{i)}. If $p\in\pi\left( \partial M\setminus \partial_\mathrm{g} M \right)$ we can make the choice $(x, v)\in \partial_\mathrm{in} M$ because by Lemma \ref{lem:G} and $\mathcal{M} = \pi(\mathring{M})\sqcup \mathcal{G}$ we must have $\pi^{-1}(p) = \{(x, v), (x, v')\}$. Then, regardless of whether $p\in\pi\left( \partial M\setminus \partial_\mathrm{g} M \right)$ or not, $\phi_T(p) = p$ implies $\varphi_T(x, v) = (x, v)$ and $\varphi_t(x, v) = (x, v)$ for $t\in (0, T)$ would imply the contradiction $\phi_t(p) = p$. Claim \emph{ii)} can be checked directly by using the relation $\phi_t\circ \pi = \pi\circ\varphi_t$.
\end{proof}

Finally we discuss hyperbolicity of our smooth models: The model flow $\varphi$ is called \emph{hyperbolic on its trapped set $\mathcal{K}$} if the following condition similar to \Cref{def:hypbilliard} holds: For any $p\in \mathcal{K}$ the tangent bundle $T_p \mathcal{M}$ splits in a continuous and flow invariant fashion as
\begin{equation} \label{eq:open_splitting}
T_p \mathcal{M} = \mathbb{R}\cdot \mathbf{X}(p) \oplus \mathcal{E}_s(p) \oplus \mathcal{E}_u(p) ,
\end{equation}
and there exist constants $C_0, C_1 > 0$ such that
\begin{equation} \label{eq:open_hyperbolic}
\begin{split}
\Arrowvert \mathrm{d}\phi_t(p) W\Arrowvert_{\phi_t(p)} &\leq C_0 \exp(-C_1 t) \Arrowvert W\Arrowvert_p, \quad t\geq 0,~ W\in \mathcal{E}_s(p) \\
\Arrowvert \mathrm{d}\phi_t(p) W\Arrowvert_{\phi_t(p)} &\geq C_0^{-1} \exp(C_1 t) \Arrowvert W\Arrowvert_p, \quad t\geq 0,~ W\in \mathcal{E}_u(p) ,
\end{split}
\end{equation}
where $\Arrowvert \cdot\Arrowvert$ denotes any continuous norm on $T\mathcal{M}$. The next proposition connects hyperbolicity of $\varphi$ with hyperbolicity of its smooth model flows:
\begin{prop}\label{thm:hypsmooth}
	Let $\varphi: D\rightarrow M$ be a non-grazing billiard flow that is hyperbolic on its trapped set $K$. Then any smooth model $(\mathcal{M},\pi,\phi)$ for $\varphi$ is hyperbolic on its trapped set $\mathcal{K}$.
\end{prop}
\begin{proof}
	We construct the hyperbolic splitting over $\mathcal{K}$ as follows: On $\mathring{M}$ the natural candidate is the one already given in \eqref{eq:hyperbolic_splitting} and transported via the differential of the diffeomorphism $\kappa\defgr \pi|_{\mathring{M}}$, i.e., for $p = \pi(x, v)\in \mathcal{M}\setminus\mathcal{G}$ we have
	\begin{equation} \label{eq:transported_splitting}
	T_{p} \mathcal{M} = \R\cdot \mathbf{X}(p) \oplus \mathcal{E}_s(p) \oplus \mathcal{E}_u(p),
	\end{equation}
	where $\mathcal{E}_{s/u}(\pi(x, v)) \defgr \mathrm{d}\kappa(x, v) E_{s/u}(x, v)$ and $\mathbf{X}(\pi(x, v)) = \mathrm{d}\kappa(x, v) X(x, v)$ is the generator of $\phi$ evaluated at $\pi(x, v)$. This splitting is again invariant under $\phi_t$ whenever $\phi_t(\pi(x, v))\in \mathcal{K}\cap \pi(\mathring{M}) = \mathcal{K}\cap (\mathcal{M} \setminus \mathcal{G})$ by the relation $\phi_t\circ \pi = \pi\circ \varphi_t$ and the flow invariance of the original splitting.
	
	We now extend this splitting to all of $\mathcal{K}$ as follows: For  $p = \pi(x, v)\in \mathcal{K}\cap \mathcal{G}$ we define
	\begin{equation*}
	\mathcal{E}_{s/u}(p) = \mathrm{d}\phi_{-t} \left( \mathcal{E}_{s/u}(\phi_t(p)) \right)
	\end{equation*}
	for any $t\in\R$ such that $\phi_{t}(p)\notin\mathcal{G}$; in particular, any $t\neq 0$ close enough to $0$ will do the job. By the flow property of $\phi$ and the flow invariance of the original splitting \eqref{eq:hyperbolic_splitting} this definition is independent of $t$, \eqref{eq:transported_splitting} holds for $p$ as $\mathrm{d}\phi_{-t}$ is an isomorphism $T_{\phi_t(p)}\mathcal{M}\rightarrow T_p\mathcal{M}$, and the obtained splitting is continuous by continuity of $\phi$.
	
	It remains to show that the hyperbolicity estimates \eqref{eq:def_hyperbolicity} hold. Given $W\in \mathcal{E}_s(p)$ and arbitrary $t'\geq 0$ we calculate
	\begin{equation*}
	\Arrowvert \mathrm{d}\phi_{t'}(p) W\Arrowvert_{\phi_{t'}(p)} \leq C_0 \mathrm{e}^{-C_1 (t' - t)} \Arrowvert \mathrm{d}\phi_t(p) W \Arrowvert_{\phi_t(p)} ,
	\end{equation*}
	where $t>0$ is sufficiently small such that $\phi((0, t]\times \{p\}) \subset\mathcal{M}\setminus \mathcal{G}$. In the limit $t\rightarrow 0$ we obtain the desired estimate.
\end{proof}

\begin{remark} \label{remark:hyp2}
Let $\varphi$ be the non-grazing billiard flow constructed from the Euclidean metric on $\R^n$ and the disjoint union of finitely many compact, connected, strictly convex obstacles with smooth boundaries. Then \Cref{thm:hypsmooth} combined with \Cref{remark:hyp1} shows that any smooth model for $\varphi$ is hyperbolic on its trapped set.
\end{remark}

\section{Construction of a smooth model} \label{construction}
This section is devoted to proving the existence \Cref{thm:smooth_billiard} by explicitly constructing the required objects. In view of the strong uniqueness result \Cref{prop:uniqueness} our construction method is essentially unique. We break up the proof into several lemmas and corollaries until we arrive at the final Corollary \ref{cor:main}. 

\subsection{The topological space $\M$ and continuous flow $\phi$} \label{sec:M}
We first define $\M$ as a topological space and $\phi$ as a continuous flow. In the subsequent Section \ref{sec:smoothstructure} we proceed to proving that $\M$ can be equipped with a smooth structure such that $\phi$ is smooth. 

We define our model space as
\[
\M \defgr (M\setminus \partial_\mathrm{g}M)/\sim,
\]
where the equivalence relation $\sim$ on $M\setminus \partial_\mathrm{g}M$ is defined by the equivalence classes
\[
[x,v]:= \begin{cases}\{(x,v)\},\qquad& (x,v)\in \mathring M,\\
\{(x,v),(x,v')\},& (x,v)\in \partial M.
\end{cases}
\]
We equip $\M$ with the quotient topology and denote by
\bq
\pi: M\setminus \partial_\mathrm{g}M \to \M,\qquad (x,v)\mapsto [x,v],\label{eq:piconstr}
\eq
the canonical projection.  We call the set $\mathcal G:=\pi(\partial M\setminus \partial_\mathrm{g}M)\subset \M$ formed by all $2$-element equivalence classes the \emph{gluing region}.  It is a closed subset of $\M$ since $\pi^{-1}\left(\M\setminus \G\right) = \mathring M$ is open in $M$. As suggested by \Cref{def:smoothmodel}, we define the  domain
\begin{equation*}
\mathcal{D} \defgr (\mathrm{id}_\R\times \pi)(D)\subset \R\times \M,
\end{equation*}
where $D$ is the non-grazing flow domain of $\varphi$ defined in \eqref{eq:D}. The symmetry \eqref{eq:Dinv} of $D$ under tangential reflection  implies that $(\mathrm{id}_\R \times \pi)^{-1}(\mathcal{D}) = D$,  so \Cref{lem:D} implies that  $\mathcal{D}$ is open in $\R\times \M$ as required by  \Cref{def:smoothmodel}. The  compatibility property \eqref{eq:invariance} of $\varphi$  with the tangential reflection now allows us to define a flow
\[
\phi:\mathcal D\to \M,\qquad \phi(t,[x,v]):=[\varphi(t,x,v)],\qquad (t,x,v)\in D,
\]
which by construction satisfies the relation $\phi\circ (\mathrm{id}_\R\times \pi)=\pi\circ\varphi$ on $D$. 

The main motivation for the definition of $\M$ using the equivalence relation $\sim$ is the \emph{continuity} of the flow $\phi$:
\begin{lem}
The flow $\phi:\mathcal D \to \M$ is continuous.
\end{lem}
\begin{proof}
Given an open set $\mathcal{O} \subset \mathcal{M}$ we first note that $O\defgr \pi^{-1}(\mathcal{O}) \subset M\setminus \partial_\mathrm{g} M$ is open by continuity of $\pi$. Since $O\cap (\partial M\setminus \partial_\mathrm{g} M)$ is reflection-symmetric in view of the definition of $\pi$, \Cref{lem:D} tells us that $\varphi^{-1}(O)$ is open in $D$. Now we simply calculate
\begin{equation}
(\mathrm{id}_\mathbb{R} \times \pi)^{-1} (\phi^{-1}(\mathcal{O})) = (\pi\circ \varphi)^{-1}(\mathcal{O}) = \varphi^{-1}(O) ,
\end{equation}
which by definition of the quotient topology shows that $\phi^{-1}(\mathcal O)$ is open in $\mathcal D$.
\end{proof}

\subsection{Smooth structure}\label{sec:smoothstructure} The topological gluing process carried out in Section \ref{sec:M} to define $\M$ does not automatically equip $\M$ with any canonical smooth structure. However, since our goal is to make $\varphi$ smooth, it suggests itself to use \emph{flow charts} around the gluing region in $\M$ to define the smooth structure. 

More precisely, to equip $\M$ with a smooth structure we choose an open set $N = N_\mathrm{in}\sqcup N_\mathrm{out} \subset \R\times (\partial M\setminus \partial_\mathrm{g}M)$ as in Lemma \ref{lem:convexity} and define the continuous map 
\[
\Phi:N_\mathrm{in}\to \M,\qquad \Phi(t,x,v)=[\varphi(t,x,v)]=\phi(t,[x,v]).
\]
Note that choosing $N_\mathrm{in}$ over $N_\mathrm{out}$ is arbitrary; $N_\mathrm{out}$ defines an equivalent smooth structure in the arguments below since the involution $N_\mathrm{in}\to N_\mathrm{out}$, $(t,x,v)\mapsto (t,x,v')$,  is a canonical diffeomorphism between the two. The key observation is that $\Phi$ is an embedding:
\begin{cor}\label{cor:U}
The map $\Phi$ is a homeomorphism onto its image which is an open neighborhood of the gluing region $\G$.
\end{cor}
\begin{proof}
The set $\Phi(N_\mathrm{in})$ contains $\G$ because $N_\mathrm{in}$ contains $\{0\}\times  \partial_\mathrm{in}M$ and $\pi(\{0\}\times  \partial_\mathrm{in}M)= \G$. Since we have $\Phi(N_\mathrm{in}) = \pi(\varphi(N_\mathrm{in}))$ and $\pi|_{\varphi(N_\mathrm{in})}$ is an open map, we only need to prove that $\varphi: N_\mathrm{in}\to M$ is an injective open map. This is true by Lemma \ref{lem:convexity}.
\end{proof}

\begin{cor}\label{cor:main}
The model space $\M$ can be equipped with a smooth structure such that the projection $\pi:M\setminus \partial_\mathrm{g} M\to \M$ is a smooth map and the flow $\phi: \mathcal D\to\M$ is smooth. Furthermore, there exists a contact form $\alpha_\mathcal{M}$ on $\mathcal{M}$ whose Reeb vector field is the generator $\mathbf{X}$ of $\phi$.
\end{cor}
\begin{proof}
At this point we have at our disposal the two homeomorphism $\Phi:N_\mathrm{in} \to \Phi(N_\mathrm{in})$ and $\pi|_{\mathring M}: \mathring M\to \M\setminus \G$ whose codomains provide an open cover of $\mathcal M$. Since $\mathring M$ as well as $N_\mathrm{in}$ are smooth manifolds, we see that $\M$ is second-countable and Hausdorff. To equip $\M$ with an atlas we take on $\M\setminus \G$ the diffeomorphism $\pi|_{\mathring M}^{-1}$ as a chart and on $\Phi(N_\mathrm{in})$ we use $\Phi^{-1}$ as a chart. As $\varphi$ is smooth on $\varphi^{-1}(\mathring M)$, the so-defined charts are compatible on the overlap $\Phi(N_\mathrm{in})\cap (\M\setminus \G)$ and thus define a smooth structure on $\M$.  

Now $\pi|_{\mathring M}$ is a diffeomorphism and in particular smooth. On the other hand  $\pi|_{\varphi(N_\mathrm{in})}$ is smooth if $\Phi^{-1}\circ \pi|_{\varphi(N_\mathrm{in})} = (\varphi|_{N_\mathrm{in}})^{-1}: \varphi(N_\mathrm{in})\to N_\mathrm{in}$ is smooth. The latter holds true by Lemma \ref{lem:convexity} \emph{v)}. 

Since $\phi$ is a continuous flow and $\{\Phi(N_\mathrm{in}), \M\setminus \G\}$ constitutes an open cover of $\M$, proving that $\phi$ is smooth reduces to showing that both $\phi: \mathcal{D} \cap \left(\R\times (\M\setminus \G) \right) \to \M$ as well as $\phi: \mathcal{D}\cap \left(\R\times (\Phi(N_\mathrm{in})) \right) \to \M$ are smooth. Note that by the flow property one only needs to check smoothness around points $(0, [x, v])$, i.e., the problem reduces to checking smoothness in the cases $[x, v]\in \pi(\mathring{M})$ and $[x, v]\in\mathcal{G}\subset \Phi(N_\mathrm{in})$.

The former easily follows from the smoothness of $\varphi$ on $\varphi^{-1}(\mathring M)$. For the latter we take $(s, y, w)\in N_\mathrm{in}$ and calculate for sufficiently small $t\in\R$
\begin{equation} \label{eq_flow_in_coord}
\Phi^{-1}\circ \phi\circ \left(\mathrm{id}_\R\times \Phi \right) (t, s, y, w) = \Phi^{-1}\circ\phi_t(\phi_s(y, w)) = (t + s, y, w) ,
\end{equation}
by virtue of the flow property $\varphi(t, \varphi(s,x,v))=\varphi(s+t,x,v)$. This is obviously smooth.

We begin the construction of $\alpha_\mathcal{M}$ by noting that in the setting of \Cref{lem:convexity} the geodesic flow $\varphi^g$ is a diffeomorphism $N\rightarrow M\setminus\partial_\mathrm{g} M$ and the map $\Phi$ provides a chart around $\mathcal{G}$ with respect to which the generator of $\phi$ is given by $\Phi_* \mathbf{X} = \partial_t$.

Now let $\alpha\in \Omega^1(S\Sigma)$ be the canonical contact form on $(S\Sigma, g)$ whose Reeb vector field is the geodesic vector field $X^g$ (for details see \cite[Chap.~1]{Paternain.1999}). Note that the equation $\phi\circ (\mathrm{id}_\R\times \pi) = \pi\circ\varphi$ on $D$ immediately entails $\mathbf{X} = \pi_* X^g$ on $\pi(\mathring{M}) = \mathcal{M}\setminus \mathcal{G}$ and the $1$-form defined via
\begin{equation*}
\alpha_\mathcal{M} \defgr \Big( \pi\big|_{\mathring{M}}^{-1} \Big)^* \alpha \in \Omega^1(\pi(\mathring{M}))
\end{equation*}
thus still satisfies $\iota_\mathbf{X}\alpha_\mathcal{M} = 1$ and $\iota_\mathbf{X} \mathrm{d}\alpha_\mathcal{M} = 0$. We will now continue this definition smoothly to $\mathcal{G}$: First observe that $\alpha$ is $\varphi_t^g$-invariant and $(\varphi^g|_N)_* X^g = \partial_t$, ergo
\begin{equation} \label{eq_contact1}
\left( \varphi^g\big|_N \right)^* \alpha(t, x, v) = \mathrm{d}t + \alpha|_{\partial M \setminus\partial_\mathrm{g} M}(x, v).
\end{equation}
Denote the restriction away from $\mathcal{G}$ of our above flow chart as $\Phi' \defgr \Phi\big|_{\{t\neq 0\}}$, where $\{t\neq 0\} = \Phi^{-1}(\mathcal{M}\setminus \mathcal{G}) = N\setminus (\{0\}\times \mathcal{G})$. We first observe that
\begin{equation} \label{eq_contact2}
(\Phi')^* \alpha = \left( \pi\big|_{\mathring{M}}^{-1}\circ \Phi' \right)^* \alpha = \left( \varphi\big|_{N_\mathrm{in}\cap \{t\neq 0\}} \right)^* \alpha .
\end{equation}
But $\varphi|_{N_\mathrm{in}\cap \{t < 0\}} = \varphi^g|_{N_\mathrm{in}\cap \{t < 0\}}$ and $\varphi|_{N_\mathrm{in}\cap \{t > 0\}} = \varphi^g\circ \widetilde{R}|_{N_\mathrm{in}\cap \{t > 0\}}$ where $\widetilde{R}(t, x, v) = R(t, x, v')$ denotes the obvious lift of the tangential reflection  $R(x, v) = (x, v')$ to $N$. Combining this with \eqref{eq_contact1} and \eqref{eq_contact2} yields
\begin{equation} \label{eq_contact3}
\left( \Phi' \right)^* \alpha_\mathcal{M}(t, x, v) =
\begin{cases}
\mathrm{d}t + \alpha|_{\partial M \setminus\partial_\mathrm{g} M}(x, v), &t < 0\\
\mathrm{d}t + R^* (\alpha|_{\partial M \setminus\partial_\mathrm{g} M})(x, v), &t > 0 .
\end{cases}
\end{equation}
We have already seen in \Cref{lem:alphainv} that $R^* (\alpha|_{\partial M})=\alpha|_{\partial M}$ which implies that we can interpret the right-hand side of \eqref{eq_contact3} as defined on the whole coordinate domain $N_\mathrm{in}$ and the definition $\alpha_\mathcal{M}\defgr \left( \Phi^{-1} \right)^* (\mathrm{d}t + \alpha|_{\partial M \setminus\partial_\mathrm{g} M})$ extends $\alpha_\mathcal{M}$ to a well-defined $1$-form on all of $\mathcal{M}$ which is still a contact form with $\mathbf{X}$ its Reeb vector field by \eqref{eq_contact3}.
\end{proof}

\begin{remark}[Flow time vs.\ Riemannian distance as transversal coordinate]
We emphasize the fact that the particularly simple coordinate expression of the flow in \eqref{eq_flow_in_coord} is due to the usage of \emph{flow coordinates} in the direction transversal to  $\G$ or, equivalently, to $\partial M\setminus \partial_\mathrm{g} M$. Alternatively one could consider the Riemannian distance $\mathrm{dist}_g(x,\Omega)$ as the transversal coordinate of a point $(x,v)\in M\setminus \partial_\mathrm{g} M$ close to $\partial_\mathrm{out} M$ and $-\mathrm{dist}_g(x,\Omega)$ if  $(x,v)$ is close to $\partial_\mathrm{in} M$. With this choice of coordinates on $\M$ near $\G$ the model flow $\phi$ would in general be non-smooth, though, as can be directly verified for e.g.\ $\Sigma=\R^2$, $\Omega=\{x\in \mathbb{R}^2 \,|\, \vert x\vert \leq 1\}$, equipped with the Euclidean metric.
\end{remark}

\subsection{Dependence of the model space on the Riemannian metric}

Suppose that $g$ and $g'$ are two complete Riemannian metrics on $\Sigma$. Then the unit tangent bundles with respect to $g$ and $g'$ are canonically diffeomorphic via the obvious rescaling diffeomorphism, so we can consider them as one and the same space $S\Sigma$ carrying the two geodesic flows $\varphi^g$ and $\varphi^{g'}$.  With this identification the  inward, outward, and grazing boundaries of $M$ are the same for the two metrics $g$ and $g'$. However, the tangential reflections on $\partial M\setminus \partial_\mathrm{g} M$ with respect to $g$ and $g'$ will differ in general. Let us denote them by 
\[
(x,v)\mapsto (x,R_{g(x)}v),\qquad (x,v)\mapsto (x,R_{g'(x)}v),
\]
respectively. Consider now the non-grazing billiard flows $\varphi_g:D_g\to M\setminus \partial_\mathrm{g}M$ and $\varphi_{g'}:D_{g'}\to M\setminus \partial_\mathrm{g}M$ of $g$ and $g'$ on their domains $D_g,D_{g'}\subset \R \times (M\setminus \partial_\mathrm{g}M)$.  It is a natural question how the smooth models $(\M_g,\pi_g,\phi_g)$ and $(\M_{g'},\pi_{g'},\phi_{g'})$ for $\varphi_{g}$ and $\varphi_{g'}$, as constructed above,  are related. In particular, we would like to know when there is a diffeomorphism $\M_g\cong \M_{g'}$ making the diagram
\bq
\begin{tikzcd}
& M\setminus \partial_\mathrm{g}M  \arrow{ld}{\pi_g}\arrow{rd}{\pi_{g'}} & \\%
\M_g \arrow{rr}{\cong}& & \M_{g'}
\end{tikzcd}\label{eq:diffpi}
\eq
commute. In this case, one can consider  $\phi_{g}$ and $\phi_{g'}$ as flows on the same smooth manifold, which allows to compare them. 

An answer to this question is given by the following result that describes a regularity condition on the geodesic flows $\varphi^g$, $\varphi^{g'}$ and the tangential reflections with respect to $g$ and $g'$ which is both necessary and sufficient for \eqref{eq:diffpi}. In order to formulate the regularity condition we need to introduce some more terminology: Since the geodesic flows $\varphi^g$ and $\varphi^{g'}$ are transversal to $\partial M\setminus \partial_\mathrm{g}M$ the inverse function theorem tells us that we can find an open neighborhood $N$ of $\{0\}\times(\partial M\setminus \partial_\mathrm{g}M)$ in $\R\times(\partial M\setminus \partial_\mathrm{g}M)$ such that $\varphi^g|_N$ and $\varphi^{g'}|_N$ are diffeomorphisms onto their images in $S\Sigma$. For any such neighborhood we can  define ``geometric reflection maps'' $\tilde R_{g}: \varphi^g(N) \to  \varphi^g(N)$ and $\tilde R_{g'}: \varphi^{g'}(N) \to  \varphi^{g'}(N)$ by putting
\[
\tilde R_{g}(x,v):=\varphi^{g}(t,x_0,R_{g(x_0)}v_0),\qquad \varphi^{g}|_{N}^{-1}(x,v)=(t,x_0,v_0)\in N,
\]
and analogously for $g'$. With these preparations we can state
\begin{prop}\label{prop:gdependence}The following two statements are equivalent:
\begin{enumerate}
\item There is a diffeomorphism $\M_g\cong \M_{g'}$ making the diagram \eqref{eq:diffpi} commute.
\item There is an open neighborhood $N$ of $\{0\}\times(\partial M\setminus \partial_\mathrm{g}M)$ in $\R\times(\partial M\setminus \partial_\mathrm{g}M)$ such that the two maps $N\cap (\R\times \partial_\mathrm{in}M)\to S\Sigma$ given by
\begin{align*}
(t,x,v)&\mapsto\begin{cases}\varphi^g(t,x,v), & t\leq 0,\\
(\tilde R_{g'}\circ\varphi^g)(t,x,R_{g(x)}v), & t>0,\end{cases}\\
(t,x,v)&\mapsto\begin{cases}\varphi^{g'}(t,x,v), & t\leq 0,\\
(\tilde R_{g}\circ\varphi^{g'})(t,x,R_{g'(x)}v), & t>0,\end{cases}
\end{align*}
are well-defined and smooth.
\end{enumerate}
If (1) or equivalently (2) holds, then the diffeomorphism in \eqref{eq:diffpi} is unique.
\end{prop}
The proof of Proposition \ref{prop:gdependence} is given in Appendix \ref{sec:proof4}. 
\begin{remark}
An obvious case in which \Cref{prop:gdependence} can be applied is when $g$ and $g'$ differ only by a constant conformal factor in a neighborhood of $\partial\Omega$. In this case, the identification of the two unit tangent bundles directly eliminates that factor near $\partial M$ and \Cref{prop:gdependence} becomes trivial since $\M_g$ and $\M_{g'}$ coincide near their gluing regions.
\end{remark}

\section{Meromorphic continuation of the resolvent and weighted zeta function} \label{zeta}

In this section we derive two meromorphic continuation results: After recalling the meromorphically continued resolvent of \cite[Thm.~1]{Dyatlov.2016a} in Section \ref{zeta_open} we continue meromorphically a restricted resolvent of the generator of a smooth model flow for a non-grazing, hyperbolic billiard in Section \ref{zeta_resolvent}. This result immediately translates to the billiard operator $\mathbf{P}$ via the pullback $\pi^*$. Finally we derive the meromorphic continuation of weighted zeta functions for non-grazing billiard flows in Section \ref{zeta_continuation}. As a corollary we obtain meromorphic continuation of the weighted zeta function for Euclidean billiards in $\mathbb{R}^n$.

\subsection{Meromorphic continuation on open hyperbolic systems} \label{zeta_open}

In the following we will invoke the meromorphic continuation result obtained in \cite{Dyatlov.2016a} in the setting of \emph{open hyperbolic systems}. To make the paper more self-contained we recall their setting and results here: An open hyperbolic system is given by a flow $\psi$ on a compact manifold $\mathcal{U}$ with boundary satisfying the following requirements:

\textbf{(1)} The manifold boundary $\partial \mathcal{U}$ of $\mathcal{U}$ is smooth and \emph{strictly convex} w.r.t. the generator $X$ of $\psi$, i.e., for any boundary defining function $\rho\in \mathrm{C}^\infty(\mathcal{U})$
\begin{equation} \label{eq:strictconv}
p\in\partial \mathcal{U},\, (X \rho)(p) = 0 \quad\Longrightarrow\quad X(X\rho)(p) < 0 .
\end{equation}

\textbf{(2)} Let $K(\psi)$ denote the trapped set of $\psi$, i.e., the set of $p\in \mathcal{U}$ for which $\psi_t(p)$ exists $\forall t\in\R$. The flow $\psi$ is \emph{hyperbolic} on $K(\psi)$, i.e., for any $p\in K(\psi)$ the tangent bundle $T_p \mathcal{U}$ splits in a continuous and flow invariant fashion as
\begin{equation} \label{eq:open_splitting1}
T_p \mathcal{U} = \mathbb{R}\cdot X(p) \oplus E_s(p) \oplus E_u(p) ,
\end{equation}
and there exist constants $C_0, C_1 > 0$ such that
\begin{equation} \label{eq:open_hyperbolic1}
\begin{split}
\Arrowvert \mathrm{d}\psi_t(p) W\Arrowvert_{\psi_t(p)} &\leq C_0 \exp(-C_1 t) \Arrowvert W\Arrowvert_p, \quad t\geq 0,~ W\in E_s(p) \\
\Arrowvert \mathrm{d}\psi_t(p) W\Arrowvert_{\psi_t(p)} &\geq C_0^{-1} \exp(C_1 t) \Arrowvert W\Arrowvert_p, \quad t\geq 0,~ W\in E_u(p) ,
\end{split}
\end{equation}
where $\Arrowvert \cdot\Arrowvert$ denotes any continuous norm on $T\mathcal{U}$. Denote by $\mathring{\mathcal U}$ the manifold interior of $\mathcal U$. 

Now in this setting the following holds \cite[Thm.~1]{Dyatlov.2016a}: The family of operators
\begin{equation*}
\mathbf{R}(\lambda) \defgr \mathbf{1}_{\mathring{\mathcal U}} (X + \lambda)^{-1} \mathbf{1}_{\mathring{\mathcal U}}: \mathrm{C}^\infty_\mathrm{c}(\mathring{\mathcal U}) \rightarrow \mathcal{D}'(\mathring{\mathcal U})
\end{equation*}
is analytic for $\mathrm{Re}(\lambda) \gg 0$ and continues meromorphically to $\C$. Its poles are called \emph{Ruelle resonances} and the residue of $\mathbf{R}(\lambda)$ at a resonance $\lambda_0$ is given by a finite-rank operator
\begin{equation}
\Pi_{\lambda_0}: \mathrm{C}^\infty_\mathrm{c}(\mathring{\mathcal U}) \rightarrow \mathcal{D}'(\mathring{\mathcal U}) .\label{eq:Pilambda0}
\end{equation}
In particular Dyatlov and Guillarmou showed in \cite[Thm.~2]{Dyatlov.2016a} a very precise wavefront set estimate for the Schwartz kernel $K_{\Pi_{\lambda_0}}$ of $\Pi_{\lambda_0}$ which allows one to calculate the \emph{flat trace} $\mathrm{tr}^\flat$ of $\Pi_{\lambda_0}$ defined as the integral over the restriction of the kernel to the diagonal \cite[Section~4.1]{Dyatlov.2016a}:
\begin{equation*}
\mathrm{supp}( K_{\Pi_{\lambda_0}}) \subset \Gamma_+\times \Gamma_-, \quad \mathrm{WF}'( \Pi_{\lambda_0} ) \subset E^*_+\times E^*_- ,
\end{equation*}
where $\Gamma_\pm$ are the incoming/outgoing tales of $\psi$, i.e., those $p\in\mathcal{U}$ for which $\psi_{\mp t}(p)$ exists for all $t\geq 0$, and $E^*_\pm\subset T^*\mathcal{U}$ are extensions of the dual stable/unstable foliations $E^*_{u/s}$ onto $\Gamma_\pm$ constructed in \cite[Lemma~1.10]{Dyatlov.2016a}. For the regular (holomorphic) part $\mathbf{R}_H(\lambda)$ in the neighborhood of some $\lambda_0\in\mathbb{C}$ a similar estimate is known \cite[Lemma~3.5]{Dyatlov.2016a}:
\begin{equation} \label{eq_regular_wf}
\mathrm{WF}'( \mathbf{R}_H(\lambda) ) \subset \Delta( T^* \mathring{\mathcal{U}} ) \cup (E^*_+\times E^*_-)\cup \mathcal{Y}_+ ,
\end{equation}
where $\mathcal{Y}_+ \defgr \big\{\left( \mathrm{e}^{t H_p}(y, \eta), y, \eta \right) \,\big|\, t\geq 0,\, p(y, \eta) = 0,\, y\in\mathring{\mathcal{U}},\, \psi_t(\mathring{\mathcal{U}}) \big\}$ with $p(y, \eta) \defgr \langle X(y), \eta\rangle$ and $\mathrm{e}^{t H_p}$ the flow of the Hamiltonian vector field $H_p$ associated with $p$, and $\Delta( T^* \mathring{\mathcal{U}} ) \subset T^* \mathring{\mathcal{U}} \times T^* \mathring{\mathcal{U}}$ denotes the diagonal of the cotangent bundle over $\mathring{\mathcal{U}}$.

For the definition of the flat trace and the rather technical background on the related techniques we refer the reader to \cite[Section~2.4]{Dyatlov.2016}.

\begin{remark}\label{rem:vector1}
The setting of \cite{Dyatlov.2016a} covers the more general vector-valued case of a first order differential operator $\mathbf{X}:\mathrm{C}^\infty(\mathcal U;\mathcal E)\to \mathrm{C}^\infty(\mathcal U;\mathcal E)$ acting on smooth sections of a vector bundle $\mathcal E$ over $\mathcal U$ which is a lift of $X$ in the sense that
\[
\mathbf{X}(fs)=X(f)s + f \mathbf{X}(s)\qquad \forall\; s\in \mathrm{C}^\infty(\mathcal U;\mathcal E),\; f\in \mathrm{C}^\infty(\mathcal U).
\]
Our considerations of the following section therefore go through in this more general case, but we postpone the explicit treatment to \Cref{app:bundles} to keep the notation as simple as possible and the theorems self-contained for the reader who is primarily interested in the scalar case. 
\end{remark}

\subsection{A meromorphic resolvent for billiard systems} \label{zeta_resolvent}

Let $(\Sigma, g, \Omega)$ be a geodesic billiard system such that the associated non-grazing billiard flow $\varphi$ has compact trapped set $K$ and is hyperbolic on $K$ in the sense of Definition \ref{def:hypbilliard}. Before we state and prove our main theorem we first establish the following lemma concerning the generator $\mathbf{X}$ of the smooth model flow $\phi$ of a smooth model $(\mathcal{M},\pi,\phi)$ for $\varphi$ as in \Cref{def:smoothmodel} and the compact set $\mathcal K\subset \M$ defined in \eqref{eq:Ksmooth}:
\begin{lem} \label{lem:strictly_convex}
There exists a compact submanifold with boundary $\mathcal{U}_0$ of $\mathcal{M}$ with manifold interior $\mathring{\mathcal{U}}_0$ such that
\begin{equation} \label{eq:open_neighborhood}
\mathcal{K} \subset \mathring{\mathcal{U}}_0
\end{equation}
and such that there exists a smooth vector field $\mathbf{X}_0$ on $\mathcal{U}_0$ with the following properties:
\begin{enumerate}
	\item[i)] the manifold boundary $\partial \mathcal{U}_0$ of $\mathcal{U}_0$ is strictly convex w.r.t. $\mathbf{X}_0$ in the sense of \eqref{eq:strictconv};
	\item[ii)] $\mathbf{X} - \mathbf{X}_0$ is supported in an arbitrarily small neighborhood of $\partial \mathcal{U}_0$;
	\item[iii)] the trapped set of the flow of $\mathbf{X}_0$ coincides with $\mathcal K$.
\end{enumerate}
\end{lem}

\begin{proof}
First we choose a compact submanifold with boundary $\mathcal N$ of $\mathcal{M}$ with manifold interior $\mathring {\mathcal N}$ such that $\mathcal{K}\subset \mathring{\mathcal N}$. 
The existence of such an $\mathcal N$ is standard in smooth manifold theory, but we provide a proof for convenience: There exists a smooth function $F: \mathcal{M} \rightarrow \R$ such that $F^{-1}((-\infty, c])$ is compact for each $c\in\R$ and the sets $F^{-1}((-\infty, n])$, $n\in\mathbb{N}$ exhaust $\mathcal{M}$ (see e.g. \cite[Prop.~2.28]{Lee.2012}). By compactness of $\mathcal{K}$ there exists some $c_0\in\R$ such that $\mathcal{K}\subset F^{-1}((-\infty, c_0))$. Using Sard's theorem we find some $\varepsilon > 0$ such that $c_0 + \varepsilon$ is a regular value of $F$ and taking $\mathcal N \defgr F^{-1}((-\infty, c_0 + \varepsilon])$ does the trick.
	
Now we can invoke \cite[Prop.~2.2]{Guillarmou.2017}, which in turn builds upon \cite{Conley.1971,Robinson.1980}, to obtain a submanifold with boundary $\mathcal{U}_0\subset \mathring{\mathcal N}$ containing $\mathcal K$ in its manifold interior and \cite[Lemma~2.3]{Guillarmou.2017} to obtain the vector field $\mathbf{X}_0$ with the claimed properties.
\end{proof}
This lemma immediately yields the meromorphic extension of the restricted resolvent $\mathbf{1}_{\mathring{\mathcal{U}}_0} (\mathbf{X}_0 + \lambda)^{-1} \mathbf{1}_{\mathring{\mathcal{U}}_0}$ to the complex plane $\mathbb{C}$ via an application of \cite[Theorems~1, 2]{Dyatlov.2016a} to the open hyperbolic system $(\mathcal{U}_0, \mathbf{X}_0)$. Now our main theorem makes a statement about the generator $\mathbf{X}$ itself instead of the perturbation $\mathbf{X}_0$. To state and prove it we have to introduce some additional auxiliary objects and notations:

In the situation of (the proof of) \Cref{lem:strictly_convex} we may without loss of generality assume an embedding $\mathcal{N}\subset \mathcal{M}'$ into a closed manifold $\mathcal{M}'$ of the same dimension as $\mathcal{M}$. We thus arrive at the following overall situation:
\begin{equation*}
\mathcal{K}\subset \mathring{\mathcal{U}}_0\subset \mathcal{U}_0 \subset \mathring{\mathcal{N}} \subset\mathcal{N} \subset \mathcal{M}' .
\end{equation*}
Furthermore we may extend the vector fields $\mathbf{X}$ and $\mathbf{X}_0$ to $\mathcal{M}'$ arbitrarily and continue to denote such an extension by $\mathbf{X}$ and $\mathbf{X}_0$, respectively. Their respective flows $\phi$ and $\phi^0$ are therefore complete. While we choose the  extension of $\mathbf{X}$ arbitrarily, we choose the  extension of $\mathbf{X}_0$ such that for all $t\geq 0$: If  $p, \phi^0_t(p)\in \mathcal{U}_0$ then $\phi^0_s(p)\in \mathcal{U}_0$ for all  $s\in [0, t]$. This is possible by \Cref{lem:strictly_convex} and \cite[Lemma~1.1]{Dyatlov.2016a}. Analogously to \cite{Guillarmou.2017} we can define the \textit{escape times} from a compact set $\mathcal{U}\subset \mathcal{N}$ as
\begin{equation*}
\begin{split}
\tau^\pm_\mathcal{U}(p) &\defgr \pm \sup\left\{ t\geq 0 \,\big|\, \phi_{\pm s}(p)\in \mathcal{U}\, \forall s\in [0, t] \right\} ,\\
\tau^{0, \pm}_\mathcal{U}(p) &\defgr \pm \sup\left\{ t\geq 0 \,\big|\, \phi^0_{\pm s}(p)\in \mathcal{U}\, \forall s\in [0, t] \right\}, \quad p\in\mathcal{N} ,
\end{split}
\end{equation*}
together with the forward and backward trapped sets
\begin{equation*}
\Gamma_\pm(\mathcal{U}) \defgr \big\{ p\in\mathcal{U} \,\big|\, \tau^\mp_\mathcal{U}(p) = \mp\infty \big\}, \qquad 
\Gamma^0_\pm(\mathcal{U}) \defgr \big\{ p\in\mathcal{U} \,\big|\, \tau^{0, \mp}_\mathcal{U}(p) = \mp\infty \big\}.
\end{equation*}
Next, we define a natural candidate for the inverse of $(\mathbf{X} + \lambda)$ on any open set $\mathcal{O}\subset \mathcal{N}$: For any $f\in \mathrm{C}^\infty_\mathrm{c}(\mathcal{O})$ and $\lambda\in \mathbb{C}$ consider the function on $\mathcal{O}$ formally given by the following integral
\begin{equation*}
\mathbf{R}_\mathcal{O}(\lambda)f(p) \defgr \int_0^{-\tau^-_{\overline{\mathcal{O}}}(p)} \mathrm{e}^{-\lambda t} f\left( \phi_{-t}(p) \right) \mathrm{d}t, \quad p\in \mathcal{O} ~.
\end{equation*}
This definition requires formal justification for two reasons: On the one hand the integral may not converge if $\tau^-_{\overline{\mathcal{O}}}(p) = -\infty$, and on the other hand the regularity properties of $\mathbf{R}_\mathcal{O}(\lambda) f$ are not obvious from the definition. We can overcome these problems if we assume that $\mathcal O$ is chosen such that:
\begin{enumerate}
 \item $\mathbf{X} = \mathbf{X}_0$ on $\overline{\mathcal{O}}$
 \item $\overline{\mathcal{O}}$ is \textit{dynamically convex with respect to $\phi^0$}, i.e.,  for any $t\geq 0$ we have
\begin{equation*}
p,\, \phi_t^0(p)\in\overline{\mathcal{O}} \quad\Longrightarrow\quad \phi^0_s(p)\in\overline{\mathcal{O}}  \;\forall s\in [0, t].
\end{equation*}
\end{enumerate}
Then by the first assumption we get
\begin{equation*}
\mathbf{R}_\mathcal{O}(\lambda)f(p) = \int_0^{-\tau^-_{\overline{\mathcal{O}}}(p)} \mathrm{e}^{-\lambda t} f\left( \phi^0_{-t}(p) \right) \mathrm{d}t, \quad p\in \mathcal{O} .
\end{equation*}
By the second assumption and the fact that $f$ is supported in $\mathcal O$ we can replace the upper integration bound by $\infty$ and obtain for $\mathrm{Re}(\lambda) \gg 0$ that
\begin{equation*}
\mathbf{R}_\mathcal{O}(\lambda) f = \int_0^\infty \mathrm{e}^{-\lambda t} \left( \phi_{-t}^0 \right)^* f \,\mathrm{d}t = \big( \left( \mathbf{X}_0 + \lambda \right)^{-1} \mathbf{1}_\mathcal{O}\big) f ~,
\end{equation*}
which holds as an equality of e.g.\ continuous functions on $\mathcal{O}$ by the integral formula and in turn lets us conclude that
\begin{equation} \label{eq_resolvent_equality}
\mathbf{R}_\mathcal{O}(\lambda) = \mathbf{R}(\lambda)\big|_{\mathrm{C}^\infty_\mathrm{c}(\mathcal{O})}: \mathrm{C}^\infty_\mathrm{c}(\mathcal{O}) \rightarrow \mathcal{D}'(\mathring{\mathcal{U}}_0) \hookrightarrow \mathcal{D}'(\mathcal{O}) ~.
\end{equation}
In particular we have that $\mathbf{R}_\mathcal{O}(\lambda): \mathrm{C}^\infty_\mathrm{c}(\mathcal{O}) \rightarrow \mathrm{C}(\mathcal{O})$ is a holomorphic family of continuous operators on $\{ \mathrm{Re}(\lambda) \gg 0\}$ which satisfies $(\mathbf{X} + \lambda) \mathbf{R}_\mathcal{O}(\lambda) = \mathrm{id}_{\mathrm{C}^\infty_\mathrm{c}(\mathcal{O})}$. With these preliminaries at hand we can now show meromorphic continuation of $\mathbf{R}_\mathcal{O}(\lambda)$ to $\mathbb{C}$ by providing as a particular candidate for $\mathcal{O}$ a concrete dynamically convex neighborhood of the trapped set and applying the results of \cite{Dyatlov.2016a} to $\mathbf{X}_0$. Concretely we prove the following:

\begin{theorem} \label{thm:resolvent}
Let $(\Sigma, g, \Omega)$ be a geodesic billiard system with non-grazing billiard flow $\varphi$, $(\mathcal{M},\pi,\phi)$ a smooth model for $\varphi$ as in \Cref{def:smoothmodel}, and $\mathbf{X}$ the generator of $\phi$. If the trapped set $K$ of $\varphi$ is compact and $\varphi$ is hyperbolic on $K$ in the sense of Definition \ref{def:hypbilliard}, then there exists an arbitrarily small compact $\mathcal{U}\subset \mathcal{M}$ with $\mathcal{K}\subset \mathring{\mathcal{U}}$ such that $\mathbf{R}_{\mathring{\mathcal{U}}}(\lambda)$ extends from $\mathrm{Re}(\lambda) \gg 0$ to $\mathbb{C}$ as a meromorphic family of operators $\mathrm{C}^\infty_\mathrm{c}(\mathring{\mathcal{U}}) \rightarrow \mathcal{D}'(\mathring{\mathcal{U}})$. Its residue at a pole $\lambda_0$ is a finite-rank operator $\Pi_{\lambda_0}: \mathrm{C}^\infty_\mathrm{c}(\mathring{\mathcal U}) \rightarrow \mathcal{D}'(\mathring{\mathcal U})$ satisfying
\begin{equation} \label{eq_wavefront1}
\mathrm{supp}( K_{\Pi_{\lambda_0}} ) \subset \Gamma_+(\mathcal{U})\times \Gamma_-(\mathcal{U}), \quad \mathrm{WF}'\left( \Pi_{\lambda_0} \right) \subset E^*_+\times E^*_- .
\end{equation}
Furthermore, the holomorphic part $\mathbf{R}^H_{\mathring{\mathcal{U}}}(\lambda)$ of $\mathbf{R}_{\mathring{\mathcal{U}}}(\lambda)$ with $\lambda$ in a neighborhood of $\lambda_0$ satisfies the following wavefront estimate:
\begin{equation} \label{eq_wavefront2}
\mathrm{WF}'\big( \mathbf{R}^H_{\mathring{\mathcal{U}}}(\lambda) \big) \subset \Delta( T^* \mathring{\mathcal{U}} ) \cup (E^*_+\times E^*_-)\cup \mathcal{Y}_+ ,
\end{equation}
with $\Delta (T^*\mathring{\mathcal{U}})$ and $\mathcal{Y}_+$ defined after \eqref{eq_regular_wf}. Finally, $\mathcal{U}$ can be chosen to be an isolating block as defined in \cite[Section~1.C.]{Conley.1971} and such that it satisfies the dynamical convexity condition (2) introduced above.
\end{theorem}
\begin{remark}In \Cref{thm:resolvent} \emph{arbitrarily small} means that given any open neighborhood $\mathcal{O}$ of $\mathcal K$ in $\M$ we can choose $\mathcal{U}$ such that $\mathcal{U} \subset \mathcal{O}$.
\end{remark}
\begin{proof}[Proof of \Cref{thm:resolvent}]
By \eqref{eq_resolvent_equality} we only need to construct a dynamically convex neighborhood $\mathcal{U}$ of the trapped set $\mathcal{K}$ satisfying $\mathcal{U}\subset \mathring{\mathcal{U}_0}$. Then we can choose $\mathbf{X}_0$ in \Cref{lem:strictly_convex} in such a way that $\mathbf{X} - \mathbf{X}_0 = 0$ on $\mathcal{U}$ and the stated properties of the residue $\Pi_{\lambda_0}$ transfer from the respective properties of the restricted resolvent $\mathbf{1}_{\mathring{\mathcal{U}}_0} (\mathbf{X}_0 + \lambda)^{-1} \mathbf{1}_{\mathring{\mathcal{U}}_0}$. The wavefront estimate for $\mathbf{R}_H(\lambda)$ follows from \cite[Lemma~3.5]{Dyatlov.2016a}.

To construct $\mathcal{U}$ let $\mathcal{U}_0\subset \mathcal{M}$ be the compact submanifold of \Cref{lem:strictly_convex} and $\mathcal{O}$ any open neighborhood of $\mathcal{K}$ satisfying $\overline{\mathcal{O}}\subset \mathring{\mathcal{U}}_0$. By \cite[Lemma~1.4]{Dyatlov.2016a} there exists $T > 0$ such that
\begin{equation*}
\mathcal{U} \defgr \phi^0_{-T}(\mathcal{U}_0) \cap \mathcal{U}_0\cap \phi^0_T(\mathcal{U}_0) \subset \mathcal{O} .
\end{equation*}
But $\mathcal{U}$ is also dynamically convex with respect to $\phi^0$, because if $p, \phi^0_t(p)\in \mathcal{U}$, $t > 0$, then $\phi^0_{-T}(p), \phi^0_{t + T}(p)\in \mathcal{U}$ by definition of $\mathcal{U}$. But $\mathcal{U}_0$ is already strictly convex with respect to $\phi^0$ which implies $\phi^0_s(p)\in \mathcal{U}_0$ for all $s\in [-T, t + T]$ and therefore $\phi^0_{s'}(p)\in \mathcal{U}$ for all $s'\in [0, t]$.
\end{proof}

Finally, we would like to transfer the results about $\mathbf{X}$ to the geodesic billiard system. To this end we first introduce the escape time of $\varphi$ from a compact set $U\subset M\setminus \partial_gM$ and the forward/backward trapped set of $U$ as follows:
\begin{equation*}
\begin{split}
\tau^\pm_U(x, v) &\defgr \pm \sup\left\{ t\geq 0 \,\big|\, (\pm [0, t])\times \{(x, v)\}\subset D:\, \varphi_{\pm s}(x, v)\in U\, \forall s\in [0, t] \right\} ,\\
\Gamma_\pm(U) &\defgr \left\{ (x, v)\in U \,\big|\, \tau^\mp_U(x, v) = \mp\infty \right\}.
\end{split}
\end{equation*}
Note that given $(x, v)\in U$ the compactness of $U$ implies either $\mathbb{R}\times \{(x, v)\}\subset D$ and $\varphi(\mathbb{R}\times \{(x, v)\}) \subset U$ or the trajectory through $(x, v)$ can be extended to a small neighborhood of $[\tau^-_U(x, v), \tau^+_U(x, v)]$.

We can now obtain the following meromorphic continuation result as a rather immediate corollary of \Cref{thm:resolvent} and the characterization of the billiard operator $\mathbf{P}$ in \eqref{eq:XP}:

\begin{cor} \label{cor:Pmeromorphic}
Let $(\Sigma, g, \Omega)$ be a geodesic billiard system with non-grazing billiard flow $\varphi$ and $\mathbf{P}$ the differential operator of Section \ref{sec:CBillP}. If the trapped set $K$ of $\varphi$ is compact and $\varphi$ is hyperbolic on $K$, then there exists an arbitrarily small compact set $U\subset M\setminus \partial_\mathrm{g} M$ with $K\subset \mathring{U}$ and $U\cap (\partial M\setminus \partial_\mathrm{g} M)$ reflection-symmetric such that the definition
\begin{equation*}
\mathbf{R}_U(\lambda)f(x, v) \defgr \int_0^{-\tau^-_U(x, v)} \mathrm{e}^{-\lambda t} f\left( \varphi_{-t}(x, v) \right) \mathrm{d}t, \qquad f\in\mathrm{C}^\infty_\mathrm{Bill, c}(\mathring{U}) ~\mathrm{and}~ (x, v)\in U,
\end{equation*}
yields a well-defined family of linear maps $\mathbf{R}_U(\lambda): \mathrm{C}^\infty_\mathrm{Bill, c}(\mathring{U}) \rightarrow \mathrm{C}_\mathrm{Bill}(\mathring{U})$ which satisfy $(\mathbf{P} + \lambda) \mathbf{R}_U(\lambda) = \mathrm{id}$ for $\mathrm{Re}(\lambda) \gg 0$ and whose matrix coefficients
\begin{equation*}
\langle \mathbf{R}_U(\lambda) f, g\rangle_{\mathrm{L}^2(\mathrm{d}\mathrm{vol}_g)},\qquad f, g\in \mathrm{C}^\infty_\mathrm{Bill, c}(\mathring{U}),
\end{equation*}
extend from holomorphic functions on $\mathrm{Re}(\lambda) \gg 0$ to meromorphic functions on $\mathbb{C}$ with poles contained in a discrete set of complex numbers that is independent of $f$ and $g$. Here $\mathrm{d}\mathrm{vol}_g$ is the Riemannian volume density associated with the Sasaki metric on $S\Sigma$.
\end{cor}

\begin{proof}
Let $\mathcal{U}$ be chosen according to \Cref{thm:resolvent} and set $U\defgr \pi^{-1}(\mathcal{U})$. First we note that
\begin{equation*}
\mathbf{R}_U(\lambda) f(x, v) = \int_0^{-\tau^-_\mathcal{U}(\pi(x, v))} \mathrm{e}^{-\lambda t} \left( (\pi^*)^{-1} f \right) (\phi_{-t}(\pi(x, v))) \,\mathrm{d} t ,
\end{equation*}
which yields $\mathbf{R}_U(\lambda) = \pi^* \circ \mathbf{R}_\mathcal{U}(\lambda)\circ (\pi^*)^{-1}$. The first claim then follows from \Cref{thm:resolvent} and \Cref{prop:pullbacks} while the second claim follows from
\begin{equation*}
\begin{split}
\langle \mathbf{R}_U(\lambda) f , g \rangle_{\mathrm{L}^2(\mathrm{d}\mathrm{vol}_g)} &= \langle \mathbf{R}_\mathcal{U}(\lambda) \widetilde{f} , \widetilde{g} \rangle_{\mathrm{L}^2(\mu)} ,
\end{split}
\end{equation*}
where $\widetilde{f} \defgr (\pi^*)^{-1} f$, $\widetilde{g} \defgr (\pi^*)^{-1} g$, and $\mu \defgr \mathrm{d}\mathrm{vol}_g\circ \pi^{-1}$ is the pushforward measure of the measure $\mathrm{d}\mathrm{vol}_g$ along $\pi$, in combination with the meromorphic continuation result in \Cref{thm:resolvent}.
\end{proof}

\begin{remark}
\Cref{cor:Pmeromorphic} could be extended to a formulation analogous to \Cref{thm:resolvent} if one developed a theory of ``billiard distributions'' with the compactly supported smooth billiard functions  as test functions, see the discussion at the end of Section \ref{sec:CBillP}. We restrict our attention to the easier result above, though, as the functional analysis otherwise required would distract too much from the main theme of smooth models and their application.
\end{remark}

\begin{remark}[Continuation of Remark \ref{rem:vector1}]\label{rem:vector2}
\Cref{thm:resolvent} and \Cref{cor:Pmeromorphic} generalize immediately to the vector valued setting described in \Cref{app:bundles}: First we construct the smooth model bundle and the smooth substitute $\mathbb{X}$ for a given operator $\widetilde{\mathbb{X}}$. Then we can immediately replace $\widetilde{\mathbb{X}}$ by $\mathbb{X}$ in the proof of \Cref{thm:resolvent} and \cite{Schuette.2021} provides a simple procedure to obtain a vector valued generalization of the auxiliary vector field $\mathbb{X}_0$. The transition from \Cref{thm:resolvent} to \Cref{cor:Pmeromorphic} happens analogously as above but now relies on \Cref{thm:model_bundle}.
\end{remark}

\subsection{Application: meromorphic continuation of $Z_f$} \label{zeta_continuation}

In this section we derive our second main result, the meromorphic continuation of the following formal \emph{weighted zeta function} associated with a \emph{weight} $f: M\setminus\partial_\mathrm{g} M \rightarrow \mathbb{C}$ and a geodesic billiard system $(\Sigma, g, \Omega)$ whose non-grazing billiard flow $\varphi$ is hyperbolic on its trapped set $K$ in the sense of \Cref{def:hypbilliard}:
\begin{equation} \label{eq:zeta_def}
Z_f(\lambda) \defgr \sum_{\gamma} \left( \frac{\exp(-\lambda T_\gamma)}{\vert \mathrm{det}(\mathrm{id} - \mathcal{P}_\gamma) \vert} \int_{\gamma^\#} f \right),
\end{equation}
where $\lambda\in\C$, the sum runs over all closed trajectories of  $\varphi$, $T_\gamma$ is the period of the closed trajectory $\gamma$, $\gamma^\#$ is the corresponding primitive closed trajectory, and $\mathcal{P}_\gamma$ denotes the linearized Poincar\'{e} map associated with $\gamma$, i.e., given any $t\in [0, T_\gamma]\subset \mathbb{R}$ such that $\gamma(t) \in \mathring{M}$ one defines
\begin{equation*}
\mathcal{P}_\gamma \defgr \mathrm{d}\varphi_{T_\gamma}(\gamma(t)): E_s(\gamma(t))\oplus E_u(\gamma(t)) \longrightarrow E_s(\gamma(t))\oplus E_u(\gamma(t)) ,
\end{equation*}
where the determinant in \eqref{eq:zeta_def} is independent of the chosen $t$ and $E_{u/s}$ denotes the hyperbolic splitting as in \Cref{def:hypbilliard}.

\begin{theorem}[Meromorphic continuation of weighted zeta functions] \label{thm:meromorphic1}
	Let $(\Sigma, g, \Omega)$ be a geodesic billiard system such that the trapped set $K$ of the associated non-grazing billiard flow $\varphi$ is compact. Assume also that $\varphi$ is hyperbolic on $K$ in the sense of Definition \ref{def:hypbilliard} and fix a weight $f\in \mathrm{C}^\infty_{\mathrm{Bill}}(M\setminus \partial_\mathrm{g} M)$. Then the following holds:
	\begin{enumerate}
	\item[i)] The weighted zeta function $Z_f$ defined in \eqref{eq:zeta_def} is a well-defined holomorphic function on $\{\mathrm{Re}(\lambda) \gg 0\}$ that extends meromorphically to  $\C$.
	\item[ii)] Considering a smooth model $(\mathcal{M},\pi,\phi)$ for $\varphi$ as in \Cref{def:smoothmodel} and the setting of \Cref{thm:resolvent}, every pole of $Z_f$ is also a pole of the meromorphically extended resolvent $\mathbf{R}_{\mathring{\mathcal{U}}}(\lambda)$. For each pole $\lambda_0$ of $\mathbf{R}_{\mathring{\mathcal{U}}}(\lambda)$ and $k\in\mathbb{N}_0$ one has the following residue formula:
	\begin{equation} \label{eq:residue_trace}
	\mathrm{Res}_{\lambda = \lambda_0} \big( Z_f(\lambda)(\lambda - \lambda_0)^k \big) = \mathrm{tr}^\flat \left( (\mathbf{X} - \lambda_0)^k \Pi_{\lambda_0} \widetilde{f} \right),
	\end{equation}
	where $\widetilde{f} \defgr \left( \pi^* \right)^{-1} f$ and $\Pi_{\lambda_0}: \mathrm{C}^\infty_\mathrm{c}(\mathring{\mathcal U}) \rightarrow \mathcal{D}'(\mathring{\mathcal U})$ is the finite rank operator of \Cref{thm:resolvent}.
	\item[iii)] The poles of $Z_f$ with  $f\equiv 1$ coincide with the poles of the resolvent $\mathbf{R}_{\mathring{\mathcal{U}}}(\lambda)$ of ii). 
	\end{enumerate}
\end{theorem}

\begin{proof}
We begin by observing that the weighted zeta function $Z_f$ for $\varphi$ coincides with the weighted zeta function $Z_{\tilde{f}}$ for the smooth model flow $\phi$ and $\widetilde{f} \defgr (\pi^*)^{-1} f$ by \Cref{prop:trajectory_bijection}, where the latter zeta function is again defined by Formula \eqref{eq:zeta_def}. We can therefore employ \cite{Schuette.2021} where meromorphic continuation of $Z_{\tilde{f}}$ to $\C$ gets proven by writing the weighted zeta function as the flat trace over the resolvent. The result applies if we can verify that $\phi$ is hyperbolic on $\mathcal{K}$ because $Z_{\widetilde{f}}$ coincides with the weighted zeta function for $\mathbf{X}_0$ of Lemma \ref{lem:strictly_convex}. But now \Cref{thm:hypsmooth} implies hyperbolicity therefore proving \emph{i)}.
	
The claims of \emph{ii)} also follow directly from an application of \cite{Schuette.2021} to the zeta function associated with the flow of $\mathbf{X}_0$ and weight $\widetilde{f}$ combined with the observation that $\mathbf{X} = \mathbf{X}_0$ on $\mathrm{ran}(\Pi_{\lambda_0}) \subset \mathcal{D}'(\mathring{\mathcal{U}})$.

The final statement \emph{iii)} follows from the observation that $\mathrm{Res}_{\lambda = \lambda_0} \big( Z_{1}(\lambda) \big) = \mathrm{tr}^\flat(\Pi_{\lambda_0})=\mathrm{rk}(\Pi_{\lambda_0})$, where the second relation is shown in \cite[proof of Thm.\ 4]{Dyatlov.2016a}.
\end{proof}

\begin{cor}\label{cor:Uindp} The set of poles of the meromorphically continued resolvent in \Cref{thm:resolvent} is independent of the choice of the set $\mathcal U$. Consequently, the set of poles of the matrix coefficients in \Cref{cor:Pmeromorphic} is independent of the choice of the set $U$.
\end{cor}
\begin{proof}
This follows immediately from \emph{iii)} in \Cref{thm:meromorphic1} as $Z_1$ is independent of $\mathcal U$.
\end{proof}
\begin{remark}\label{rem:Thm1cor}
 For a billiard in Euclidean space $\R^n$ whose obstacles are strictly convex and fulfill the no-grazing condition (or even the stronger no-eclipse condition, see Section~\ref{intro}) Lemma~\ref{lem:compact_trapped} assures the compactness of $K$ and Remark~\ref{remark:hyp1} the hyperbolicity of $\varphi$ on $K$. As a special case of Corollaries~\ref{cor:Pmeromorphic} and \ref{cor:Uindp} we thus obtain Theorem~\ref{thm:intro} announced in the introduction.
\end{remark}
\begin{remark}
	The \emph{(ordinary) zeta function} associated with the (non-grazing or Euclidean) billiard flow $\varphi$ is defined as
	\begin{equation*}
	\zeta(\lambda) \defgr \prod_{\gamma^\#} \left( 1 - \mathrm{e}^{-\lambda T_{\gamma^\#}} \right),
	\end{equation*}
	with the product running over all closed primitive trajectories of $\varphi$. Note that, under the assumption that the stable/unstable foliations $\mathcal{E}_{s/u}$ constructed in the proof of \Cref{thm:meromorphic1} are \emph{orientable}, $\zeta(\lambda)$ also continues meromorphically to $\C$ (provided the no-grazing condition holds): We can employ the same proof as for \Cref{thm:meromorphic1} but instead of using the continuation result of \cite{Schuette.2021} we directly invoke \cite[Thm.~3]{Dyatlov.2016a}. The orientability condition can actually be dropped, see \cite{Shen.2020}. This generalizes a result of Morita \cite{Morita} who proved meromorphic continuation to a right halfplane.
\end{remark}

\begin{remark}[Conclusion of Remarks \ref{rem:vector1} and \ref{rem:vector2}]\label{rem:vector3}
	Using the vector-valued versions of \Cref{thm:resolvent} and \Cref{cor:Pmeromorphic} we obtain the meromorphic continuation to $\C$ of weighted zeta functions associated with vector-valued data as in \cite[Thm.~4]{Dyatlov.2016a} and \cite{Schuette.2021}:
	\begin{equation*}
	Z_f^{\widetilde{\mathbb{X}}, \kappa}(\lambda) \defgr \sum_{\gamma} \left( \frac{\exp(-\lambda T_\gamma) \mathrm{tr}(\widetilde{\alpha}_\gamma)}{\vert \mathrm{det}(\mathrm{id} - \mathcal{P}_\gamma) \vert} \int_{\gamma^\#} f \right),
	\end{equation*}
	where $\mathrm{tr}(\widetilde{\alpha}_\gamma) = \mathrm{tr}(\widetilde{\alpha}_{\gamma(t), T_\gamma})$ denotes the trace of the billiard parallel transport along a closed trajectory $\gamma$, which is independent of the chosen base point $\gamma(t)$. Note that $\widetilde{\alpha}$ depends crucially on both the operator $\widetilde{\mathbb{X}}$ as well as the boundary gluing map $\kappa$ as indicated by the superscripts.
	
	This proves very useful in practice, as for example using a line bundle whose restriction to $\mathring M$ is trivial and which possesses a twist at each obstacle boundary, one can treat dynamical zeta functions involving a sign depending on the number of reflections occurring in the closed trajectories. Similarly one can use the vector-valued setting to treat zeta functions with the modified denominator $\sqrt{\mathrm{id} - \mathcal{P}_\gamma}$. Such zeta functions are of particular interest as they exhibit connections with quantum resonances for obstacle scattering. They have been treated in the literature before, see e.g. \cite{Petkov2008,ChaubetPetkovBilliards}. 
\end{remark}

\begin{remark}
	Finally, we would like to point out that in principle the constructions and results of this paper also apply (with the assumption of a compact trapped set on which the non-grazing billiard flow is hyperbolic) in a more general abstract situation  where instead of the geodesic flow $\varphi^g$ on the unit tangent bundle  $S\Sigma\supset M$ and the Riemannian tangential reflection $R:\partial M\to \partial M$ one considers some arbitrary smooth flow $\varphi^0$ on an arbitrary smooth manifold $M$ with boundary $\partial M$ together with an abstract reflection map $R:\partial M\to \partial M$ that satisfies certain natural axioms such as compatibility with $\varphi^0$ which allow to define the non-grazing billiard flow $\varphi$ associated with $\varphi^0$ and $R$ in the same way as above. However, as we do not know an interesting example of such a setup that is not equivalent to a Riemannian billiard, we do not pursue the more abstract approach here. 
\end{remark}

\appendix

\section{Technical proofs}\label{sec:proofs}
Here we provide some rather technical proofs to avoid interruptions of the text flow.

\subsection{Proof of Lemma \ref{lem:convexity}}\label{sec:proof1}
First, we introduce a smooth involution $\widetilde{R}: \mathbb{R}\times \partial M\setminus \partial_\mathrm{g} M \to \mathbb{R}\times \partial M\setminus \partial_\mathrm{g} M$ by putting $\widetilde{R}(t,x,v) \defgr (t,x,v')$. Due to the transversality between the vectors in $\partial M\setminus \partial_\mathrm{g}M$ and $T (\partial\Omega)$ the inverse function theorem implies that the geodesic flow
\[
\varphi^g:\R\times (\partial M\setminus \partial_\mathrm{g}M)\to S\Sigma 
\]
restricts to a diffeomorphism from an open neighborhood $N$ of $\{0\}\times (\partial M\setminus \partial_\mathrm{g}M)$ in $\R\times (\partial M\setminus \partial_\mathrm{g}M)$ onto an open neighborhood of  $\partial M\setminus \partial_\mathrm{g}M$ in $S\Sigma $. By replacing $N$ with its intersection with the open set $(\varphi^{g})^{-1}(S\Sigma\setminus \partial_\mathrm{g}M)$ we achieve that $\varphi^g(N)$ is disjoint from  $\partial_\mathrm{g}M$. By shrinking $N$ further we achieve that $N$ is $\widetilde{R}$-invariant and that the intersection of $N$ with each connected component of $\R\times (\partial M\setminus \partial_\mathrm{g}M)$ is connected. Then a point $(t,x,v)\in \R\times (\partial M\setminus \partial_\mathrm{g}M)$ with $t\geq 0$ lies in $N$ iff all points $(s,x,v)$ and $(s,x,v')$ with $s\in [0,t]$ lie in $N$, and similarly for $t<0$.

Now for every $(x,v)\in  \partial_\mathrm{in}M$ and small $t > 0$ we have that $\varphi^g_t(x,v)\not \in M$, while for every $(x,v)\in \partial_\mathrm{out} M$ and small $t<0$ we have that $\varphi^g_t(x,v)\not \in M$. This shows that 
\bq
\begin{aligned}
\{(t,x,v)\in N\,|\, \varphi^g(t,x,v)\in M\}&=N^+_\mathrm{out}\cup N^-_\mathrm{in},\\
 \{(t,x,v)\in N\,|\, \varphi^g(t,x,v')\in M\}&=N^-_\mathrm{out}\cup N^+_\mathrm{in}.\label{eq:Ninoutpm}
\end{aligned}
\eq
Recalling \eqref{eq:tpm} and taking into account that $N\subset\R\times (\partial M\setminus \partial_\mathrm{g}M)$ is open and the projection $\R\times (\partial M\setminus \partial_\mathrm{g}M) \to \R$ is an open map, \eqref{eq:Ninoutpm} shows that every point $(t,x,v)\in N$ satisfies $t_-(x,v)<t<t_+(x,v)$. This proves the inclusion $N\subset D$ and recalling \eqref{eq:localflow2} we get \eqref{eq:varphivarphig}.

Let $(t_1,x_1,v_1),(t_2,x_2,v_2)\in N_\mathrm{in}$ satisfy $\varphi_{t_1}(x_1,v_1)=\varphi_{t_2}(x_2,v_2)$. Then by  \eqref{eq:Ninoutpm} the signs of $t_1$ and $t_2$ must coincide and by \eqref{eq:varphivarphig} we have
\[
\begin{cases}\varphi^g_{t_1}(x_1,v_1)=\varphi^g_{t_2}(x_2,v_2),\qquad &t_1\leq 0,t_2\leq 0,\\
\varphi^g_{t_1}(x_1,v_1')=\varphi^g_{t_2}(x_2,v_2'), &t_1> 0,t_2> 0.\end{cases}
\]
We conclude that $(x_1,v_1) = (x_2,v_2)$ because $\varphi^g$ and $\varphi^g\circ \widetilde{R}$ are injective on $N_\mathrm{in}$. This proves that $\varphi$ is injective on $N_\mathrm{in}$. We argue analogously for $N_\mathrm{out}$.  

That $\varphi:N\to M$ is open follows once we know that $\varphi(N_\mathrm{in/out})$ are open in $M$ and $\varphi|_{N_\mathrm{in/out}}^{-1}:\varphi(N_\mathrm{in/out})\to N_\mathrm{in/out}$ are smooth, because $N=N_\mathrm{in}\sqcup N_\mathrm{out}$. By \eqref{eq:varphivarphig}, combined with the fact that $\widetilde{R}$ interchanges $N_\mathrm{in}$ with $N_\mathrm{out}$ as well as $N^\pm_\mathrm{in}$ with $N^\pm_\mathrm{out}$, one has
\bq
\begin{aligned}
\varphi(N^-_\mathrm{in}) &= M\cap \varphi^g(N_\mathrm{in}), &  \varphi(N^+_\mathrm{out}) &= M\cap \varphi^g(N_\mathrm{out}),\\
\varphi(N_\mathrm{in}\setminus N^-_\mathrm{in}) &= M\cap \varphi^g(N_\mathrm{out}\setminus N^-_\mathrm{out}),\qquad&  \varphi(N_\mathrm{out}\setminus N^+_\mathrm{out}) &= M\cap \varphi^g(N_\mathrm{in}\setminus N^+_\mathrm{in}).\label{eq:foursets}
\end{aligned}
\eq 
This shows that the sets on the left-hand sides of the equalities are open in $M$ because they are intersections with $M$ of open subsets of $S\Sigma$, $\varphi^g:N\to S\Sigma$ being an open map. As we already know that $\varphi$ is injective on $N_\mathrm{in/out}$, we now see that each of the sets $\varphi(N_\mathrm{in/out})\subset M$ decomposes into a disjoint union of two open subsets of $M$ as follows:
\[
\varphi(N_\mathrm{in})= \varphi(N^-_\mathrm{in})\sqcup \varphi(N_\mathrm{in}\setminus N^-_\mathrm{in}),\qquad  \varphi(N_\mathrm{out})= \varphi(N^+_\mathrm{out})\sqcup \varphi(N_\mathrm{out}\setminus N^+_\mathrm{out}).
\] 
To prove that $\varphi|_{N_\mathrm{in}}^{-1}$ is a smooth we use that $(\varphi^g)^{-1}:\varphi^g(N)\to N$ and $(\varphi^g\circ \widetilde{R})^{-1}:\varphi^g(N)\to N$ are smooth: By \eqref{eq:varphivarphig} the map $\varphi|_{N_\mathrm{in}}^{-1}$ coincides with $(\varphi^g)^{-1}$ on $\varphi(N^-_\mathrm{in})$ and with $(\varphi^g\circ \widetilde{R})^{-1}$ on $\varphi(N_\mathrm{in}\setminus N^-_\mathrm{in})$. The latter two disjoint open sets cover $ \varphi(N_\mathrm{in})$ proving that $\varphi|_{N_\mathrm{in}}^{-1}$ is smooth. We argue analogously for $\varphi|_{N_\mathrm{out}}^{-1}$.

Finally, we note that $\varphi(N)$ is disjoint from  $\partial_\mathrm{g}M$ since $\varphi^g(N)$ is disjoint from  $\partial_\mathrm{g}M$. \qed

\subsection{Proof of Lemma \ref{lem:D}}\label{sec:proof2}
Given $(t,x, v)\in \varphi^{-1}(O)$ with $t\geq 0$, suppose first that the compact trajectory segment $\varphi([0,t]\times \{(x,v)\})$ lies in $\mathring M$, so that it agrees with a trajectory segment of the geodesic flow: $\varphi([0,t]\times \{(x,v)\})=\varphi^g([0,t]\times \{(x,v)\})\subset\mathring M$. Then it follows from the continuity of $\varphi^g$ that there is an $\eps>0$  and an open set $U\subset \mathring M$ containing $(x,v)$ such that $\varphi^g([t-2\eps,t+2\eps]\times U)\subset \mathring M\cap O$. Now $t_-(\tilde x,\tilde v)<  t-\eps$, $t_+(\tilde x,\tilde v)>t+\eps$ for all $(\tilde x,\tilde v)\in U$. It follows that $(t-\eps,t+\eps)\times U\subset \varphi^{-1}(O)$, so that  $(t,x,v)$ is an interior point of $\varphi^{-1}(O)$.  

As a second case, suppose that $t>0$, $\varphi([0,t)\times \{(x,v)\})\subset\mathring M$, and $(x_0,v_0):=\varphi_t(x,v)\in (\partial M\setminus \partial_\mathrm{g} M)\cap O$. Then $\varphi_t(x,v)\in \partial_\mathrm{in} M\cap O$ because the trajectory is incoming.  Recalling the definition of $\varphi$ departing from \eqref{eq:localflow1}, the half-open trajectory segment $\varphi([0,t)\times \{(x,v)\})$ agrees with a  trajectory segment of the geodesic flow:  $\varphi([0,t)\times \{(x,v)\})=\varphi^g([0,t)\times \{(x,v)\})\subset\mathring M$. Let $N_\mathrm{in}\subset\R\times \partial_\mathrm{in} M$ be an open set as in Lemma \ref{lem:convexity}. Then, since $N_\mathrm{in}$ is a neighborhood of $\{0\}\times \partial_\mathrm{in} M$ in $\R\times \partial_\mathrm{in} M$ and by the continuity of $\varphi^g$ we can choose a small open subset $S_\mathrm{in}\subset \partial_\mathrm{in} M$ and a small $\delta>0$ such that $S_\mathrm{in}$ contains $(x_0,v_0)=\varphi_t(x,v)$, $\overline S_\mathrm{in}\subset \partial_\mathrm{in} M\cap O$ is compact, $(-\delta,\delta)\times \overline S_\mathrm{in}\subset N_\mathrm{in}$, and $\varphi^g((-\delta, 0] \times S_\mathrm{in})\subset  O$. Then \eqref{eq:varphivarphig} gives us $\varphi((-\delta,0]\times S_\mathrm{in})=\varphi^g((-\delta,0]\times S_\mathrm{in})\subset  O$. In addition we introduce the open set $S_\mathrm{out}:=\{(x,v')\,|\, (x,v)\in S_\mathrm{in}\}\subset \partial_\mathrm{out} M$ which contains $(x_0,v_0')$. Then by the reflection-symmetry of $(\partial M\setminus \partial_\mathrm{g} M)\cap O$  we have $\overline{S}_\mathrm{out}\subset\partial_\mathrm{out} M\cap O$. Using again the continuity of $\varphi^g$ and \eqref{eq:varphivarphig} we can achieve, shrinking $\delta$ if necessary, that $\varphi^g([0,\delta)\times S_\mathrm{out})=\varphi([0,\delta)\times S_\mathrm{out})\subset  O$. 

Now, the continuity of  $\varphi^g$ and the fact that $\varphi((-\delta,\delta)\times S_\mathrm{in})\cap \mathring M$ is open in $S\Sigma$  by Lemma \ref{lem:convexity}  \emph{iii)} imply that there is a small $\eps\in (0,t)$ and a small open set $U\subset \mathring M$ containing $(x,v)$ such that $\varphi^g([0,t-\eps]\times U)\subset \mathring M$ and $\varphi_{t-\eps}^g(U)\subset  \varphi((-\delta,0)\times S_\mathrm{in})\cap O$. Then by \eqref{eq:localflow1} we have $[0,t-\eps]\times U\subset D$ and $\varphi([0,t-\eps]\times U)=\varphi^g([0,t-\eps]\times U)$.  See Figure \ref{fig32} for an illustration.

\begin{figure}[h]
\centering
\includegraphics[width=\textwidth, trim={0cm 0cm 0cm 0cm}, clip]{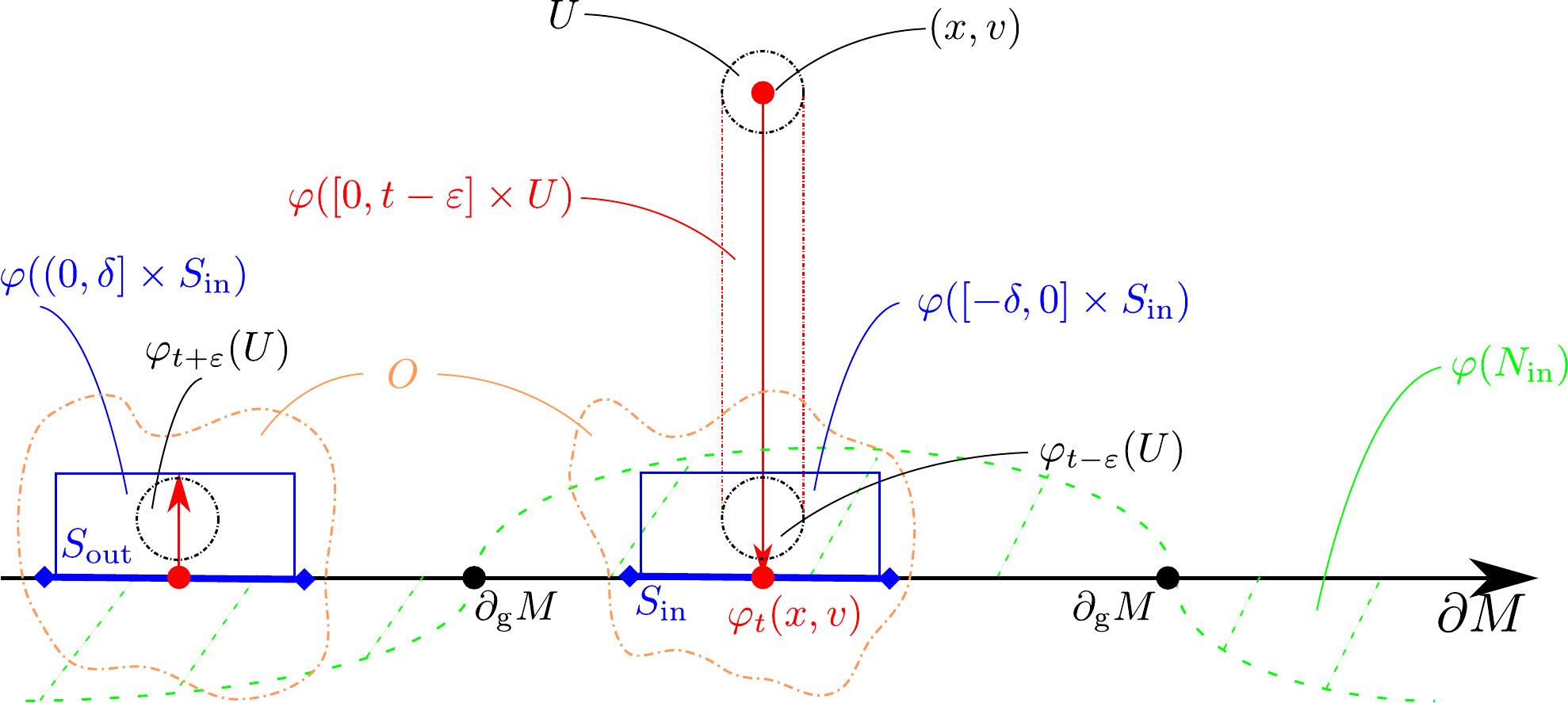}
\caption{Illustration of the argument in the proof of \Cref{lem:D}. The transversality properties of $\varphi$ established in Lemma \ref{lem:convexity} allow us to extend the trajectories of all points in $U$ to the time interval $[0,t+\eps)$.}
\label{fig32}
\end{figure}

Lemma \ref{lem:convexity} \emph{iv)} and the flow property $\varphi_s(\varphi_{s'}(\tilde x,\tilde v))=\varphi_{s+s'}(\tilde x,\tilde v)$ now imply the inclusion $(t-\eps,t+\eps)\times U\subset \varphi^{-1}(O)$, so that $(t,x,v)$ is an interior point of $\varphi^{-1}(O)$.  

If the trajectory segment $\varphi([0,t]\times \{(x,v)\})$ does not lie entirely in $\mathring M$, we can cut it into finitely many segments each of which lies either entirely in $\mathring M$ or intersects $\partial M\setminus \partial_\mathrm{g} M$ precisely in one of its endpoints. It then suffices to  repeat the arguments above inductively finitely many times to show that $(t,x, v)$ is an interior point of $\varphi^{-1}(O)$. The case $t<0$ is treated analogously, finishing the proof that $\varphi^{-1}(O)$ is open.

Finally, \eqref{eq:Dinv} follows immediately from a combination of \eqref{eq:invariance} with the presupposed reflection-symmetry of $O\cap (\partial M\setminus \partial_\mathrm{g} M)$. This completes the proof. \qed

\subsection{Proof of Proposition \ref{prop:uniqueness}}\label{sec:proof3}
The proof of Proposition \ref{prop:uniqueness} relies on the following auxiliary result which describes some general technical properties of smooth models for the non-grazing  billiard flow:
\begin{lem}\label{lem:G} Let $(\M,\pi,\phi)$ be a smooth model for $\varphi$. Then $\pi$ is a proper map,  one has
 \[
\pi(\partial_\mathrm{in}M)=\pi(\partial_\mathrm{out}M) \defgl \G,\qquad\qquad \M=\pi(\mathring M)\sqcup \G,
\]
the set $\G$ is a codimension $1$ submanifold of $\M$, and the maps $\pi|_{\partial_\mathrm{in/out}M}:\partial_\mathrm{in/out}M\to \G$ are diffeomorphisms. In particular, for every set $\mathcal A\subset \M$ the  set $\pi^{-1}(\mathcal A)\cap (\partial M\setminus \partial_\mathrm{g} M)$ is reflection-symmetric. Furthermore, the flow $\phi$ is transversal to  $\G$.
\end{lem}
\begin{proof}From \eqref{eq:flowcompat}, \eqref{eq:localflow2}, and the continuity of $\pi$ and $\phi$ it follows that
\bq
\pi(x,v)=\pi(x,v')\qquad \forall\; (x,v)\in \partial M\setminus \partial_\mathrm{g}M.\label{eq:pirefl}
\eq
Since the tangential reflection $(x,v)\mapsto (x,v')$ interchanges $\partial_\mathrm{in}M$ and $\partial_\mathrm{out}M$, this shows that $\pi(\partial_\mathrm{in}M)=\pi(\partial_\mathrm{out}M)$. From the fact that $M\setminus \partial_\mathrm{g}M=\mathring M\sqcup \partial_\mathrm{in}M\sqcup \partial_\mathrm{out}M$ and the surjectivity of $\pi$ it follows that $\M=\pi(\mathring M)\cup \G$. Suppose that there is a point $p\in \G\cap \pi(\mathring M)$. Then there are points $\tilde p\in \mathring M$ and $\tilde p'\in \partial M\setminus \partial_\mathrm{g} M$ such that $\pi(\tilde p)=\pi(\tilde p')=p$. Since each trajectory of the  billiard flow $\varphi$ intersects $\partial M\setminus \partial_\mathrm{g} M$ only at a discrete set of time parameters, we can find an $\eps>0$ such that $\varphi_\eps(\tilde p),\varphi_\eps(\tilde p')\in\mathring M$. Then \eqref{eq:flowcompat} implies
\[
\pi|_{\mathring M}(\varphi_\eps(\tilde p))=\phi_\eps(p)=\pi|_{\mathring M}(\varphi_\eps(\tilde p')),
\]
and the injectivity of  $\pi|_{\mathring M}$ and $\varphi_\eps$ yields $\tilde p=\tilde p'$, contradicting the fact that $\mathring M$ and $\partial M\setminus \partial_\mathrm{g} M$ intersect non-trivially. This proves that $\G\cap \pi(\mathring M)=\emptyset$  hence $\M=\pi(\mathring M)\sqcup \G$.  

To prove that $\G$ is a smooth submanifold of $\M$, let $N_\mathrm{in}\subset\R\times \partial_\mathrm{in} M$ be an open set as in Lemma \ref{lem:convexity} and put $\mathcal N:=(\mathrm{id}_\R\times \pi)(N_\mathrm{in})\subset \R\times \G$. Then $\{0\}\times \G\subset \mathcal N\subset \mathcal D$  by Lemma \ref{lem:convexity} \emph{i)}.  Let $(x,v)\in \partial_\mathrm{in}M$ and choose a small open set $S\subset \partial_\mathrm{in}M$ containing $(x,v)$ and a small $\eps>0$ such that $(-\eps,\eps)\times S\subset N_\mathrm{in}$. This is possible because $N_\mathrm{in}$ is a neighborhood of $\{0\}\times \partial_\mathrm{in} M$ in $\R\times \partial_\mathrm{in} M$ by Lemma \ref{lem:convexity} \emph{i)}. Consider the set $N^-_\mathrm{in}$ from \eqref{eq:Ndecomp} and recall from \eqref{eq:varphivarphig} that $\varphi|_{N^-_\mathrm{in}}=\varphi^g|_{N^-_\mathrm{in}}$. Now, since $\pi|_{\mathring M}$ is a diffeomorphism and $\varphi^g$ is an open map, we get that $\pi|_{\mathring M}(\varphi^g((-3\eps/4,-\eps/4)\times S))$ is an open subset of $\M$. Consequently, the set
\[
\mathcal U:=\phi_{\eps/2}(\pi|_{\mathring M}(\varphi^g((-3\eps/4,-\eps/4)\times S)))\subset \M
\]
is open in $\M$ because $\phi_{\eps/2}$ is an open map. Moreover, using \eqref{eq:varphivarphig} and Lemma \ref{lem:convexity} \emph{v)}, we see that $\phi_{\eps/2}\circ \pi|_{\mathring M}\circ \varphi^g|_{(-3\eps/4,-\eps/4)\times S}$ is a diffeomorphism from $(-3\eps/4,-\eps/4)\times S$ onto $\mathcal U$. 

  However, thanks to  \eqref{eq:flowcompat},  \eqref{eq:varphivarphig}, and the flow property of $\varphi$ we have
\[
\mathcal U=(\pi\circ\varphi)((-\eps/4,\eps/4)\times S),\qquad \phi_{\eps/2}\circ \pi|_{\mathring M}\circ \varphi^g|_{(-3\eps/4,-\eps/4)\times S}=\pi\circ \varphi|_{(-\eps/4,\eps/4)\times S},
\]
which shows that $\G\subset \mathcal U$ and that $\pi\circ \varphi|_{(-\eps/4,\eps/4)\times S}:(-\eps/4,\eps/4)\times S\to \mathcal U$ is a diffeomorphism mapping $\{0\}\times S$ onto $\mathcal U\cap\G \defgl \mathcal S$. In particular, the identity $\pi|_S=\phi_{\eps/2}\circ\pi|_{\mathring M}\circ\varphi^g_{-\eps/2}|_S$ shows that $\pi|_S:S\to \mathcal S$ is a diffeomorphism.  Applying \eqref{eq:flowcompat} again we find that $\phi|_{(-\eps/4,\eps/4)\times \mathcal S}:(-\eps/4,\eps/4)\times \mathcal S\to \mathcal U$ is a diffeomorphism. This shows that $\phi$ is transversal to $\G$ at each point in $\mathcal S$.

 Since $(x,v)\in\partial_\mathrm{in}M$ was arbitrary and we know that $\pi(\partial_\mathrm{in}M)=\G$,  we have proved that $\G$ is a smooth submanifold of $\M$ of the same dimension as $\partial_\mathrm{in}M$, i.e., of codimension $1$, that $\pi|_{\partial_\mathrm{in}M}:\partial_\mathrm{in}M \to \G$ is a local diffeomorphism, and that $\phi$ is transversal to $\G$. 

To prove that $\pi|_{\partial_\mathrm{in}M}$ is injective it suffices to observe that given $(x_1,v_1),(x_2,v_2)\in \partial_\mathrm{in}M$ we can repeat the above argument with an $S\subset \partial_\mathrm{in}M$ containing both $(x_1,v_1)$ and $(x_2,v_2)$ and some small enough $\eps>0$ depending on $(x_1,v_1)$ and $(x_2,v_2)$; then the same trick of writing $\pi|_S=\phi_{\eps/2}\circ\pi|_{\mathring M}\circ\varphi^g_{-\eps/2}|_S$ shows that the assumption $\pi(x_1,v_1)=\pi(x_2,v_2)$ implies $(x_1,v_1)=(x_2,v_2)$  by injectivity of $\phi_{\eps/2}$, $\pi|_{\mathring M}$, and $\varphi^g_{-\eps/2}$.

The restriction $\pi|_{\partial_\mathrm{out}M}$ is treated analogously, finishing the proof that the maps $\pi|_{\partial_\mathrm{in/out}M}:\partial_\mathrm{in/out}M\to \G$ are diffeomorphisms. 

Finally, that $\pi$ is proper follows from the continuity of $\pi$  and the observation that by the above  the inverse image $\pi^{-1}(p)$ of any point $p\in \mathcal M$ contains at most $2$ points.
\end{proof}
We are now in a position to prove Proposition \ref{prop:uniqueness}:
\begin{proof}[Proof of \Cref{prop:uniqueness}]
Using Lemma \ref{lem:G}, we define $F:\M\to \M'$ by
\[
F(p):=\begin{cases}
(\pi'\circ \pi|_{\mathring M}^{-1})(p),& p\in \pi(\mathring M),\\
(\pi'\circ\pi|_{\partial_\mathrm{in}M}^{-1})(p), & p\in \pi(\partial_\mathrm{in}M)=\G.
\end{cases}
\]
Then $F$ satisfies by construction the relation $F\circ \pi=\pi'$ and $F$ is bijective with inverse
\[
F^{-1}(p')=\begin{cases}
(\pi\circ \pi'|_{\mathring M}^{-1})(p'),& p'\in \pi'(\mathring M),\\
(\pi\circ \pi'|_{\partial_\mathrm{in}M}^{-1})(p'), & p'\in \pi'(\partial_\mathrm{in}M)=\G'.
\end{cases}
\]
From \eqref{eq:flowcompat} and the relations $\mathcal{D} = (\mathrm{id}_\R\times \pi)(D)$, $\mathcal{D}' = (\mathrm{id}_\R\times \pi')(D)$ we get  $(\mathrm{id}_\R\times F)(\mathcal D)=\mathcal D'$ and $F\circ \phi=\phi'\circ (\mathrm{id}_\R\times F)|_{\mathcal D}$. 

It remains to prove that $F$ and $F^{-1}$ are smooth. By Definition \ref{def:smoothmodel} and Lemma \ref{lem:G} the maps $F:\pi(\mathring M)\to\pi'(\mathring M)$ and  $F:\G\to \G'$ are diffeomorphisms, in particular $F$ and $F^{-1}$ are smooth on the open sets $\pi(\mathring M)\subset \M$ and $\pi'(\mathring M)\subset \M'$, respectively. To prove smoothness of $F$ and $F^{-1}$ near $\G$ and $\G'$, we note that since $\mathcal D$ and $\mathcal D'$ are open and $\phi$, $\phi'$ are flows transversal to $\G$ and $\G'$, respectively by Lemma \ref{lem:G}, the inverse function theorem implies that there is an open set $U\subset \M$ containing $\G$, an open set $V\subset (\R\times \G)\cap \mathcal{D}$ containing $\{0\}\times \G$, an open set $U'\subset \M'$ containing $\G'$, and an open set $V'\subset (\R\times \G')\cap \mathcal{D}'$ containing $\{0\}\times \G'$ such that
\[
\phi: V\to U,\qquad \phi':V'\to U'
\]
are diffeomorphisms. In fact, since $F\circ \phi=\phi'\circ (\mathrm{id}_\R\times F)|_{\mathcal D}$, we can achieve
\[
V'=(\mathrm{id}_\R\times F|_\G)(V)
\]
by shrinking $V$ or $V'$. To show that $F$ is smooth on $U$ and $F^{-1}$ is smooth on $U'$ it now suffices to prove that
\[
(\mathrm{id}_\R\times F|_\G^{-1})\circ\phi'|_{V'}^{-1}\circ F\circ \phi: V\to V,\qquad (\mathrm{id}_\R\times F)\circ \phi|_{V'}^{-1}\circ F^{-1}\circ \phi': V'\to V'
\]
are smooth maps. However, the latter are nothing but the identity maps as one sees by employing again the relation $F\circ \phi=\phi'\circ (\mathrm{id}_\R\times F)|_{\mathcal D}$.

Finally, the uniqueness of $F$ follows from the fact that by Lemma \ref{lem:G} the complement $\M\setminus \G$ is dense in $\M$ and $F$ is uniquely determined by $\pi$ and $\pi'$ on $\M\setminus \G$.
\end{proof}
\subsection{Proof of Proposition \ref{prop:gdependence}}\label{sec:proof4}
Consider the gluing regions  $\mathcal G_g=\pi_g(\partial M\setminus \partial_\mathrm{g}M)\subset \M_g$ and $\mathcal G_{g'}=\pi_{g'}(\partial M\setminus \partial_\mathrm{g}M)\subset \M_{g'}$, respectively. Then by the fact that $\pi_g|_{\mathring M}$ and $\pi_{g'}|_{\mathring M}$ are diffeomorphisms onto $\M_g \setminus \mathcal G_g$ and  $\M_{g'} \setminus \mathcal G_{g'}$, respectively, we immediately get the built-in diffeomorphism
\bq
\pi_{g'}|_{\mathring M}\circ\pi_g|_{\mathring M}^{-1}: \M_g \setminus \mathcal G_g\stackrel{\cong}{\longrightarrow}  \M_{g'} \setminus \mathcal G_{g'}\label{eq:diffeoint}
\eq
that makes the analogue of the diagram \eqref{eq:diffpi}, with $M\setminus \partial_\mathrm{g}M$ replaced by $\mathring M$, commute. Moreover, we see that any diffeomorphism  $\M_g\cong \M_{g'}$ making \eqref{eq:diffpi} commute coincides with \eqref{eq:diffeoint} on the dense set $\M_g \setminus \mathcal G_g$, so that it is uniquely determined by this property. 

In order to extend  \eqref{eq:diffeoint} to a (necessarily unique) diffeomorphism $\M_g\cong \M_{g'}$, let $N_g,N_{g'}\subset \R\times (\partial M\setminus \partial_\mathrm{g}M)$ be two sets as in \Cref{lem:convexity}, applied separately for $g$ and $g'$, respectively, and consider the intersection $N_\mathrm{in}:=(N_g)_\mathrm{in}\cap (N_{g'})_\mathrm{in}=N\cap (\R\times \partial_\mathrm{in}M)$, where $N:=N_g\cap N_{g'}$. Then the set $O:=\varphi_g(N_\mathrm{in})\cap \varphi_{g'}(N_\mathrm{in})$ is an open neighborhood of $\partial M\setminus \partial_\mathrm{g}M$ in $M\setminus \partial_\mathrm{g}M$ by  \Cref{lem:convexity}.  In particular, $O\cap(\partial M\setminus \partial_\mathrm{g}M)=\partial M\setminus \partial_\mathrm{g}M$ is invariant under tangential reflection with respect to $g$ and $g'$. Thus,  by \Cref{lem:D} (also applied separately for $g$ and $g'$), the sets
\[
N_{g,g'}:=\varphi_g^{-1}(O)\cap N_\mathrm{in},\qquad N_{g',g}:=\varphi_{g'}^{-1}(O)\cap N_\mathrm{in}
\]
are open in $\R\times \partial_\mathrm{in}M$, and by definition of the quotient topologies on $\M_g$ and $\M_{g'}$ the identity map $O\to O$ descends to a homeomorphism
\bq
\mathcal O_g\cong \mathcal O_{g'} \label{eq:diffG3}
\eq	
 between the open neighborhood $\mathcal O_g:=\pi_g(O)$ of $\mathcal G_g$ in $\M_{g}$ and the open neighborhood $\mathcal O_{g'}:=\pi_{g'}(O)$ of $\mathcal G_{g'}$ in $\M_{g'}$. Moreover, by definition of the quotient maps $\pi_g$ and $\pi_{g'}$, the diffeomorphism \eqref{eq:diffeoint} and the homeomorphism \eqref{eq:diffG3} agree on   $\mathcal O_g\cap (\M_g \setminus \mathcal G_g)$, so that they glue to a global homeomorphism $\M_g\cong \M_{g'}$. It remains to prove that the latter is a diffeomorphism, which reduces to proving that the map \eqref{eq:diffG3} and its inverse are smooth. By definition of the smooth structures on $\M$ and  $\M_{g'}$, this reduces to checking that 
 \bq\begin{split}
 \varphi_g|_{N_{g,g'}}^{-1}\circ\varphi_{g'}|_{N_{g',g}}:N_{g',g}&\to N_{g,g'}, \\
 \varphi_{g'}|_{N_{g',g}}^{-1}\circ \varphi_g|_{N_{g,g'}}: N_{g,g'}&\to N_{g',g}\label{eq:nggN}\end{split}
 \eq
are smooth maps. By \eqref{eq:varphivarphig}  the maps \eqref{eq:nggN} can be expressed in terms of the geodesic flows $\varphi^g$, $\varphi^{g'}$ and the reflection maps $R_{g},R_{g'}: \R\times (\partial M\setminus \partial_\mathrm{g}M)\to \R\times (\partial M\setminus \partial_\mathrm{g}M)$, $(t,x,v)\mapsto (t,x,R_{g(x)}v)$, $(t,x,v)\mapsto (t,x,R_{g'(x)}v)$, by\\
\[
(\varphi_g|_{N_{g,g'}}^{-1}\circ\varphi_{g'}|_{N_{g',g}})(t,x,v)=\begin{cases}(\varphi^g|_{N_{g,g'}}^{-1}\circ\varphi^{g'}|_{N_{g',g}})(t,x,v), & t\leq 0,\\
(R_{g}\circ\varphi^g|_{R_g(N_{g,g'})}^{-1}\circ\varphi^{g'}|_{R_{g'}(N_{g',g})}\circ R_{g'})(t,x,v), & t>0,\end{cases}
\]
\[
(\varphi_{g'}|_{N_{g',g}}^{-1}\circ \varphi_g|_{N_{g,g'}})(t,x,v)=\begin{cases}(\varphi^{g'}|_{N_{g',g}}^{-1}\circ \varphi^g|_{N_{g,g'}})(t,x,v), & t\leq 0,\\
(R_{g'}\circ\varphi^{g'}|_{R_{g'}(N_{g',g})}^{-1}\circ \varphi^g|_{R_{g}(N_{g,g'})}\circ R_{g})(t,x,v), & t>0.\end{cases}
\]
Since $\varphi^g|_{N_{g,g'}}$ and $\varphi^{g'}|_{N_{g',g}}$ are diffeomorphisms onto their images, the above maps are smooth iff their post-compositions with $\varphi^g|_{N_{g,g'}}$ and $\varphi^{g'}|_{N_{g',g}}$ are smooth, respectively. In view of the definitions of $R_{g}$ and $\tilde R_{g}$, the proof is finished. \qed

\section{Construction of smooth model bundles} \label{app:bundles}

Here we provide a concrete construction of smooth models in the vector-valued setting, i.e., of smooth model bundles. The reader who is not interested in the vector-valued case may safely skip this appendix. The construction follows ideas very similar to those employed in the above construction of smooth models for non-grazing billiard flows. 

We begin by reminding the reader of some notation used in the main text: There we introduced the non-grazing billiard flow $\varphi$ acting on the phase-space $M\setminus \partial_\mathrm{g} M$, defined on the domain $D\subset \R\times(M\setminus \partial_\mathrm{g} M)$ from \eqref{eq:D} which is open by Lemma \ref{lem:D}. For the analytic treatment of this dynamical system we constructed a model manifold $\mathcal{M}$ together with a smooth surjection $\pi: M\setminus \partial_\mathrm{g} M \rightarrow \mathcal{M}$ and a  smooth model flow $\phi$ on $\mathcal{M}$, defined on the domain $\mathcal D\subset \R\times\mathcal M$, such that $\pi\circ \varphi_t = \phi_t\circ \pi$. This was necessary because $\varphi$ is non-smooth (in fact non-continuous and not even a flow in the proper sense, recall Remark \ref{rem:notflow}) due to the presence of the instantaneous boundary reflections $R: \partial M\setminus \partial_\mathrm{g} M \rightarrow \partial M\setminus \partial_\mathrm{g} M,\, (x, v) \mapsto (x, v')$.

For the remainder of this appendix we now assume a smooth $\mathbb{C}$-vector bundle $$\pi_{\widetilde{\mathcal E}}: \widetilde{\mathcal{E}} \rightarrow M\setminus \partial_\mathrm{g} M$$ of rank $r$ to be given, the fibers of which we denote by $\widetilde{\mathcal E}_{(x, v)} \defgr \pi_{\widetilde{\mathcal E}}^{-1}(\{(x, v)\})$.

Furthermore we require a first-order differential operator $\widetilde{\mathbb{X}}$ acting on smooth sections of $\widetilde{\mathcal{E}}$ and satisfying the following Leibniz rule:
\begin{equation} \label{eq:leibniz_app}
\widetilde{\mathbb{X}}\big( \widetilde{f}\cdot \widetilde{\sigma} \big) = (\mathbf{P}\widetilde{f})\cdot \widetilde{\sigma} + \widetilde{f}\cdot \widetilde{\mathbb{X}}\widetilde{\sigma},\quad \forall \widetilde{f}\in \mathrm{C}^\infty_\mathrm{Bill}(M\setminus \partial_\mathrm{g}M),\, \widetilde{\sigma}\in \mathrm{C}^\infty(M\setminus\partial_\mathrm{g} M, \widetilde{\mathcal{E}}),
\end{equation}
where $\mathbf{P}$ denotes the billiard generator defined in Section \ref{sec:CBillP}. An additional piece of data necessary for the construction of a smooth model for $\widetilde{\mathcal{E}}$ is  a bundle isomorphism
\begin{equation*}
\kappa: \widetilde{\mathcal{E}}\big|_{\partial_\mathrm{in} M} \longrightarrow \widetilde{\mathcal{E}}\big|_{\partial_\mathrm{out} M}
\end{equation*}
such that $\pi_{\widetilde{\mathcal E}}\circ \kappa = R\circ \pi_{\widetilde{\mathcal E}}\big|_{\pi_{\widetilde{\mathcal E}}^{-1}(\partial_\mathrm{in} M)}$ holds. For example, such an isomorphism exists if both  $\widetilde{\mathcal{E}}\big|_{\partial_\mathrm{in} M}$ and  $\widetilde{\mathcal{E}}\big|_{\partial_\mathrm{out} M}$ can be trivialized: Then we can simply define $\kappa$ as the composition of the first trivialization, the map $R\times \mathrm{id}_{\C^r}$, and the inverse of the second trivialization.

Before proving our main theorem we first have to describe a dynamical quantity associated with the above data, namely the \emph{billiard parallel transport}. Morally it is derived from the operator $\widetilde{\mathbb{X}}$ in the same intuitive manner as the billiard flow is derived from the geodesic flow:
\begin{lem}\label{lem:billparalltransp}
There exists a unique map $\widetilde{\alpha}: \widetilde{\mathcal{D}} \rightarrow \widetilde{\mathcal{E}}$ on the flow domain
\begin{equation*}
\widetilde{\mathcal{D}} \defgr \big\{ (t, e) \,\big|\, \exists (x, v)\in M\setminus\partial_\mathrm{g} M:\, (t, x, v)\in D ~\text{and}~ e\in \widetilde{\mathcal{E}}_{(x, v)} \big\} ~,
\end{equation*}
called the \emph{billiard parallel transport}, with the following properties:

\begin{enumerate}
\item For each $(t, x, v)\in D$ the map
\begin{equation} 
\widetilde{\alpha}_{(x, v), t}: \widetilde{\mathcal{E}}_{(x, v)} \longrightarrow \widetilde{\mathcal{E}}_{\varphi_t(x, v)},\qquad e\mapsto \widetilde{\alpha}_{(x, v), t}(e) \defgr \widetilde{\alpha}(t, e) ,
\end{equation}
is a well-defined linear isomorphism.
\item  $\widetilde{\alpha}$ is a flow up to composition with $\kappa$ on boundary fibers. More precisely, one has
\bq
\widetilde{\alpha}_{(x, v), 0}=\mathrm{id}_{\widetilde{\mathcal{E}}_{(x, v)}}\qquad \forall\;(x,v)\in M\setminus\partial_\mathrm{g}M\label{eq:alpha0tilde}
\eq
and if  $(t,x, v)\in D$ and $t'\in \R$ are such that $(t',\varphi_t(x, v))\in D$, then the following generalization of \eqref{eq:flowprop} holds:
\begin{equation} \label{eq:flowpropalpha}
\widetilde{\alpha}_{\varphi_t(x, v), t'}\circ\widetilde{\alpha}_{(x, v), t}=\begin{dcases}\widetilde{\alpha}_{(x, v), t+t'},& t + t' \neq 0\text { or }\varphi_{t+t'}(x,v)\in \mathring M\text{ or } \\
&t<0,(x, v) \in\partial_\mathrm{in} M\text { or } t>0,(x, v) \in\partial_\mathrm{out} M,\\
\kappa|_{\widetilde{\mathcal{E}}_{(x, v)}}, & t+t'=0\text { and }t>0,(x, v) \in\partial_\mathrm{in} M,\\
\kappa^{-1}|_{\widetilde{\mathcal{E}}_{(x, v)}},& t+t'=0\text { and } t<0,(x, v) \in\partial_\mathrm{out} M.
\end{dcases}
\end{equation}
\item For each $\widetilde{\sigma}\in \mathrm{C}^\infty(M\setminus\partial_\mathrm{g} M, \widetilde{\mathcal{E}})$ the map $$\varphi^{-1}(\mathring M)\to  \widetilde{\mathcal{E}},\qquad (t,x,v)\mapsto \widetilde{\alpha}_{(x, v), t}\big(\widetilde{\sigma}(x,v)\big),$$ is smooth.
\item For each $(x,v)\in \mathring{M}$ and $\widetilde{\sigma}\in \mathrm{C}^\infty(M\setminus\partial_\mathrm{g} M, \widetilde{\mathcal{E}})$  one has
\[
(\widetilde{\mathbb{X}}\widetilde{\sigma})(x,v)=\frac{d}{dt}\Big|_{t=0}\widetilde{\alpha}_{\varphi_t(x, v), -t}\widetilde{\sigma}(\varphi_t(x, v)).
\]
\end{enumerate}
\end{lem}
\begin{proof}
As a preliminary step we embed a neighborhood of $M\setminus\partial_\mathrm{g}M$ in $S\Sigma$ into a closed manifold $N$ and extend $\widetilde{\mathcal{E}}$ and $\widetilde{\mathbb{X}}$ arbitrarily to $N$ such that near $M\setminus\partial_\mathrm{g}M$ they satisfy the Leibniz rule \eqref{eq:leibniz_app} with $\mathbf{P}$ replaced by the geodesic vector field $X^g$ (recall from \eqref{eq:restrictionP} that $\mathbf{P}$ agrees with $X^g$ on billiard functions). We continue to denote these extensions by $\widetilde{\mathcal E}, \widetilde{\mathbb{X}}$ and obtain a well-defined transfer operator $\exp(-t \widetilde{\mathbb{X}})$ acting on smooth sections of $\widetilde{\mathcal E}$. 

Given $(t,x, v)\in D$ we begin by assuming that $t \geq 0$ is small enough such that $\varphi_s(x, v)$, $s\in [0, t]$, intersects $\partial M\setminus \partial_\mathrm{g} M$ only at its endpoint $\varphi_t(x, v)$, if at all. Then two cases must be distinguished:
\begin{enumerate}
	\item $(x, v)\in \mathring{M}$: Given $e\in \widetilde{\mathcal{E}}_{(x, v)}$ choose $\widetilde{\sigma}\in \mathrm{C}^\infty(N,\widetilde{\mathcal E})$ with $\widetilde{\sigma}(x, v) = e$ and supported in $\mathring{M}$. Then we define
	\begin{equation*}
	\widetilde{\alpha}_{(x, v), t}(e) \defgr \big(\hspace*{-0.15em}\exp(-t \widetilde{\mathbb{X}}) \widetilde{\sigma} \big)(\varphi_t(x, v)) ,
	\end{equation*}
	which is independent of the choice of $\widetilde{\sigma}$ by the Leibniz rule \eqref{eq:leibniz_app}, which in turn applies by the support property of $\widetilde{\sigma}$ and where we use that any smooth function supported in $\mathring{M}$ is a billiard function.
	\item $(x, v)\in \partial M\setminus \partial_\mathrm{g}M$: If $(x, v)\in \partial_\mathrm{out} M$, given $e\in \widetilde{\mathcal{E}}_{(x, v)}$,  choose $\widetilde{\sigma}\in \mathrm{C}^\infty(N,\widetilde{\mathcal E})$ with $\widetilde{\sigma}(x, v) = e$ by multiplying some local frame with cutoffs that restrict to billiard functions on $M\setminus \partial_\mathrm{g}M$. Again we define
	\begin{equation*}
	\widetilde{\alpha}_{(x, v), t}(e) \defgr \big(\hspace*{-0.15em}\exp(-t \widetilde{\mathbb{X}}) \widetilde{\sigma} \big)(\varphi_t(x, v)) ,
	\end{equation*}
	independently of the choice of $\widetilde{\sigma}$.
	If instead $(x, v)\in \partial_\mathrm{in} M$ we proceed in the same way but with $\kappa(e)$ instead of $e$.
\end{enumerate}
This construction can analogously be transferred to sufficiently small $t < 0$. 

Now, without these smallness assumptions on $t$, we use that by definition of $D$ and $\varphi$ there exists a unique finite sequence $t_0, t_1, ..., t_N$ with $t=t_0+\cdots+t_N$ such that $\varphi_t(x, v)$ can be written in terms of the geodesic flow $\varphi^g$ as $\varphi_t(x, v) = \varphi_{t_N}^g\circ R\circ \cdots\circ R\circ \varphi_{t_0}^g(x, v)$. We then define the billiard parallel transport of $e\in \widetilde{\mathcal{E}}_{(x, v)}$ as
\begin{equation*}
\widetilde{\alpha}_{(x, v), t}(e) \defgr \widetilde{\alpha}_{(x, v), t_N}\circ \cdots \circ \widetilde{\alpha}_{(x, v), t_0}(e) .
\end{equation*}
The claimed Properties (1), (2) of $\widetilde{\alpha}$ are now satisfied by construction and Properties (3), (4) follow from the properties of $\exp(-t \widetilde{\mathbb{X}})$ since $\varphi$ is smooth on $\varphi^{-1}(\mathring M)$. Finally, Property (1), Eq.\ \eqref{eq:alpha0tilde}, and Property (4) determine $\widetilde{\alpha}$ uniquely on $\varphi^{-1}(\mathring M)\times \widetilde{\mathcal{E}}$ and \eqref{eq:flowpropalpha} then implies that the full  map $\widetilde{\alpha}$ is unique because its values at points outside $\varphi^{-1}(\mathring M)\times \widetilde{\mathcal{E}}$ are determined by $\kappa$ and values at points inside $\varphi^{-1}(\mathring M)\times \widetilde{\mathcal{E}}$.
\end{proof}
\begin{defn}We call the operator
\begin{align*}
\mathcal{L}:D\times \mathrm{C}^\infty(M\setminus\partial_\mathrm{g} M, \widetilde{\mathcal{E}})&\longrightarrow \Gamma(M\setminus\partial_\mathrm{g} M, \widetilde{\mathcal{E}})\\
((t, x, v),\widetilde{\sigma})&\longmapsto \widetilde{\alpha}_{\varphi_{-t}(x, v), t}(\widetilde{\sigma}(\varphi_{-t}(x, v))) \defgl (\mathcal{L}_t\widetilde{\sigma})(x,v)
\end{align*}
the \emph{billiard transfer operator} associated with $\widetilde{\mathbb{X}}$. Here $\Gamma(M\setminus\partial_\mathrm{g} M, \widetilde{\mathcal{E}})$ denotes the set of arbitrary  sections of $\widetilde{\mathcal{E}}$ without any continuity or smoothness assumption.
\end{defn}
This terminology is of course motivated by Lemma \ref{lem:billparalltransp}, which implies that for each  $\widetilde{\sigma}\in\mathrm{C}^\infty(M\setminus\partial_\mathrm{g} M, \widetilde{\mathcal{E}})$ the map  
$$\varphi^{-1}(\mathring M)\cap (\R\times \mathring M) \to \widetilde{\mathcal{E}}|_{\mathring M},\qquad (t,x,v)\mapsto (\mathcal{L}_t\widetilde{\sigma})(x,v),$$
is smooth and one has
\bq
(\widetilde{\mathbb{X}}\widetilde{\sigma})(x,v)=-\frac{d}{dt}\Big|_{t=0}(\mathcal{L}_t\widetilde{\sigma})(x,v)\quad \forall\; (x,v)\in \mathring{M}.\label{eq:XLprop}
\eq
We can now introduce the vector-valued equivalent of the billiard functions:
\begin{defn}The set of \emph{smooth billiard sections} of $\widetilde{\mathcal E}$ is
\[
\mathrm{C}^\infty_\mathrm{Bill}(M\setminus\partial_\mathrm{g} M, \widetilde{\mathcal E}) \defgr \left\{ \widetilde{\sigma}\in \mathrm{C}^\infty(M\setminus \partial_\mathrm{g} M) \,\big|\, \big( (t,x,v)\mapsto(\mathcal{L}_t\widetilde{\sigma})(x,v)\big)\in \mathrm{C}^\infty(D, \widetilde{\mathcal E}) \right\}.
\]
\end{defn}
An elementary property of the smooth billiard sections is that they are stable with respect to multiplication by smooth billiard functions -- in other words, $\mathrm{C}^\infty_\mathrm{Bill}(M\setminus\partial_\mathrm{g} M, \widetilde{\mathcal E})$ is a $\mathrm{C}^\infty_\mathrm{Bill}(M\setminus\partial_\mathrm{g} M)$-module.  The main siginificance of the smooth billiard sections is that they are preserved by the operator $\widetilde{\mathbb{X}}$ and the latter acts on them by differentiation of the billiard transfer operator:
\[
\widetilde{\mathbb{X}}: \mathrm{C}^\infty_\mathrm{Bill}(M\setminus\partial_\mathrm{g} M, \widetilde{\mathcal E})\longrightarrow\mathrm{C}^\infty_\mathrm{Bill}(M\setminus\partial_\mathrm{g} M, \widetilde{\mathcal E}),\qquad (\widetilde{\mathbb{X}}\widetilde{\sigma})(x,v)=-\frac{d}{dt}\Big|_{t=0}(\mathcal{L}_t\widetilde{\sigma})(x,v).
\]
This follows from \eqref{eq:XLprop}. It provides the vector-valued generalization of the formula \eqref{eq:P}. 

With this data as our point of departure we can now prove our main theorem in the vector-valued situation:
\begin{theorem}[Existence of smooth model bundles] \label{thm:model_bundle}
There exists a smooth vector bundle $\pi_\mathcal{E}: \mathcal{E} \rightarrow \mathcal{M}$ and a smooth surjection $\Pi:\widetilde{\mathcal E} \to \mathcal{E}$ such that the diagram
\bq
\begin{tikzcd}
\widetilde{\mathcal E} \arrow{r}{\Pi}\arrow{d}{\pi_{\widetilde{\mathcal E}}}& \mathcal{E} \arrow{d}{\pi_{\mathcal E}}\\%
M\setminus \partial_\mathrm{g}M  \arrow{r}{\pi} & \mathcal{M}
\end{tikzcd}\label{eq:piPi}
\eq
commutes, as well as a linear isomorphism
\begin{equation} \label{eq:def_billiard_sections}
\Sigma_\mathcal{E}: \mathrm{C}^\infty_\mathrm{Bill}(M\setminus\partial_\mathrm{g} M, \widetilde{\mathcal E}) \longrightarrow \mathrm{C}^\infty(\mathcal{M}, \mathcal{E}) 
\end{equation}
that is uniquely characterized by the relation
\[
\Sigma_\mathcal{E}(\widetilde{\sigma})\circ \pi=\Pi\circ \widetilde{\sigma}\qquad \forall\; \widetilde{\sigma}\in \mathrm{C}^\infty_\mathrm{Bill}(M\setminus\partial_\mathrm{g} M, \widetilde{\mathcal E}).
\]
Furthermore, introducing the first order differential operator 
$$
\mathbb{X}:=\Sigma_\mathcal{E} \circ \widetilde{\mathbb{X}} \circ \Sigma_\mathcal{E}^{-1}:\mathrm{C}^\infty(\mathcal{M}, \mathcal{E})\to \mathrm{C}^\infty(\mathcal{M}, \mathcal{E})
$$
and for  $(t,p)\in \mathcal D$ the parallel transport 
\begin{equation} \label{eq:def_parallel_transport}
\begin{split}
\alpha_{p, t}: \mathcal{E}_p &\longrightarrow \mathcal{E}_{\phi_t(p)},\\
e &\longmapsto \big(\hspace*{-0.15em}\exp(-t \mathbb{X}) \sigma_e\big)(\phi_t(p)),
\end{split}
\end{equation}
where $\exp(-t \mathbb{X})$ is the transfer operator of $\mathbb X$ and $\sigma_e$ denotes any smooth section with $\sigma_e(p) = e$, then the trace of $\alpha_{p, t}$ on a periodic trajectory (i.e., when $\phi_t(p)=p$) coincides with the trace of the billiard parallel transport from Lemma \ref{lem:billparalltransp} on the corresponding periodic trajectory of the non-grazing billiard flow $\varphi$. 
\end{theorem}
\begin{proof}
Our proof is constructive and uses the abstract vector bundle construction lemma well established in the differential geometry literature, see e.g. \cite[Lemma~10.6]{Lee.2012}. We start by specifying the total space of our new bundle:
\begin{equation}
\mathcal{E} \defgr \bigsqcup_{p\in \mathcal{M}\setminus\mathcal{G}} \widetilde{\mathcal{E}}_{\pi^{-1}(p)} \sqcup \bigsqcup_{(x, v)\in \partial_\mathrm{in} M} \widetilde{\mathcal{E}}_{(x, v)}\oplus \widetilde{\mathcal{E}}_{(x, v')} \slash \sim_\kappa ,
\end{equation}
where $\sim_\kappa$ means that we quotient out the linear subspace of  $\widetilde{\mathcal{E}}_{(x, v)}\oplus \widetilde{\mathcal{E}}_{(x, v')}$  defined by the elements of the form $(e,-\kappa(e))$, which is possible by the relation $\pi_{\widetilde{\mathcal E}}\circ \kappa = R\circ \pi_{\widetilde{\mathcal E}}\big|_{\pi_{\widetilde{\mathcal E}}^{-1}(\partial_\mathrm{in} M)}$. Note that over each point $p=[x,v]\in \mathcal G$ the fiber $\mathcal E_p=\widetilde{\mathcal{E}}_{(x, v)}\oplus \widetilde{\mathcal{E}}_{(x, v')} \slash \sim_\kappa$ is canonically isomorphic to $\widetilde{\mathcal{E}}_{(x, v)}$ as well as $\widetilde{\mathcal{E}}_{(x, v')}$ via the maps $e\mapsto [e,0]$ and $e\mapsto [0,e]$, respectively.  Next we need to specify the trivializations of $\mathcal{E}$: In a neighborhood of any point of $\mathcal{M}\setminus \mathcal{G}$ we make the obvious choice and take trivializations of $\widetilde{\mathcal{E}}$ composed with (a suitable restriction of) $\pi\times \mathrm{id}_{\mathbb{C}^r}$. Around a point $p = [x, v] \in\mathcal{G}$, $(x,v)\in  \partial_\mathrm{in} M$, we define \emph{flow-trivializations} using the billiard parallel transport map: Take any trivialization $t_\mathrm{in}$ of $\widetilde{\mathcal{E}}\big|_{\partial_\mathrm{in} M}$ on an open set $U_\mathrm{in}\subset \partial_\mathrm{in} M$ around $(x, v)$ and put
\begin{equation*}
\begin{split}
\mathbf{t}: \pi^{-1}_\mathcal{E}\left( \Phi(N_\mathrm{in}\cap\mathbb{R}\times U_\mathrm{in} ) \right) &\longrightarrow \Phi(N_\mathrm{in}\cap\mathbb{R}\times U_\mathrm{in}) \times \mathbb{C}^r, \\
\mathcal{E}_{\Phi(t, [y, w])} \ni e &\longmapsto \left( (\phi_t\circ \pi)\times \mathrm{id}_{\mathbb{C}^r} \right) \circ t_\mathrm{in} \left(\widetilde{\alpha}_{\varphi_t(y, w), -t}(e) \right),
\end{split}
\end{equation*}
where $\Phi: N_\mathrm{in} \rightarrow \mathcal{M}$ denotes the flow chart from Section \ref{sec:smoothstructure}, $(y, w)$ is the unique lift of $[y,w]$ to $\partial_\mathrm{in} M$, and if $t=0$ we identified $[e,0]$ and $e$. This defines a trivialization $\mathbf{t}$ around $p$. Now, $\mathbf{t}$ transitions smoothly with any trivialization around points in  $\mathcal{M}\setminus \mathcal{G}$ thanks to Lemma \ref{lem:billparalltransp} (1), and furthermore $\mathbf{t}$ transitions smoothly with any trivialization $\mathbf{t}'$ built analogously but from another trivialization $t'_\mathrm{in}$ since $t'_\mathrm{in}\circ t_\mathrm{in}^{-1}$ is smooth and the smooth structure of $\mathcal M$ near $\mathcal G$ has been defined using flow charts. 

The desired surjection $\Pi:\widetilde{\mathcal{E}}\to\mathcal{E}$ is given by
\[
\Pi(e) \defgr \begin{cases}e,\qquad & e\in \bigsqcup_{p\in \mathcal{M}\setminus\mathcal{G}} \widetilde{\mathcal{E}}_{\pi^{-1}(p)},\\
[e,0], &e\in \bigsqcup_{(x, v)\in \partial_\mathrm{in} M} \widetilde{\mathcal{E}}_{(x, v)},\\
[0,e], &e\in \bigsqcup_{(x, v)\in \partial_\mathrm{out} M} \widetilde{\mathcal{E}}_{(x, v)}.
\end{cases}
\]
This map is continuous by definition of the topology on $\mathcal{E}$ and it is clearly smooth on $\bigsqcup_{p\in\mathcal{M}\setminus\mathcal{G}} \widetilde{\mathcal{E}}_{\pi^{-1}(p)}$. To check that $\Pi$ is also smooth near $\bigsqcup_{(x, v)\in \partial_\mathrm{in} M} \widetilde{\mathcal{E}}_{(x, v)}$, we compose it with a trivializaton $\mathbf{t}$ as above.  For $e\in \bigsqcup_{(x, v)\in \partial_\mathrm{in} M} \widetilde{\mathcal{E}}_{(x, v)}$ the point $e$ has been identified with $[e,0]$ in the definition of $\mathbf{t}$, so that by definition of the smooth structure on $\mathcal E$ the composition $\mathbf t\circ \Pi$ is smooth near $e$. On the other hand, for $e\in \bigsqcup_{(x, v)\in \partial_\mathrm{out} M} \widetilde{\mathcal{E}}_{(x, v)}$ the point $[0,e]=[\kappa^{-1}(e),0]$ appearing in $\mathbf{t}$ at $t=0$ is identified with $\kappa^{-1}(e)$. Using the relation $\pi_{\widetilde{\mathcal E}}\circ \kappa = R\circ \pi_{\widetilde{\mathcal E}}\big|_{\pi_{\widetilde{\mathcal E}}^{-1}(\partial_\mathrm{in} M)}$ and the definition of $\widetilde{\alpha}$ from the proof of Lemma \ref{lem:billparalltransp}, we obtain
\[
\widetilde{\alpha}_{\varphi_t(y, w), -t}\big(\kappa^{-1}(e)\big)=\widetilde{\alpha}_{\varphi_t(y, w), -t}(e)
\]
and hence $\mathbf t\circ \Pi$ is again smooth near $e$.

Next we define for a smooth billiard section $\widetilde{\sigma}$ of $\widetilde{\mathcal E}$ the section $\Sigma_\mathcal{E}(\widetilde{\sigma}) \defgr \sigma$ as 
\begin{equation*}
\sigma([x, v]) \defgr
\begin{cases}
\widetilde{\sigma}(x, v), &[x, v]\notin\mathcal{G},\\
[\widetilde{\sigma}(x, v),0]_\kappa, &[x, v]\in \mathcal{G},\; (x, v)\in \partial_\mathrm{in} M
\end{cases}
\end{equation*}
where $[\cdot,\cdot]_\kappa$ denotes the equivalence class with respect to $\sim_\kappa$. The map $\sigma$ is well-defined (i.e., independently of the choice of the lift $(x,v)\in\partial_\mathrm{in} M$ of $[x,v]$) because $\widetilde{\sigma}$ being a billiard section implies that $\widetilde{\sigma}(x,v')=\kappa (\widetilde{\sigma}(x,v))$ for  $(x, v)\in \partial_\mathrm{in} M$, and hence
$[\widetilde{\sigma}(x, v),0]_\kappa=[0,\widetilde{\sigma}(x, v')]_\kappa$. 
Clearly, $\sigma$ is a section and smooth at any point of $\mathcal{M}\setminus \mathcal{G}$. To test smoothness at any point of the gluing region $\mathcal{G}$ we have to compose with flow charts and flow trivializations to obtain a coordinate expression. Doing so yields
\begin{equation*}
\begin{split}
\mathbb{R}\times U_\mathrm{in} \cap N_\mathrm{in} &\longrightarrow \left( \mathbb{R}\times U_\mathrm{in} \cap N_\mathrm{in} \right) \times \mathbb{C}^r\\
(t, x, v) &\longmapsto t_\mathrm{in}\left( \widetilde{\alpha}_{\varphi_t(x, v), -t}\left( \widetilde{\sigma}(\varphi_t(x, v)) \right) \right) .
\end{split}
\end{equation*}
To prove the mapping properties claimed in \eqref{eq:def_billiard_sections} it therefore only remains to show that $\Sigma_\mathcal{E}$ is bijective on the given domains, but this follows easily from the observation that
\begin{equation*}
\Sigma_\mathcal{E}^{-1}(\sigma)(x, v) =
\begin{cases}
\sigma([x, v]), &(x, v)\in\mathring{M},\\
e, &(x, v)\in\partial_\mathrm{in} M,\; \sigma([x, v])=[e,0]_\kappa,\\
e', &(x, v)\in\partial_\mathrm{out} M,\; \sigma([x, v])=[0,e']_\kappa
\end{cases}
\end{equation*}
is indeed the inverse of $\Sigma_\mathcal{E}$ and has its image in $\mathrm{C}^\infty_\mathrm{Bill}(M\setminus\partial_\mathrm{g} M, \widetilde{\mathcal E})$ due to a similar coordinate calculation.

The definition of $\alpha$ is independent of the chosen section $\sigma_e$ by virtue of the Leibniz rule which trivially follows for $\mathbb{X}$ from the Leibniz rule for $\widetilde{\mathbb{X}}$. The claimed equality of traces follows immediately because the model bundle $\mathcal{E}$ was constructed in such a way that the boundary map $\kappa$ present in the definition of $\widetilde{\alpha}$ acts trivially on it.
\end{proof}

\begin{remark}
We refrain from stating and proving a vector-valued uniqueness result analogous to \Cref{prop:uniqueness}, as well as a vector-valued version of the resolvent study as in Corollary \ref{cor:Pmeromorphic} for the sake of brevity.
\end{remark}



\bibliographystyle{amsalpha}
\bibliography{bibo}

\bigskip


\end{document}